\documentclass[12pt]{article}
\pdfoutput=1

\usepackage{amsmath, amsthm, amssymb}
\usepackage{enumerate}
\usepackage{ifpdf}
\ifpdf
\usepackage[pdftex]{graphicx}
\else
\usepackage[dvips]{graphicx}
\fi
\usepackage{tikz}
 	 \usetikzlibrary{arrows,backgrounds}
\usepackage[all]{xy}

\usepackage{yfonts}

\input xy
\xyoption{all}

\usepackage[pdftex,plainpages=false,hypertexnames=false,pdfpagelabels]{hyperref}
\newcommand{\arXiv}[1]{\href{http://arxiv.org/abs/#1}{\tt arXiv:\nolinkurl{#1}}}
\newcommand{\arxiv}[1]{\href{http://arxiv.org/abs/#1}{\tt arXiv:\nolinkurl{#1}}}
\newcommand{\doi}[1]{\href{http://dx.doi.org/#1}{{\tt DOI:#1}}}

\newcommand{\googlebooks}[1]{(preview at \href{http://books.google.com/books?id=#1}{google books})}

\usepackage{bigstrut}

\usepackage{xcolor}
\definecolor{dark-red}{rgb}{0.7,0.25,0.25}
\definecolor{dark-blue}{rgb}{0.15,0.15,0.55}
\definecolor{medium-blue}{rgb}{0,0,.8}
\definecolor{DarkGreen}{RGB}{0,150,0}
\hypersetup{
   colorlinks, linkcolor={purple},
   citecolor={medium-blue}, urlcolor={medium-blue}
}

\makeatletter
\newcommand{\hashdef}[2]{\@namedef{#1}{#2}}
\newcommand{\hashlookup}[1]{\@nameuse{#1}}
\makeatother

\IfFileExists{../../graphs/lookup.tex}%
	{\newcommand{\pathtographs}{../../graphs/}}%
	{\newcommand{\pathtographs}{diagrams/graphs/}}

\hashdef{bwd1v1v1p1v1x0p1x0p0x1p0x1v0x1x1x0duals1v1v1x3x2x4}{CC3B89144EA94D74}%
\hashdef{bwd1v1v1p1v1x1v1v1duals1v1v1v1}{A207DA908F977FE3}%
\hashdef{bwd1v1v1v1p1p1v1x0x0p0x1x0p0x0x1v1x0x0p0x1x0p0x0x1duals1v1v1x2x3v1x2x3}{772136FAF425B059}%
\hashdef{bwd1v1v1v1p1p1v1x0x0p0x1x0p0x0x1v1x0x0p0x1x0p0x0x1duals1v1v1x2x3v1x3x2}{B5425B05363851E6}%
\hashdef{bwd1v1v1v1p1p1v1x0x0p0x1x0p0x0x1v1x0x0p0x1x0p0x0x1duals1v1v1x3x2v1x2x3}{AD7F91E884816713}%
\hashdef{bwd1v1v1v1p1p1v1x0x0p0x1x0p0x1x0v1x0x0duals1v1v1x2x3v1}{DB2C121504201C9D}%
\hashdef{bwd1v1v1v1p1v1x0p0x1v1x0p0x1duals1v1v1x2v2x1}{25781EDEA61320E6}%
\hashdef{bwd1v1v1v1p1v1x0p1x0duals1v1v1x2}{E26689FBED470E90}%
\hashdef{bwd1v1v1v1v1v1v1v1p1v1x0p0x1v1x0p0x1duals1v1v1v1v1x2v2x1}{C53E32050C862F37}%
\hashdef{bwd1v1v1v1v1v1v1v1p1v1x0p1x0duals1v1v1v1v1x2}{CA1952BED6C8EB18}%
\hashdef{gbg1v1v1p1v1x1v1v1}{4F5AE97C48AAE720}%
\hashdef{gbg1v1v1v1p1p1v1x0x0p0x1x0p0x0x1v1x0x0p0x1x0p0x0x1}{BAF17D013FA99673}%
\hashdef{gbg1v1v1v1p1v1x0p0x1p0x1}{093BD179AB17B87C}%
\hashdef{gbg1v1v1v1p1v1x0p0x1v0x1p0x1}{953E54EAD5E7BF7F}%
\hashdef{gbg1v1v1v1p1v1x0p0x1v1x0p0x1}{3C61F07AACCEE2CC}%
\hashdef{gbg1v1v1v1p1v1x0p0x1v1x1}{37EE9489DAFBC516}%
\hashdef{gbg1v1v1v1p1v1x0p0x1v2x0}{B2ABE12C270A8A32}%
\hashdef{gbg1v1v1v1p1v1x1}{BDA0F6E63ABA6AAF}%
\hashdef{gbg1v1v1v1p1v2x0}{D635C81E0061B4CB}%

\newcommand{\bigraph}[1]{{\hspace{-3pt}\begin{array}{c}%
  \raisebox{-2.5pt}{\includegraphics[height=6mm]{\pathtographs \hashlookup{#1}}}%
\end{array}\hspace{-3pt}}}

\usepackage{longtable}

\theoremstyle{plain}
\newtheorem{thm}{Theorem}[section]
\newtheorem*{thm*}{Theorem}

\newtheorem{lem}[thm]{Lemma}
\newtheorem{fact}[thm]{Fact}

\newtheorem{prop}[thm]{Proposition}

\newtheorem{step}{Step}

\newtheorem*{quest*}{Question}

\theoremstyle{definition}
\newtheorem{defn}[thm]{Definition}
\newtheorem{assumption}[thm]{Assumption}
\newtheorem{nota}[thm]{Notation}

\newtheorem{ex}[thm]{Example}
\newtheorem{rem}[thm]{Remark}
\newtheorem*{rem*}{Remark}

\DeclareMathOperator{\coeff}{coeff}

\DeclareMathOperator{\id}{id}

\DeclareMathOperator{\spann}{span}

\DeclareMathOperator{\Tr}{Tr}

\newcommand{\D}{\displaystyle}
\newcommand{\comment}[1]{}
\newcommand{\hs}{\hspace{.07in}}

\newcommand{\be}{\begin{enumerate}[(1)]}
\newcommand{\ee}{\end{enumerate}}

\newcommand{\Z}{\mathbb{Z}}
\newcommand{\Q}{\mathbb{Q}}

\newcommand{\C}{\mathbb{C}}

\newcommand{\set}[2]{\left\{#1 \middle| #2\right\}}

\newcommand{\cA}{\mathcal{A}}

\newcommand{\cC}{\mathcal{C}}

\newcommand{\cF}{\mathcal{F}}
\newcommand{\cG}{\mathcal{G}}

\newcommand{\cL}{\mathcal{L}}

\newcommand{\cP}{\mathcal{P}}
\newcommand{\cQ}{\mathcal{Q}}
\newcommand{\cR}{\mathcal{R}}

\newcommand{\cT}{\mathcal{T}}

\newcommand{\cX}{\mathcal{X}}
\newcommand{\cY}{\mathcal{Y}}
\newcommand{\gA}{{\textgoth{A}}}
\newcommand{\fB}{{\mathfrak{B}}}

\newcommand{\noshow}[1]{}

\newcommand{\MR}[1]{}

\newcommand{\traincirc}[1]{\underset{#1}{\circ}}
\newcommand{\jw}[1]{f^{(#1)}}
\newcommand{\norm}[1]{\left\|#1\right\|}

\tikzstyle{shaded}=[fill=red!10!blue!20!gray!30!white]
\tikzstyle{unshaded}=[fill=white]
\tikzstyle{empty box}=[circle, draw, thick, fill=white, opaque, inner sep=2mm]
\tikzstyle{annular}=[scale=.7, inner sep=1mm, baseline]
\tikzstyle{rectangular}=[scale=.75, inner sep=1mm, baseline=-.1cm]

\newcommand{\rt}{\cR\cT}

\newcommand{\rtac}{\cR\cT\hspace{-.05cm}\cA\cC}

\newcommand{\TwoTrain}[3]{
\begin{tikzpicture}[baseline = -.2cm]
	\clip (-.5,-.8)--(-.5,.7)--(2.4,.7)--(2.4,-.8);
	\draw (0,0)--(1.6,0)--(1.6,-1)--(0,-1)--(0,0);
	\node at (.8,.15) {{\scriptsize{$n-#1$}}};
	\node at (.4,-.6) {{\scriptsize{$n+#1$}}};
	\node at (2,-.6) {{\scriptsize{$n+#1$}}};
	\draw[thick, unshaded] (0,0) circle (.4);
	\node at (0,0) {$#2$};
	\draw[thick, unshaded] (1.6,0) circle (.4);
	\node at (1.6,0) {$#3$};
	\node at (1.6,.55) {$\star$};
	\node at (0,.55) {$\star$};
\end{tikzpicture}}

\newcommand{\BoxTwoTrain}[2]{
\begin{tikzpicture}[baseline = -.7cm]
	\clip (-.5,-1.8)--(-.5,.7)--(2.4,.7)--(2.4,-1.8);
	\draw (0,0)--(1.6,0)--(1.6,-2)--(0,-2)--(0,0);
	\node at (.8,.15) {{\scriptsize{$n-2$}}};
	\node at (.4,-.6) {{\scriptsize{$n+2$}}};
	\node at (2,-.6) {{\scriptsize{$n+2$}}};
	\draw[thick, unshaded] (0,0) circle (.4);
	\node at (0,0) {$#1$};
	\draw[thick, unshaded] (1.6,0) circle (.4);
	\node at (1.6,0) {$#2$};
	\node at (1.6,.55) {$\star$};
	\node at (0,.55) {$\star$};
	\draw[thick, unshaded] (-.4,-1.5)--(-.4,-.8)--(2,-.8)--(2,-1.5)--(-.4,-1.5);
	\node at (.8,-1.15) {$f^{(2 n+4)}$};
	\node at (-.55,-1.15) {$\star$};
\end{tikzpicture}}

\newcommand{\ThreeTrain}[3]{
\begin{tikzpicture}[baseline = -.2cm]
	\clip (-.5,-.8)--(-.5,.7)--(4,.7)--(4,-.8);
	\draw (0,0)--(3.2,0);
	\draw (0,0)--(0,-.8);
	\draw (1.6,0)--(1.6,-.8);
	\draw (3.2,0)--(3.2,-.8);
	\node at (.8,.15) {{\scriptsize{$n-1$}}};
	\node at (2.4,.15) {{\scriptsize{$n-1$}}};
	\node at (.4,-.6) {{\scriptsize{$n+1$}}};
	\node at (1.8,-.6) {{\scriptsize{$2$}}};
	\node at (3.6,-.6) {{\scriptsize{$n+1$}}};
	\draw[thick, unshaded] (0,0) circle (.4);
	\node at (0,0) {$#1$};
	\node at (0,.55) {$\star$};
	\draw[thick, unshaded] (1.6,0) circle (.4);
	\node at (1.6,0) {$#2$};
	\node at (1.6,.55) {$\star$};
	\draw[thick, unshaded] (3.2,0) circle (.4);
	\node at (3.2,0) {$#3$};
	\node at (3.2,.55) {$\star$};
\end{tikzpicture}}

\newcommand{\BoxThreeTrain}[3]{
\begin{tikzpicture}[baseline = -.7cm]
	\clip (-.5,-1.8)--(-.5,.7)--(4,.7)--(4,-1.8);
	\draw (0,0)--(3.2,0);
	\draw (0,0)--(0,-1.8);
	\draw (1.6,0)--(1.6,-1.8);
	\draw (3.2,0)--(3.2,-1.8);
	\node at (.8,.15) {{\scriptsize{$n-1$}}};
	\node at (2.4,.15) {{\scriptsize{$n-1$}}};
	\node at (.4,-.6) {{\scriptsize{$n+1$}}};
	\node at (1.8,-.6) {{\scriptsize{$2$}}};
	\node at (3.6,-.6) {{\scriptsize{$n+1$}}};
	\draw[thick, unshaded] (0,0) circle (.4);
	\node at (0,0) {$#1$};
	\node at (0,.55) {$\star$};
	\draw[thick, unshaded] (1.6,0) circle (.4);
	\node at (1.6,0) {$#2$};
	\node at (1.6,.55) {$\star$};
	\draw[thick, unshaded] (3.2,0) circle (.4);
	\node at (3.2,0) {$#3$};
	\node at (3.2,.55) {$\star$};
	\draw[thick, unshaded] (-.4,-1.5)--(-.4,-.8)--(3.6,-.8)--(3.6,-1.5)--(-.4,-1.5);
	\node at (1.6,-1.15) {$f^{(2 n+4)}$};
	\node at (-.55,-1.15) {$\star$};
\end{tikzpicture}}

\newcommand{\JellyfishSquared}[2]{
\begin{tikzpicture}[baseline = -.3cm]
	\filldraw[shaded] (-.9,-.8)--(-.9,0) arc (180:0:.9cm)--(.9,-.8);
	\filldraw[unshaded] (-.7,-.8)--(-.7,0) arc (180:0:.7cm)--(.7,-.8);
	\draw (0,0)--(0,-.8);
	\node at (.2,-.6) {{\scriptsize{$2 #1$}}};
	\draw[thick, unshaded] (0,0) circle (.4);
	\node at (0,0) {$#2$};
	\node at (-.25,-.5) {$\star$};
\end{tikzpicture}
}

\newcommand{\BoxJellyfishSquared}[2]{
\begin{tikzpicture}[baseline = -.7cm]
	\filldraw[shaded] (-.9,-1.8)--(-.9,0) arc (180:0:.9cm)--(.9,-1.8);
	\filldraw[unshaded] (-.7,-1.8)--(-.7,0) arc (180:0:.7cm)--(.7,-1.8);
	\draw (0,0)--(0,-1.8);
	\node at (.2,-.6) {{\scriptsize{$2 #1$}}};
	\draw[thick, unshaded] (0,0) circle (.4);
	\node at (0,0) {$#2$};
	\node at (-.25,-.5) {$\star$};
	\draw[thick, unshaded] (-1,-1.5)--(-1,-.8)--(1,-.8)--(1,-1.5)--(-1,-1.5);	
	\node at (0,-1.15) {$f^{(2 #1+4)}$};
	\node at (-1.15,-1.15) {$\star$};
\end{tikzpicture}
}

\usetikzlibrary{calc}
\newcommand{\nbox}[5]{
	\draw[thick, #1] ($#2+(-.4,-.4)+(-#3,0)$) -- ($#2+(-.4,.4)+(-#3,0)$) -- ($#2+(.4,.4)+(#4,0)$) -- ($#2+(.4,-.4)+(#4,0)$) -- ($#2+(-.4,-.4)+(-#3,0)$); 
	\coordinate (a) at ($#2+(-#3,0)$);
	\coordinate (b) at ($#2+(#4,0)$);
	\node at ($1/2*(a)+1/2*(b)$) {$#5$};
}

\newcommand{\TL}{\cT\hspace{-.08cm}\cL}

\begin{document}
\title{
Calculating two-strand jellyfish relations
}
\author{David Penneys and Emily Peters}
\date{\today}
\maketitle
\begin{abstract}
We construct subfactors where one of the principal graphs is a spoke graph using an algorithm which computes two-strand jellyfish relations.
One of the subfactors we construct is a $3^{\Z/4}$ subfactor known to Izumi, which has not previously appeared in the literature.
To do so, we provide a systematic treatment of the space of second annular consequences, which is analogous to Jones' treatment of the space of first annular consequences in his quadratic tangles article. 

This article is the natural followup to two recent articles on spoke subfactor planar algebras and the jellyfish algorithm.
Work of Bigelow-Penneys explains the connection between spoke subfactor planar algebras and the jellyfish algorithm, and work of Morrison-Penneys automates the construction of subfactors where both principal graphs are spoke graphs using one-strand jellyfish.
\end{abstract}
\section{Introduction}

Every subfactor planar algebra embeds in the graph planar algebra (first defined in \cite{MR1865703}) of its principal graph \cite{MR2812459,tvc}. 
Thus one can construct a subfactor planar algebra by finding candidate generators in the appropriate graph planar algebra, and then showing the planar algebra they generate is a subfactor planar algebra with the correct principal graph. 

By now, these methods have been used to construct a large handful of examples, some new and some well known, including the $E_6$ and $E_8$ subfactors \cite{MR1929335}, group-subgroup subfactors \cite{MR2511128}, the Haagerup subfactor \cite{MR2679382}, the extended Haagerup subfactor \cite{MR2979509}, the Izumi-Xu 2221 subfactor \cite{1102.2052}, and certain spoke subfactors, e.g., 4442 \cite{1208.3637}.
These techniques have also been used to prove uniqueness results \cite{MR2979509,1102.2052} and obstructions to possible principal graphs \cite{MR2679382,1302.5148}.

Early applications of the embedding theorem to construct or obstruct subfactors were mostly ad hoc. Recent work of Bigelow-Penneys \cite{1208.1564} has explained why some of the previous constructions work and how they fit into the same family of examples.
If the principal graph of a subfactor is a spoke graph with simple arms connected to one central vertex, the planar algebra can be constructed using two-strand jellyfish relations. If both graphs are spokes, one can use one-strand relations, which are easier to compute. Recent work of Morrison-Penneys \cite{1208.3637} found an algorithm to compute these one-strand relations, provided one has the generators in the graph planar algebra. 

This article is the natural followup to \cite{1208.1564,1208.3637}. 
The main result of this article is an algorithm to find two-strand jellyfish relations for a subfactor planar algebra for which one of the principal graphs is a spoke graph, provided we are given the generators in a graph planar algebra.
This article's most interesting application of this algorithm is the following theorem.

\begin{thm*}
There exists a $3^{\Z/4}$ subfactor with principal graphs
$$
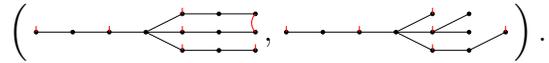
\left(
\bigraph{bwd1v1v1v1p1p1v1x0x0p0x1x0p0x0x1v1x0x0p0x1x0p0x0x1duals1v1v1x2x3v1x3x2},
\bigraph{bwd1v1v1v1p1p1v1x0x0p0x1x0p0x1x0v1x0x0duals1v1v1x2x3v1}
\right).
$$
\end{thm*}

\begin{rem*}
Interestingly, we construct this $3^{\Z/4}$ subfactor planar algebra in a graph planar algebra not coming from one of these principal graphs (see Appendix \ref{generators:3333ZMod4}).
(Of course, by the embedding theorem, it is also a planar subalgebra of the $3^{\Z/4}$ graph planar algebra.)
\end{rem*}

The motivation for this article is to systematically study a conjectural infinite family of $3^G$ spoke subfactors for certain finite abelian groups $G$, first studied by Izumi \cite{MR1832764}, and later by Evans-Gannon \cite{MR2837122}. 
A $3^G$ subfactor has principal graph consisting of $|G|$ spokes of length 3, and the dual data is determined by the inverse law of the group $G$.
In fact, Izumi has an unpublished construction of a $3^{\Z/4}$ subfactor using Cuntz algebras analogous to his treatment for odd order $G$ in \cite{MR1832764}. 
Moreover, he shows such a subfactor is unique, which our approach does not attempt to show.
In theory, all $3^G$ subfactors can be constructed using two-strand jellyfish \cite{1208.1564}. 
The major hurdle is finding the generators in the graph planar algebra. Once given the generators, the machinery of this article produces the two-strand relations.

The foundation for this article, which underlies the previously discussed constructions and obstructions, is Jones' annular tangles point of view. 
Each unitary planar algebra can be orthogonally decomposed into irreducible annular Temperley-Lieb modules. 
In doing so, we seem to lose a lot of information, namely the action of higher genus tangles. 
However, we find ourselves in the simpler situation of analyzing irreducible annular Temperley-Lieb modules, which have been completely classified \cite{MR1659204,MR1929335}. 
Such a module is generated by a single low-weight rotational eigenvector.
This perspective is particularly useful for small index subfactors, which can only have a few small low-weight vectors.

This article is also a natural followup to Jones' exploration of quadratic tangles \cite{MR2972458}.
There are necessarily strong quadratic relations among the few smallest low-weight generators of a subfactor planar algebra of small modulus.
In \cite{MR2972458}, Jones studies the space of first annular consequences of the low-weight vectors to find explicit formulas for these relations.
We provide an analogous systematic treatment of the space of second annular consequences of a set of low-weight generators of a subfactor planar algebra. 
Studying this space was fruitful in Peters' planar algebra construction of the Haagerup subfactor \cite{MR2679382}.

While we do not provide any explicit quadratic relations in this article, such relations surely exist and will provide strong obstructions for subfactor planar algebras.
In particular, such an understanding of the space of second annular consequences may be helpful in deriving obstructions in the case where only one of the principal graphs is a spoke graph \cite{1208.1564}.

\subsection{Outline}

In Section \ref{sec:Background}, we give the necessary background for this article, including 
conventions for graph planar algebras, 
tetrahedral structure constants,
the jellyfish algorithm,
and reduced trains.
In Section \ref{sec:SecondAC}, we give a basis for the second annular consequences of a low-weight element when $\delta>2$.

In Section \ref{sec:Formulas}, we analyze the space of reduced trains, in particular their projections to Temperley-Lieb and annular consequences.
We then calculate many pairwise inner products of such trains and their projections. 
In Section \ref{sec:ComputingJellyfishRelations}, we provide the algorithm for computing two-strand jellyfish relations given generators in our graph planar algebras. 

In Section \ref{sec:relations}, we provide the results of applying the algorithm from Section \ref{sec:ComputingJellyfishRelations} to construct the $3^{\Z/3}$, $3^{\Z/2\times \Z/2}$, and $3^{\Z/4}$ subfactor planar algebras.
We compute the principal graphs of our examples in Section \ref{sec:PrincipalGraphs}.

Finally, we have two appendices where we record the data necessary for the above computations. The generators are specified in Appendix  \ref{sec:Generators} via their values on collapsed loops, and we give the moments and tetrahedral structure constants for our generators in Appendix \ref{sec:MomentsAndTetrahedrals}.

\subsection{Acknowledgements}
The authors would like to thank Masaki Izumi and Scott Morrison for many helpful conversations.
Some of this work was completed when David Penneys visited Scott Morrison at The Australian National University and Masaki Izumi at Kyoto University. He would like to thank both of them for supporting those trips and for their hospitality.

David Penneys was supported in part by the Natural Sciences and Engineering Research Council of Canada.
Emily Peters was supported by an NSF RTG grant at Northwestern University, DMS-0636646.  David Penneys and Emily Peters were both supported by DOD-DARPA grant HR0011-12-1-0009.

\subsection{The {\tt FusionAtlas} (adapted from \cite{1208.3637})} 
This article relies on some substantial calculations. In particular, our efforts to find the generators in the various graph planar algebras made use of a variety of techniques, some ad-hoc, some approximate, and some computationally expensive. This article essentially does not address that work. Instead, we merely present the discovered generators and verify some relatively easy facts about them. In particular, the proofs presented in this article rely on the computer in a much weaker sense. We need to calculate certain numbers of the form $\Tr(PQRS)$, where $P,Q,R,S$ are rather large martrices, and the computer does this for us. We also entered all the formulas derived in this article into {\tt Mathematica}, and had the computer automatically evaluate the various quantities which appear in our derivation of jellyfish relations. As a reader may be interested in seeing these programs, we include a brief instruction on finding and running these programs.

The {\tt arXiv} sources of this article contain in the {\tt code} subdirectory a number of files, including:
\begin{itemize}
\item {\tt Generators.nb}, which reconstructs the generators from our terse descriptions of them in Appendix \ref{sec:Generators}.
\item {\tt TwoStrandJellyfish.nb}, which calculates the requisite moments and tetrahedral structure constants of these generators, and performs the linear algebra necessary to derive the jellyfish relations.
\item {\tt GenerateLaTeX.nb}, which typesets each subsection of Section \ref{sec:relations} for each planar algebra, and many mathematical expressions in Appendices \ref{sec:Generators} and \ref{sec:MomentsAndTetrahedrals}.
\end{itemize}

The {\tt Mathematica} notebook {\tt Generators.nb} can be run by itself. The final cells of that notebook write the full generators to the disk; this must be done before running {\tt TwoStrandJellyfish.nb}. The {\tt TwoStrandJellyfish.nb} notebook relies on the {\tt FusionAtlas}, a substantial body of code the authors have developed along with Scott Morrison, Noah Snyder, and James Tener to perform calculations with subfactors and fusion categories. To obtain a local copy, you first need to ensure that you have {\tt Mercurial}, the distributed version control system, installed on your machine. With that, the command 
\begin{quote}
{\tt hg clone https://bitbucket.org/fusionatlas/fusionatlas}
\end{quote}
will create a local directory called {\tt fusionatlas} containing the latest version. In the {\tt TwoStrandJellyfish.nb} notebook, you will then need to adjust the paths appearing in the first input cell to ensure that your local copy is included. After that, running the entire notebook reproduces all the calculations described below.

We invite any interested readers to contact us with questions or queries about the use of these notebooks or the {\tt FusionAtlas} package.

\section{Background}\label{sec:Background}

We now give the background material for the calculations that occur in the later sections. We refer the reader to \cite{MR2679382,MR2979509,MR2972458,JonesPANotes} for the definition of a (subfactor) planar algebra.

\begin{nota}
When we draw planar diagrams, we often suppress the external boundary disk. 
In this case, the external boundary is assumed to be a large rectangle whose distinguished interval contains the upper left corner. 
We draw one string with a number next to it instead of drawing that number of parallel strings. 
We shade the diagrams as much as possible, but if the parity is unknown, we  often cannot know how to shade them. 
Finally, projections are usually drawn as rectangles with the same number of strands emanating from the top and bottom, while other elements may be drawn as circles.
\end{nota}

Some parts of this introduction are adapted from \cite{1208.3637,1208.1564}.

\subsection{Working in graph planar algebras}\label{sec:GPAEmbedding}

Graph planar algebras, defined in \cite{MR1865703}, have proven to be a fruitful place to work because of the following theorem. 
Strictly speaking, our constructions do not rely on this theorem.
However, it motivates our search for generators in the appropriate graph planar algebra.

\begin{thm}[\cite{MR2812459,tvc}]
Every subfactor planar algebra embeds in the graph planar algebra of its principal graph.
\end{thm}

In \cite[Section 2.2]{1208.3637}, it was observed that many of Jones' quadratic tangles \cite{MR2972458} formulas for subfactor planar algebras hold for certain collections of elements in unitary, spherical, shaded $*$-planar algebras which are not necessarily evaluable (see Theorem \ref{thm:QTMoreGeneral}). 
The main example of such a planar algebra is the graph planar algebra of a finite bipartite graph.
We give the necessary definitions and discuss our conventions for working in such planar algebras in this subsection.

\begin{defn}\label{defn:UnitarySpherical}
A shaded planar $*$-algebra is \underline{evaluable} if $\dim(\cP_{n,\pm})<\infty$ for all $n\geq 0$, and $\cP_{0,\pm}\cong \C$ via the map that sends the empty diagram to 1.

Suppose $\cP_\bullet$ is a shaded planar $*$-algebra which is not necessarily evaluable. 
We call $\cP_\bullet$ \underline{unitary} if for all $n\geq 0$, the $\cP_{0,\pm}$-valued sesquilinear form on $\cP_{n,\pm}$ given by $\langle x,y\rangle=\Tr(y^*x)$ is positive definite (in the operator-valued sense).

We call such a planar algebra \underline{spherical} if any time we have a closed diagram in $\cP_\bullet$ which equals a scalar multiple of the empty diagram, then performing isotopy on a sphere still gives us the same scalar multiple of the appropriate empty diagram.
\end{defn}

\begin{rem}
The above is only one possible definition of unitarity for a planar $*$-algebra.
One might also want to require the existence of a faithful state on $\cP_{0,\pm}$ which induces a $C^*$-algebra structure on the algebras $\cP_{n,\pm}$ in the usual GNS way.
However, the above frugal definition is sufficient for our purposes, since the following theorem holds.
\end{rem}

\begin{thm}
Suppose $\cP_\bullet$ is a spherical, unitary, shaded planar $*$-algebra which is not necessarily evaluable.
If $\cQ_\bullet\subset \cP_\bullet$ is an evaluable planar $*$-subalgebra, then $\cQ_\bullet$ is a subfactor planar algebra.
\end{thm}
\begin{proof}
Since $\cQ_\bullet$ is evaluable, sphericality of $\cQ_\bullet$ follows from sphericality of $\cP_\bullet$.
Now, the sesquilinear form $\langle x,y\rangle = \Tr(y^*x)$ on $\cQ_{n,\pm}$ is operator-valued positive definite. 
Since $\cQ_\bullet$ is evaluable, by identifying the appropriate empty diagram with $1\in\C$, we get a positive definite inner product.
\end{proof}

\begin{nota}
Recall that the Fourier transform $\cF$ is given by
$$
\cF=
\begin{tikzpicture}[baseline = 0cm]
	\clip (0,0) circle (1cm);
	\draw[] (0,0) circle (1cm);
	\draw[] (0,0)--(-10:1.1cm) arc (-10:30:1.1cm) --(0,0); 
	\draw[] (0,0)--(70:1.1cm) arc (70:110:1.1cm) --(0,0); 
	\draw[] (0,0)--(150:1.1cm) arc (150:190:1.1cm) --(0,0); 
	\draw[] (0,0)--(230:1.1cm) arc (230:310:1.1cm) --(0,0); 
	\draw[thick, unshaded] (0,0) circle (.4cm);
	\node at (90:.52) {$\star$};
	\node at (130:.86) {$\star$};
	\node at (0,-.6) {$\cdots$};
	\node at (0,0) {};
	\draw[ultra thick] (0,0) circle (1cm);
\end{tikzpicture}\,.
$$
For a rotational eigenvector $S\in \cP_{n,\pm}$ corresponding to an eigenvalue $\omega_S=\sigma_S^2$, we define another rotational eigenvector $\check{S}\in \cP_{n,\mp}$ by $\check{S}=\sigma_S^{-1}\cF(S)$. Note that $\cF(\check{S})=\sigma_S S$, so $\check{\check{S}}=S$.
\end{nota}

\begin{defn}
Suppose $\cP_\bullet$ is a unitary, spherical, shaded planar $*$-algebra with modulus $\delta>2$ which is not necessarily evaluable.
A finite  set $\fB\subset \cP_{n,+}$ is called a \underline{set of minimal generators for $\cQ_\bullet$} if the elements of $\fB$ generate the planar $*$-subalgebra $\cQ_\bullet\subset\cP_\bullet$ and are linearly independent, self-adjoint, low-weight eigenvectors for the rotation, i.e, for all $S\in\fB$,
\begin{itemize}
\item $S=S^*$,
\item $S$ is uncappable, and
\item $\rho(S)=\omega_S S$ for some $n$-th root of unity $\omega_S$.
\end{itemize}
In the sequel, when we refer to a set of minimal generators without mentioning $\cQ_\bullet$, assume that $\cQ_\bullet$ is the planar $*$-subalgebra generated by $\fB$.

Given a set of minimal generators $\fB$, we get a set of dual minimal generators $\check{\fB}=\set{\check{S}}{S\in\fB}$. 

We say a set of minimal generators $\fB$ has \underline{scalar moments} if $\Tr(R),\Tr(RS),\Tr(RST)$ and $\Tr(\check{R}),\Tr(\check{R}\check{S}),\Tr(\check{R}\check{S}\check{T})$ are scalar multiples of the empty diagram in $\cP_{0,+}$ and $\cP_{0,-}$ respectively for each $R,S,T\in\fB$. 

If a set of minimal generators $\fB$ has scalar moments, we say $\fB$ is 
\begin{itemize}
\item
\underline{orthogonal} if $\langle S,T\rangle =\Tr(ST)=0$ if $S\neq T$ for all $S,T\in \fB$, and
\item
\underline{orthonormal} if $\fB$ is orthogonal and $\Tr(S^2)=\langle S,S\rangle = 1$ for all $S\in\fB$.
\end{itemize}
\end{defn}

The point of working with sets of minimal generators is the following theorem, first observed in \cite{1208.3637}.

\begin{thm}[{\cite[Theorem 2.5]{1208.3637}}]\label{thm:QTMoreGeneral}
All the formulas of Section 4 of  \cite{MR2972458} hold in any unitary, spherical, shaded planar $*$-algebra with modulus $\delta>2$ for any orthonormal set of minimal generators $\fB$ with scalar moments.
\end{thm}

\begin{assumption}\label{assume:Generators}
For the rest of the article, unless otherwise specified, we assume $\cP_\bullet$ is a unitary, spherical, shaded $*$-planar algebra with modulus $\delta>2$ which is not necessarily evaluable, and $\fB\subset \cP_{n,+}$ is an orthogonal set of minimal generators with scalar moments.
\end{assumption}

Since we do not assume our generators in $\fB$ are orthonormal, our formulas will differ slightly in appearance than those of \cite{MR2972458} and \cite{1208.3637}.

\begin{rem}
For diagram evaluation, it is useful to have our standard equations for our set of minimal generators in one place. For $S\in\fB$,
\begin{align*}
S&=S^*
&
\cF^2&=\rho
&
\rho(S)&=\omega_S S
&
\cF(S)&=\sigma_S\check{S}
\\
\check{S}&=\check{S}^*
&
\sigma_S^2&=\omega_S
&
\rho(\check{S})&=\omega_S \check{S}
&
\cF(\check{S})&=\sigma_S S.
\end{align*}
When moving $\star$ on the distinguished interval of a generator, the resulting diagram is multiplied by some exponent of $\sigma_S$:
\begin{itemize}
\item
if you shift $\star$ counterclockwise by one strand, multiply by $\sigma_S$ and switch $\check{}\,$:
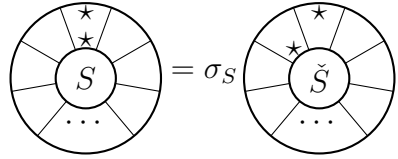
$$
\begin{tikzpicture}[baseline = 0cm]
	\draw[thick, unshaded] (0,0) circle (1cm);
	\draw (0,0)--(-10:1cm);
	\draw (30:1cm) --(0,0); 
	\draw (0,0)--(70:1cm);
	\draw (110:1cm) --(0,0); 
	\draw (0,0)--(150:1cm);
	\draw (190:1cm) --(0,0); 
	\draw (0,0)--(230:1cm);
	\draw (310:1cm) --(0,0); 
	\draw[thick, unshaded] (0,0) circle (.4cm);
	\node at (90:.52) {$\star$};
	\node at (0,-.6) {$\cdots$};
	\node at (0,0) {$S$};
	\node at (90:.88) {$\star$};
\end{tikzpicture}
=
\sigma_S\,
\begin{tikzpicture}[baseline = 0cm]
	\draw[thick, unshaded] (0,0) circle (1cm);
	\draw (0,0)--(-10:1cm);
	\draw (30:1cm) --(0,0); 
	\draw (0,0)--(70:1cm);
	\draw (110:1cm) --(0,0); 
	\draw (0,0)--(150:1cm);
	\draw (190:1cm) --(0,0); 
	\draw (0,0)--(230:1cm);
	\draw (310:1cm) --(0,0); 
	\draw[thick, unshaded] (0,0) circle (.4cm);
	\node at (130:.52) {$\star$};
	\node at (0,-.6) {$\cdots$};
	\node at (0,0) {$\check{S}$};
	\node at (90:.88) {$\star$};
\end{tikzpicture}
$$
\item
if you shift $\star$ clockwise by one strand, multiply by $\sigma_S^{-1}$ and switch $\check{}\,$:
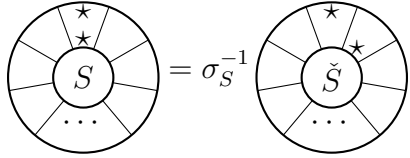
$$
\begin{tikzpicture}[baseline = 0cm]
	\draw[thick, unshaded] (0,0) circle (1cm);
	\draw (0,0)--(-10:1cm);
	\draw (30:1cm) --(0,0); 
	\draw (0,0)--(70:1cm);
	\draw (110:1cm) --(0,0); 
	\draw (0,0)--(150:1cm);
	\draw (190:1cm) --(0,0); 
	\draw (0,0)--(230:1cm);
	\draw (310:1cm) --(0,0); 
	\draw[thick, unshaded] (0,0) circle (.4cm);
	\node at (90:.52) {$\star$};
	\node at (0,-.6) {$\cdots$};
	\node at (0,0) {$S$};
	\node at (90:.88) {$\star$};
\end{tikzpicture}
=
\sigma_S^{-1}\,
\begin{tikzpicture}[baseline = 0cm]
	\draw[thick, unshaded] (0,0) circle (1cm);
	\draw (0,0)--(-10:1cm);
	\draw (30:1cm) --(0,0); 
	\draw (0,0)--(70:1cm);
	\draw (110:1cm) --(0,0); 
	\draw (0,0)--(150:1cm);
	\draw (190:1cm) --(0,0); 
	\draw (0,0)--(230:1cm);
	\draw (310:1cm) --(0,0); 
	\draw[thick, unshaded] (0,0) circle (.4cm);
	\node at (50:.52) {$\star$};
	\node at (0,-.6) {$\cdots$};
	\node at (0,0) {$\check{S}$};
	\node at (90:.88) {$\star$};
\end{tikzpicture}
$$
\end{itemize}
\end{rem}

Using notation from \cite{MR2972458}, for $P,Q,R\in\fB$, we set
\begin{align*}
a_R^{PQ}&=\Tr(PQR) \text{ and}\\
b_R^{PQ}&=\Tr(\check{P}\check{Q}\check{R}).
\end{align*}

\begin{rem}\label{rem:SpanAlgebras}
Once we have determined our set of minimal generators $\fB$ has scalar moments, the next thing to do is to verify that the complex spans of $\fB\cup\{\jw{n}\}$ and $\check{\fB}\cup\{\jw{n}\}$ form algebras under the usual multiplication. If this is the case, for $P,Q\in\fB$, we necessarily have
\begin{equation}\label{eqn:ClosedUnderMultiplication}
PQ = 
\frac{\Tr(PQ)}{[n+1]} \jw{n}+\sum_{R\in\fB} \frac{a^{PQ}_R}{\|R\|^2}R.
\end{equation}
Immediately, we get that all higher moments of $\fB,\check{\fB}$ are scalars, as well as certain tetrahedral structure constants (see Remark \ref{rem:ReduceTerahedrals} and Example \ref{ex:ReduceTetrahedral}). For example, we have that
\begin{equation}\label{eqn:QuarticMoment}
\Tr(PQRS)=\frac{\Tr(PQ)\Tr(RS)}{[n+1]}+\sum_{T\in \fB} \frac{a^{PQ}_T}{\|T\|^2} a^{RS}_T. 
\end{equation}
for $P,Q,R,S\in\fB$.
\end{rem}

\begin{assumption}\label{assume:SpanAlgebras}
We now assume the complex spans of $\fB\cup\{\jw{n}\}$ and $\check{\fB}\cup\{\jw{n}\}$ form algebras under the usual multiplication.
\end{assumption}

\begin{rem}
The assumptions of this subsection are significant.
A randomly chosen subset of a graph planar algebra will not satisfy Assumption \ref{assume:Generators}.
Given an orthogonal set of minimal generators $\fB$ with scalar moments, it is still possible it will not satisfy Assumption \ref{assume:SpanAlgebras}.
For example, if we start with a $\fB$ satisfying Assumptions \ref{assume:Generators} and \ref{assume:SpanAlgebras} and we discard one element, the resulting set together with $\jw{n}$ may not span an algebra.
\end{rem}

\subsection{Tetrahedral structure constants}\label{sec:TetrahedralAndLopsided}

We will also need the tetrahedral structure constants defined in \cite[Section 3.3]{quadraticPreprint}.

\begin{defn} For $P,Q,R,S\in\fB$, we define
$$
\Delta_{a,b}(P,Q,R\mid S)=
\begin{tikzpicture}[baseline = -.2cm, scale=1.5]
	\node (tp) [circle, unshaded, thick, draw] at (0,1) {$Q$};
	\node (md) [circle, unshaded, thick, draw] at (0,-.1) {$S^{\vee b}$};
	\node (bl) [circle, unshaded, thick, draw]  at (-1,-.8) {$P$};
	\node (br) [circle, unshaded, thick, draw]  at (1,-.8) {$R$};
	\draw (bl)-- node [below] {{\scriptsize{$b$}}} (br);
	\draw (bl)-- node [above] {{\scriptsize{$c$}}} (md);
	\draw (bl)-- node [left] {{\scriptsize{$a$}}} (tp);
	\draw (br)-- node [above] {{\scriptsize{$a$}}} (md);
	\draw (br)-- node [right] {{\scriptsize{$c$}}} (tp);
	\draw (md)-- node [left] {{\scriptsize{$b$}}} (tp);
	\node [above=-3pt] at (tp.north) {$\star$};
	\node [below=-3pt] at (md.south) {$\star$};
	\node [left=-3pt] at (bl.west) {$\star$};
	\node [right=-3pt] at (br.east) {$\star$};
\end{tikzpicture}
$$
where $c=2n-a-b$, and 
$$
S^{\vee b}=
\begin{cases}
S & \text{if $b$ is even}\\
\check{S} & \text{if $b$ is odd.}\\
\end{cases}
$$
Note that the $\Delta_{a,b}(P,Q,R\mid S)$ for $P,Q,R,S\in\fB$ determine all the tetrahedral structure constants by  \cite[Section 3.3]{quadraticPreprint}.
\end{defn}

\begin{rem}\label{rem:ReduceTerahedrals}
For this article, we only need the following tetrahedral structure constants:
\begin{itemize}
\item
$\Delta_{n-1,2}(P,Q,R\mid S)$
\item
$\Delta_{n,1}(P,Q,R\mid S)$
\item
$\Delta_{n-1,1}(P,Q,R \mid S)
=
\overline{\Delta_{n,1}(R,Q,P\mid S)}.
$
\end{itemize}
By Assumption \ref{assume:SpanAlgebras}, 
we can express the second and third tetrahedral structure constants above in terms of the moments and chiralities of $\fB,\check{\fB}$, since one of $a,b,c\geq n$. We do this computation in Example \ref{ex:ReduceTetrahedral}. 
Thus for convenience, we will just write $\Delta(P,Q,R\mid S)$ instead of $\Delta_{n-1,2}(P,Q,R\mid S)$, and we will only write  subscripts  $a,b$ if $a\neq n-1$ or $b\neq 2$.
For each of our planar algebras in this article, we give the tetrahedral structure constants $\Delta(P,Q,R\mid S)$ in Appendix \ref{sec:MomentsAndTetrahedrals}.
\end{rem}

\begin{ex}\label{ex:ReduceTetrahedral}
\begin{align*}
\Delta_{n,1}&(P,Q,R\mid S)\\
&=
\begin{tikzpicture}[baseline = 0cm, scale=1.5]
	\node (tp) [circle, unshaded, thick, draw] at (0,1) {$Q$};
	\node (md) [circle, unshaded, thick, draw] at (0,-.1) {$\check{S}$};
	\node (bl) [circle, unshaded, thick, draw]  at (-1,-.8) {$P$};
	\node (br) [circle, unshaded, thick, draw]  at (1,-.8) {$R$};
	\draw (bl)-- node [below] {{\scriptsize{$1$}}} (br);
	\draw (bl)-- node [right] {{\scriptsize{$n-1$}}} (md);
	\draw (bl)-- node [left] {{\scriptsize{$n$}}} (tp);
	\draw (br)-- node [above] {{\scriptsize{$n$}}} (md);
	\draw (br)-- node [right] {{\scriptsize{$n-1$}}} (tp);
	\draw (md)-- node [left] {{\scriptsize{$1$}}} (tp);
	\node [above=-3pt] at (tp.north) {$\star$};
	\node [right=1pt] at (md.south) {$\star$};
	\node [left=-3pt] at (bl.west) {$\star$};
	\node [right=-3pt] at (br.east) {$\star$};
\end{tikzpicture}
\\
&=
\sigma_R^{-1}
\begin{tikzpicture}[baseline = -.1cm]
	\node (lt) [circle, unshaded, thick, draw] at (-2,0) {$PQ$};
	\node (md) [circle, unshaded, thick, draw] at (0,0) {$\check{S}$};
	\node (rt) [circle, unshaded, thick, draw]  at (2,0) {$\check{R}$};
	\draw (lt)-- node [below] {{\scriptsize{$n$}}} (md);
	\draw (md)-- node [below] {{\scriptsize{$n$}}} (rt);
	\draw (rt) to [in= -45,out=-135] node [above] {{\scriptsize{$1$}}} (lt);
	\draw (rt) to [in= 45,out=135] node [below] {{\scriptsize{$n-1$}}} (lt);
	\node [left=-3pt] at (lt.west) {$\star$};
	\node [below=-3pt] at (md.south) {$\star$};
	\node [below=1pt] at (rt.west) {$\star$};
\end{tikzpicture}
\\
&=
\sigma_R^{-1}\frac{\Tr(PQ)\Tr(\check{R}\check{S})}{[n+1]}\underset{\in\jw{n}}{\coeff}\left(
\begin{tikzpicture}[baseline = -.1cm]
	\draw[thick] (-.8,-.6)--(-.8,.6)--(.6,.6)--(.6,-.6)--(-.8,-.6);
	\draw (0,.6) arc (-180:0:.2cm);
	\draw (.4,-.6)--(-.4,.6);
	\draw (-.4,-.6) arc (180:0:.2cm);
	\node at (-.42,0) {{\scriptsize{$n-2$}}};
\end{tikzpicture}
\,
\right)
+
\sum_{T\in\fB}
\sigma_R^{-1}
\frac{a_T^{PQ}}{\|T\|^2}
\begin{tikzpicture}[baseline = -.1cm]
	\node (lt) [circle, unshaded, thick, draw] at (-1,0) {$T$};
	\node (rt) [circle, unshaded, thick, draw]  at (1,0) {$\check{R}\check{S}$};
	\draw (lt)-- node [above] {{\scriptsize{$n$}}} (rt);
	\draw (rt) to [in= -45,out=-135] node [above] {{\scriptsize{$1$}}} (lt);
	\draw (rt) to [in= 45,out=135] node [above] {{\scriptsize{$n-1$}}} (lt);
	\node [left=-3pt] at (lt.west) {$\star$};
	\node [below=1pt] at (rt.west) {$\star$};
\end{tikzpicture}
\\
&=
(-1)^{n-1}\sigma_R^{-1}\frac{\Tr(PQ)\Tr(\check{R}\check{S})}{[n][n+1]}
+
\sum_{T\in\fB}
\sigma_T\sigma_R^{-1}\frac{a_T^{PQ}b_T^{RS}}{\|T\|^2}.
\end{align*}
In the above calculation, we used Equation \eqref{eqn:ClosedUnderMultiplication} for the third equality, and we used the formula in Section \ref{sec:SecondDualBasis} for the coefficient in the Jones-Wenzl idempotent appearing in the fourth line.
By symmetry, we get
\begin{align*}
\Delta_{n-1,1}(P,Q,R \mid S)
&=
\overline{\Delta_{n,1}(R,Q,P\mid S)}
\\
&=
(-1)^{n-1}\sigma_P\frac{\Tr(QR)\Tr(\check{P}\check{S})}{[n][n+1]}
+
\sum_{T\in\fB}
\sigma_T^{-1}\sigma_P\frac{a_T^{QR}b_T^{SP}}{\|T\|^2}.
\end{align*}
\end{ex}

\begin{lem}
We have the following symmetries:
\begin{align*}
\Delta(P,Q,R\mid S)
&=
\overline{\Delta(R,Q,P\mid S)}\\
&=
\omega_P\omega_R^{-1} \Delta(R, S, P\mid Q)\\
&=
\omega_P\omega_R^{-1} \overline{\Delta(P, S, R\mid Q)}
\\
&=
\sigma_P^{1-n}\sigma_Q^{n-1}\sigma_R^{n-1}\sigma_S^{1-n} \Delta(Q^{\vee (n-1)},P^{\vee (n-1)}, S^{\vee (n-1)}\mid R^{\vee (n-1)})\\
&=
\sigma_P^{1-n}\sigma_Q^{n-1}\sigma_R^{n-1}\sigma_S^{1-n} \overline{\Delta(S^{\vee (n-1)},P^{\vee (n-1)}, Q^{\vee (n-1)}\mid R^{\vee (n-1)})}\\
&=
\sigma_P^{1-n}\sigma_Q^{n+1}\sigma_R^{n-1}\sigma_S^{-1-n} \Delta(S^{\vee (n-1)},R^{\vee (n-1)}, Q^{\vee (n-1)}\mid P^{\vee (n-1)})\\
&=
\sigma_P^{1-n}\sigma_Q^{n+1}\sigma_R^{n-1}\sigma_S^{-1-n} \overline{\Delta(Q^{\vee (n-1)},R^{\vee (n-1)}, S^{\vee (n-1)}\mid P^{\vee (n-1)})}\\
\end{align*}
\end{lem}
\begin{proof}
Immediate from drawing diagrams using unitarity and sphericality of $\cP_\bullet$.
\end{proof}
%
%
%


\begin{rem}
As in \cite{1205.2742,1208.3637}, when doing calculations in the graph planar algebra, we work with the lopsided convention rather than the spherical convention (see \cite{1205.2742}).
The lopsided convention treats shaded and unshaded contractible loops differently, which has the advantage that there are fewer square roots, so arithmetic is easier.

The translation map $\natural \colon \cP_\bullet^{spherical}\to \cP_\bullet^{lopsided}$ between the conventions from \cite{1205.2742} is not a planar algebra map, but it commutes with the action of the planar operad up to a scalar. 
We determine the scalar by first drawing the tangle in a standard rectangular form where each box has the same number of strings attached to the top and bottom.
We then get one factor of $\delta^{\pm 1}$ for each critical point which is shaded above, and the power of $\delta$ corresponds to the sign of the critical point:
$$
\begin{tikzpicture}[baseline = -.15cm]
	\fill[unshaded] (-.5,-.5)--(-.5,.25)--(.5,.25)--(.5,-.5);
	\draw[shaded] (.25,.25)--(.25,0) arc (0:-180:.25cm) --(-.25,.25);
\end{tikzpicture}
\longleftrightarrow \delta
\hspace{1.5cm}
\begin{tikzpicture}[baseline = -.15cm]
	\clip (-.5,-.25)--(-.5,.5)--(.5,.5)--(.5,-.25);
	\fill[shaded] (-.5,-.5)--(-.5,.5)--(.5,.5)--(.5,-.5);
	\draw[unshaded] (-.25,-.5)--(-.25,0) arc (180:0:.25cm) --(.25,-.5);
\end{tikzpicture}
\longleftrightarrow \delta^{-1}.
$$
\end{rem}

Correction factors for the lopsided convention for the Fourier transform and the trace were worked out in \cite[Examples 2.6 and 2.7]{1208.3637}, and we work out another correction factor in the following example.

\begin{ex}
We will work out the correction factors for the lopsided convention when calculating $\Delta(P,Q,R\mid S)$.
We have
$$
\Delta(P,Q,R\mid S)
=
\Tr\left(
\begin{tikzpicture}[baseline = -.1cm]
	\fill[shaded] (1,1.8)--(1,2.6) -- (1.5,2.6) -- (1,1.8);
	\draw (1,1.8)--(.5,2.6);
	\draw (1,1.8)--(1,2.6);
	\draw (1,1.8)--(1.5,2.6);
	\draw (0,-1.8)--(.5,-2.6);
	\fill[shaded] (0,.2) -- (.2,.2) .. controls ++(-60:.7cm) and ++(270:.8cm) .. (1,1.8)  .. controls ++(315:1.5cm) and ++(60:.5cm) .. (0,-.5) -- (0,.6);
	\draw[] (0,.2) -- (.2,.2) .. controls ++(-60:.7cm) and ++(270:.8cm) .. (1,1.8);
	\draw (1,1.8)  .. controls ++(315:1.5cm) and ++(60:.5cm) .. (0,-.5);
	\fill[shaded] (0,-.8) .. controls ++(-60:.4cm) and ++(90:1.3cm) .. (1.5,-2.6) -- (1,-2.6) .. controls ++(90:.8cm) and ++(60:.5cm) .. (.2,-1.4) -- (0,-1.4) -- (0,-.6);
	\draw (0,-.8) .. controls ++(-60:.4cm) and ++(90:1.3cm) .. (1.5,-2.6);
	\draw (1,-2.6) .. controls ++(90:.8cm) and ++(60:.5cm) .. (.2,-1.4);
	\draw (0,-2.6)--(0,2.6);
	\draw (0,.8)--(1,1.6);
	\nbox{unshaded}{(1,1.8)}{0}{0}{S}
	\nbox{unshaded}{(0,.6)}{0}{0}{R}
	\nbox{unshaded}{(0,-.6)}{0}{0}{Q}
	\nbox{unshaded}{(0,-1.8)}{0}{0}{P}
	\node at  (-.2,1.8) {{\scriptsize{$2$}}};
	\node at  (-.4,0) {{\scriptsize{$n-1$}}};
	\node at  (-.4,-1.2) {{\scriptsize{$n-1$}}};
	\node at  (-.2,-2.4) {{\scriptsize{$2$}}};
	\node at  (.5,-2.65) {{\scriptsize{$n-2$}}};
	\node at  (.5,2.65) {{\scriptsize{$n-2$}}};
	\node at  (.45,1.2) {{\scriptsize{$n-2$}}};
\end{tikzpicture}
\right),
$$
where the shading assumes $n$ is even.
The above diagram contributes a factor of $\delta^{-1}$, and the trace tangle contributes no factors of $\delta$. When $n$ is odd, the above diagram contribues a factor of $\delta$, and the trace tangle contributes a factor of $\delta$. (See \cite[Example 2.6]{1208.3637} as well.) Hence we have the formula
$$
\Delta(P,Q,R\mid S)
=\natural \Delta(P,Q,R\mid S)
=
\begin{cases}
\delta^{-1}\Delta(\natural P,\natural Q,\natural R\mid \natural S) &\text{if $n$ is even}\\
\delta^2\Delta(\natural P,\natural Q,\natural R\mid \natural S)&\text{if $n$ is odd.}
\end{cases}
$$
\end{ex}

\begin{assumption}\label{assume:Tetrahedral}
For the rest of the article, we assume that for all $P,Q,R,S\in \fB$, the tetrahedral structure constants $\Delta(P,Q,R\mid S)$ are scalar multiples of the empty diagram.
\end{assumption}

\subsection{The jellyfish algorithm and reduced trains}\label{sec:JellyfishAlgorithm}

The \underline{jellyfish algorithm} was invented in \cite{MR2979509} to construct the extended Haagerup subfactor planar algebra with principal graphs
$$
\left(\bigraph{bwd1v1v1v1v1v1v1v1p1v1x0p0x1v1x0p0x1duals1v1v1v1v1x2v2x1}, \bigraph{bwd1v1v1v1v1v1v1v1p1v1x0p1x0duals1v1v1v1v1x2}\right).
$$
One uses the jellyfish algorithm to evaluate closed diagrams on a set of minimal generators. There are two ingredients:
\begin{enumerate}
\item[(1)] The generators in $\fB\subset \cP_{n,+}$ satisfy \underline{jellyfish relations}, i.e., for each generator $S,T$,
$$
j(\check{S})=
\begin{tikzpicture}[baseline = -.3cm]
	\filldraw[shaded] (-.7,-.8)--(-.7,0) arc (180:0:.7cm)--(.7,-.8);
	\draw (0,0)--(0,-.8);
	\node at (.2,-.6) {{\scriptsize{$2n$}}};
	\draw[thick, unshaded] (0,0) circle (.4);
	\node at (0,0) {$\check{S}$};
	\node at (-.25,-.5) {$\star$};
\end{tikzpicture}
\,,\,\,\,
j^2(T)= 
\JellyfishSquared{n}{T}
$$
can be written as linear combinations of \underline{trains}, which are diagrams where any region meeting the distinguished interval of a generator meets the distinguished interval of the external disk, i.e.,
$$
\begin{tikzpicture}[baseline = -.3cm]
	\draw (-2,-1.2)--(3,-1.2);
	\draw (-1,0)--(-1,-1);
	\draw (0,0)--(0,-1);
	\draw (2,0)--(2,-1);
	\filldraw[thick, unshaded] (-1.5,-.8)--(-1.5,-1.6)--(2.5,-1.6)--(2.5,-.8)--(-1.5,-.8);
	\draw[thick, unshaded] (-1,0) circle (.4);
	\draw[thick, unshaded] (0,0) circle (.4);	
	\draw[thick, unshaded] (2,0) circle (.4);
	\node at (-1.7,-.8) {$\star$};
	\node at (-1.3,.5) {$\star$};	
	\node at (-.3,.5) {$\star$};
	\node at (1.7,.5) {$\star$};
	\node at (1,0) {$\cdots$};
	\node at (.5,-1.2) {$\cT$};
	\node at (-1,0) {$S_1$};
	\node at (0,0) {$S_2$};
	\node at (2,0) {$S_\ell$};
	\node at (-1.7,-1.4) {{\scriptsize{$k$}}};
	\node at (2.7,-1.4) {{\scriptsize{$k$}}};
	\node at (-1.2,-.6) {{\scriptsize{$2n$}}};
	\node at (-.2,-.6) {{\scriptsize{$2n$}}};
	\node at (1.8,-.6) {{\scriptsize{$2n$}}};
\end{tikzpicture}
$$
where $S_1,\dots, S_\ell\in \fB$, and $\cT$ is a single Temperley-Lieb diagram.

\item[(2)] The generators in $\fB$ are uncappable and together with the Jones-Wenzl projection $f^{(n)}$ form an algebra under the usual multiplication 
$$
ST=
\begin{tikzpicture}[baseline = .8cm]
	\draw (0,2.2)--(0,-.4);	
	\filldraw[unshaded,thick] (0,1.5) circle (.4cm);
	\node at (0,1.5) {$T$};
	\node at (-.55,1.5) {$\star$};
	\filldraw[unshaded, thick] (0,.3) circle (.4cm);
	\node at (0,.3) {$S$};
	\node at (-.55,.3) {$\star$};
	\node at (-.2,.9) {{\scriptsize{$n$}}};
	\node at (-.2,-.2) {{\scriptsize{$n$}}};
	\node at (-.2,2) {{\scriptsize{$n$}}};
\end{tikzpicture}
=
\sum_{R}
\alpha^R_{S,T}
\begin{tikzpicture}[baseline = -.1cm]
	\draw (0,.8)--(0,-.8);
	\filldraw[unshaded,thick] (0,0) circle (.4cm);
	\node at (0,0) {$R$};
	\node at (-.55,0) {$\star$};
	\node at (-.2,.6) {{\scriptsize{$n$}}};
	\node at (-.2,-.6) {{\scriptsize{$n$}}};
\end{tikzpicture}\,.
$$
(Note that the {\texttt{Mathematica}} package {\texttt{FusionAtlas}} also multiplies in this order; reading from left to right in products corresponds to reading from bottom to top in planar composites.)
\end{enumerate}
Given these two ingredients, one can evaluate any closed diagram using the following two step process.
\begin{enumerate}
\item[(1)] Pull all generators $S$ to the outside of the diagram using the jellyfish relations, possibly getting diagrams with more $S$'s.
\item[(2)] Use uncappability and the algebra property to iteratively reduce the number of generators. Any non-zero train which is a closed diagram is either a Temperley-Lieb diagram, has a capped generator, or has two generators $S,T$ connected by at least $n$ strings, giving $ST$.
\end{enumerate}

Section \ref{sec:ComputingJellyfishRelations} is devoted to our procedure for computing the jellyfish relations necessary for the first part of the jellyfish algorithm. The second part is rather easy, and amounts to verifying Equation \eqref{eqn:ClosedUnderMultiplication} (see the beginning of Section \ref{sec:relations}).

\begin{defn}\label{defn:ReducedTrain}
A $\fB$-train is called \underline{reduced} if no two generators are connected by more than $n-1$ strands, and no generator is connected to itself.
\end{defn}

\begin{ex}
In $\cP_{n+1,+}$, the set of reduced trains is given by
$$
\set{
P\traincirc{n-1}Q 
=
\TwoTrain{1}{P}{Q}
}
{P,Q\in\fB}.
$$
\end{ex}

To describe the reduced trains in $\cP_{n+2,+}$, we introduce the following notation.

\begin{defn}
Let $C_i[P\traincirc{n-1} Q]\in \cP_{n+2,+}$ for $i=1,\dots, 2n+3$ be the reduced train obtained from $P\traincirc{n-1}Q$ by putting $C_i$ underneath, where $C_i$ is the diagram given by
$$
C_i=
\begin{tikzpicture}[baseline = -.1cm]
	\draw[thick] (-.6,-.4)--(-.6,.4)--(.6,.4)--(.6,-.4)--(-.6,-.4);
	\draw (-.4,-.4)--(-.4,.4);
	\draw (.4,-.4)--(.4,.4);
	\draw (-.2,-.4) arc (180:0:.2cm);
	\node at (-.2,-.6) {{\scriptsize{$i$}}};
	\node at (-.4,.6) {{\scriptsize{$i-1$}}};
\end{tikzpicture}\,.
$$
This can be thought of as multiplying $C_i$ by $P\traincirc{n-1}Q$ for a fixed arrangement of boundary strings.
For example, we have
$$
C_1[P\traincirc{n-1}Q] 
=
\begin{tikzpicture}[baseline = -.2cm]
	\clip (-.65,-.8)--(-.65,.7)--(2.4,.7)--(2.4,-.8);
	\draw (-.6,-.8) arc (180:0:.2cm);
	\draw (0,0)--(1.6,0)--(1.6,-1)--(0,-1)--(0,0);
	\node at (.8,.15) {{\scriptsize{$n-1$}}};
	\node at (.4,-.6) {{\scriptsize{$n+1$}}};
	\node at (2,-.6) {{\scriptsize{$n+1$}}};
	\draw[thick, unshaded] (0,0) circle (.4);
	\node at (0,0) {$P$};
	\draw[thick, unshaded] (1.6,0) circle (.4);
	\node at (1.6,0) {$Q$};
	\node at (1.6,.55) {$\star$};
	\node at (0,.55) {$\star$};
\end{tikzpicture}
\text{ and }
C_{n+2}[P\traincirc{n-1}Q] 
=
\begin{tikzpicture}[baseline = -.2cm]
	\clip (-.8,-.8)--(-.8,.7)--(2.4,.7)--(2.4,-.8);
	\draw (.6,-.8) arc (180:0:.2cm);
	\draw (0,0)--(1.6,0)--(1.6,-1)--(0,-1)--(0,0);
	\node at (.8,.15) {{\scriptsize{$n-1$}}};
	\node at (-.4,-.6) {{\scriptsize{$n+1$}}};
	\node at (2,-.6) {{\scriptsize{$n+1$}}};
	\draw[thick, unshaded] (0,0) circle (.4);
	\node at (0,0) {$P$};
	\draw[thick, unshaded] (1.6,0) circle (.4);
	\node at (1.6,0) {$Q$};
	\node at (1.6,.55) {$\star$};
	\node at (0,.55) {$\star$};
\end{tikzpicture}.
$$
for $P,Q,R\in\fB$.
\end{defn}

\begin{ex}
In $\cP_{n+2,+}$, we have many more reduced trains. 
First, we have those annular consequences of the $P\traincirc{n-1} Q$'s which are still trains in $\cP_{n+2,+}$. These are exactly the $C_i[P\traincirc{n-1}Q]$ for $i=1,\dots, 2n+3$.

Now the only reduced trains which are non-zero when we put a copy of $\jw{2n+4}$ underneath are 
$$
P\underset{n-2}{\circ}  Q=
\TwoTrain{2}{P}{Q}
\,,\,
P\underset{n-1}{\circ}  Q\underset{n-1}{\circ}  R=
\ThreeTrain{P}{Q}{R}
$$
for $P,Q,R\in\fB$.
\end{ex}

\subsection{The second annular basis}\label{sec:SecondAC}

Given a nonzero low-weight rotational eigenvector $R\in \cP_{n,+}$, the space $\gA_{n+2}(R)\subset \cP_{n+2,+}$ of second annular consequences of $R$ is spanned by diagrams with two cups on the outer boundary.
We now describe a distinguished basis of $\gA_{n+2}(R)$ when $\delta>2$ along the lines of \cite{MR1929335,MR2972458}.

\begin{defn}
The element $\cup_{i,j}(R)\in \gA_{n+2}(R)$ is the annular consequence of $R$ given in the following diagrams, where each row consists of $2n+4$ elements. 
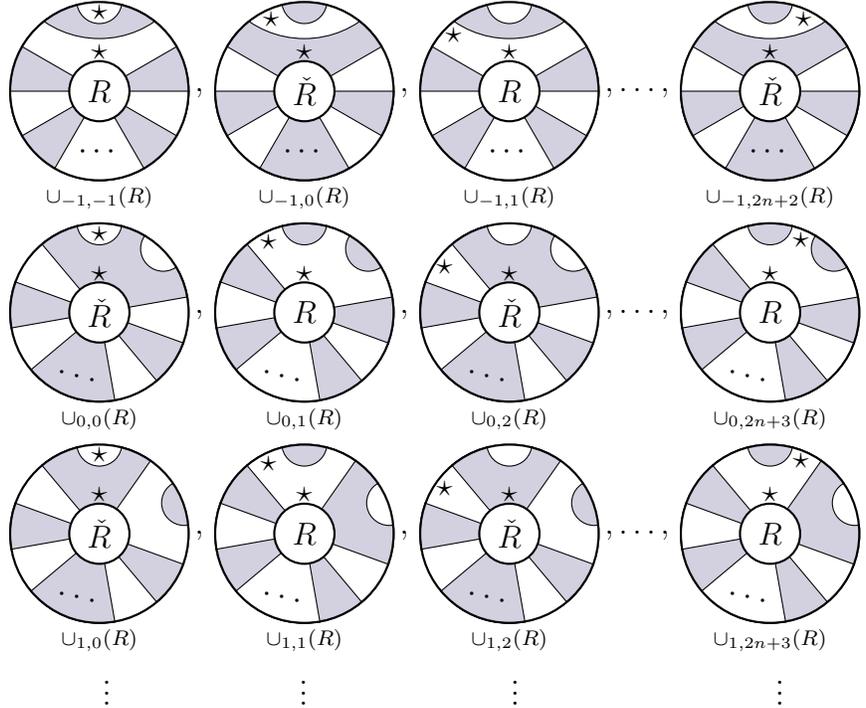
\begin{align*}
\underset{\cup_{-1,-1}(R)}{
\begin{tikzpicture}[baseline = 0cm]
	\clip (0,0) circle (1.2cm);
	\draw[unshaded] (0,0) circle (1.2cm);
	\draw[shaded] (90:2) circle (1.3cm);
	\draw[unshaded] (90:1.2) circle (.3cm);
	\draw[shaded] (0,0)--(150:1.3cm) arc (150:180:1.3cm) --(0,0); 
	\draw[shaded] (0,0)--(210:1.3cm) arc (210:240:1.3cm) --(0,0); 
	\draw[shaded] (0,0)--(30:1.3cm) arc (30:0:1.3cm) --(0,0); 
	\draw[shaded] (0,0)--(-30:1.3cm) arc (-30:-60:1.3cm) --(0,0); 
	\draw[thick, unshaded] (0,0) circle (.4cm);
	\node at (90:.52) {$\star$};
	\node at (90:1.05) {$\star$};
	\node at (0,-.8) {$\cdots$};
	\node at (0,0) {$R$};
	\draw[ultra thick] (0,0) circle (1.2cm);
\end{tikzpicture}
}
\,,\,
\underset{\cup_{-1,0}(R)}{
\begin{tikzpicture}[baseline = 0cm]
	\clip (0,0) circle (1.2cm);
	\draw[shaded] (0,0) circle (1.2cm);
	\draw[unshaded] (90:2) circle (1.3cm);
	\draw[shaded] (90:1.2) circle (.3cm);
	\draw[unshaded] (0,0)--(150:1.3cm) arc (150:180:1.3cm) --(0,0); 
	\draw[unshaded] (0,0)--(210:1.3cm) arc (210:240:1.3cm) --(0,0); 
	\draw[unshaded] (0,0)--(30:1.3cm) arc (30:0:1.3cm) --(0,0); 
	\draw[unshaded] (0,0)--(-30:1.3cm) arc (-30:-60:1.3cm) --(0,0); 
	\draw[thick, unshaded] (0,0) circle (.4cm);
	\node at (90:.52) {$\star$};
	\node at (115:1.05) {$\star$};
	\node at (0,-.8) {$\cdots$};
	\node at (0,0) {$\check{R}$};
	\draw[ultra thick] (0,0) circle (1.2cm);
\end{tikzpicture}
}
\,,\,
\underset{\cup_{-1,1}(R)}{
\begin{tikzpicture}[baseline = 0cm]
	\clip (0,0) circle (1.2cm);
	\draw[unshaded] (0,0) circle (1.2cm);
	\draw[shaded] (90:2) circle (1.3cm);
	\draw[unshaded] (90:1.2) circle (.3cm);
	\draw[shaded] (0,0)--(150:1.3cm) arc (150:180:1.3cm) --(0,0); 
	\draw[shaded] (0,0)--(210:1.3cm) arc (210:240:1.3cm) --(0,0); 
	\draw[shaded] (0,0)--(30:1.3cm) arc (30:0:1.3cm) --(0,0); 
	\draw[shaded] (0,0)--(-30:1.3cm) arc (-30:-60:1.3cm) --(0,0); 
	\draw[thick, unshaded] (0,0) circle (.4cm);
	\node at (90:.52) {$\star$};
	\node at (135:1.05) {$\star$};
	\node at (0,-.8) {$\cdots$};
	\node at (0,0) {$R$};
	\draw[ultra thick] (0,0) circle (1.2cm);
\end{tikzpicture}
}
\,,\dots,\,
\underset{\cup_{-1,2n+2}(R)}{
\begin{tikzpicture}[baseline = 0cm]
	\clip (0,0) circle (1.2cm);
	\draw[shaded] (0,0) circle (1.2cm);
	\draw[unshaded] (90:2) circle (1.3cm);
	\draw[shaded] (90:1.2) circle (.3cm);
	\draw[unshaded] (0,0)--(150:1.3cm) arc (150:180:1.3cm) --(0,0); 
	\draw[unshaded] (0,0)--(210:1.3cm) arc (210:240:1.3cm) --(0,0); 
	\draw[unshaded] (0,0)--(30:1.3cm) arc (30:0:1.3cm) --(0,0); 
	\draw[unshaded] (0,0)--(-30:1.3cm) arc (-30:-60:1.3cm) --(0,0); 
	\draw[thick, unshaded] (0,0) circle (.4cm);
	\node at (90:.52) {$\star$};
	\node at (65:1.05) {$\star$};
	\node at (0,-.8) {$\cdots$};
	\node at (0,0) {$\check{R}$};
	\draw[ultra thick] (0,0) circle (1.2cm);
\end{tikzpicture}
}
\,,\displaybreak[1]\\
\underset{\cup_{0,0}(R)}{
\begin{tikzpicture}[baseline = 0cm]
	\clip (0,0) circle (1.2cm);
	\draw[shaded] (0,0) circle (1.2cm);
	\draw[unshaded] (90:1.2) circle (.3cm);
	\draw[unshaded] (45:1.2) circle (.3cm);
	\draw[unshaded] (0,0)--(130:1.3cm) arc (130:160:1.3cm) --(0,0); 
	\draw[unshaded] (0,0)--(190:1.3cm) arc (190:220:1.3cm) --(0,0); 
	\draw[unshaded] (0,0)--(10:1.3cm) arc (10:-20:1.3cm) --(0,0); 
	\draw[unshaded] (0,0)--(-50:1.3cm) arc (-50:-80:1.3cm) --(0,0); 
	\draw[thick, unshaded] (0,0) circle (.4cm);
	\node at (90:.52) {$\star$};
	\node at (90:1.05) {$\star$};
	\node at (-.5,-.8) {$\cdot$};
	\node at (-.3,-.85) {$\cdot$};
	\node at (-.1,-.9) {$\cdot$};
	\node at (0,0) {$\check{R}$};
	\draw[ultra thick] (0,0) circle (1.2cm);
\end{tikzpicture}
}
\,,\,
\underset{\cup_{0,1}(R)}{
\begin{tikzpicture}[baseline = 0cm]
	\clip (0,0) circle (1.2cm);
	\draw[unshaded] (0,0) circle (1.2cm);
	\draw[shaded] (90:1.2) circle (.3cm);
	\draw[shaded] (45:1.2) circle (.3cm);
	\draw[shaded] (0,0)--(130:1.3cm) arc (130:160:1.3cm) --(0,0); 
	\draw[shaded] (0,0)--(190:1.3cm) arc (190:220:1.3cm) --(0,0); 
	\draw[shaded] (0,0)--(10:1.3cm) arc (10:-20:1.3cm) --(0,0); 
	\draw[shaded] (0,0)--(-50:1.3cm) arc (-50:-80:1.3cm) --(0,0); 
	\draw[thick, unshaded] (0,0) circle (.4cm);
	\node at (90:.52) {$\star$};
	\node at (117:1.05) {$\star$};
	\node at (-.5,-.8) {$\cdot$};
	\node at (-.3,-.85) {$\cdot$};
	\node at (-.1,-.9) {$\cdot$};
	\node at (0,0) {$R$};
	\draw[ultra thick] (0,0) circle (1.2cm);
\end{tikzpicture}
}
\,,\,
\underset{\cup_{0,2}(R)}{
\begin{tikzpicture}[baseline = 0cm]
	\clip (0,0) circle (1.2cm);
	\draw[shaded] (0,0) circle (1.2cm);
	\draw[unshaded] (90:1.2) circle (.3cm);
	\draw[unshaded] (45:1.2) circle (.3cm);
	\draw[unshaded] (0,0)--(130:1.3cm) arc (130:160:1.3cm) --(0,0); 
	\draw[unshaded] (0,0)--(190:1.3cm) arc (190:220:1.3cm) --(0,0); 
	\draw[unshaded] (0,0)--(10:1.3cm) arc (10:-20:1.3cm) --(0,0); 
	\draw[unshaded] (0,0)--(-50:1.3cm) arc (-50:-80:1.3cm) --(0,0); 
	\draw[thick, unshaded] (0,0) circle (.4cm);
	\node at (90:.52) {$\star$};
	\node at (145:1.05) {$\star$};
	\node at (-.5,-.8) {$\cdot$};
	\node at (-.3,-.85) {$\cdot$};
	\node at (-.1,-.9) {$\cdot$};
	\node at (0,0) {$\check{R}$};
	\draw[ultra thick] (0,0) circle (1.2cm);
\end{tikzpicture}
}
\,,\dots,\,
\underset{\cup_{0,2n+3}(R)}{
\begin{tikzpicture}[baseline = 0cm]
	\clip (0,0) circle (1.2cm);
	\draw[unshaded] (0,0) circle (1.2cm);
	\draw[shaded] (90:1.2) circle (.3cm);
	\draw[shaded] (45:1.2) circle (.3cm);
	\draw[shaded] (0,0)--(130:1.3cm) arc (130:160:1.3cm) --(0,0); 
	\draw[shaded] (0,0)--(190:1.3cm) arc (190:220:1.3cm) --(0,0); 
	\draw[shaded] (0,0)--(10:1.3cm) arc (10:-20:1.3cm) --(0,0); 
	\draw[shaded] (0,0)--(-50:1.3cm) arc (-50:-80:1.3cm) --(0,0); 
	\draw[thick, unshaded] (0,0) circle (.4cm);
	\node at (90:.52) {$\star$};
	\node at (67:1.05) {$\star$};
	\node at (-.5,-.8) {$\cdot$};
	\node at (-.3,-.85) {$\cdot$};
	\node at (-.1,-.9) {$\cdot$};
	\node at (0,0) {$R$};
	\draw[ultra thick] (0,0) circle (1.2cm);
\end{tikzpicture}
}
\,,\displaybreak[1]\\
\underset{\cup_{1,0}(R)}{
\begin{tikzpicture}[baseline = 0cm]
	\clip (0,0) circle (1.2cm);
	\draw[shaded] (0,0) circle (1.2cm);
	\draw[unshaded] (90:1.2) circle (.3cm);
	\draw[unshaded] (0,0)--(130:1.3cm) arc (130:160:1.3cm) --(0,0); 
	\draw[unshaded] (0,0)--(190:1.3cm) arc (190:220:1.3cm) --(0,0); 
	\draw[unshaded] (0,0)--(55:1.3cm) arc (55:-20:1.3cm) --(0,0); 
	\draw[unshaded] (0,0)--(-50:1.3cm) arc (-50:-80:1.3cm) --(0,0); 
	\draw[thick, unshaded] (0,0) circle (.4cm);
	\draw[shaded] (20:1.2) circle (.3cm);
	\node at (90:.52) {$\star$};
	\node at (90:1.05) {$\star$};
	\node at (-.5,-.8) {$\cdot$};
	\node at (-.3,-.85) {$\cdot$};
	\node at (-.1,-.9) {$\cdot$};
	\node at (0,0) {$\check{R}$};
	\draw[ultra thick] (0,0) circle (1.2cm);
\end{tikzpicture}
}
\,,\,
\underset{\cup_{1,1}(R)}{
\begin{tikzpicture}[baseline = 0cm]
	\clip (0,0) circle (1.2cm);
	\draw[unshaded] (0,0) circle (1.2cm);
	\draw[shaded] (90:1.2) circle (.3cm);
	\draw[shaded] (0,0)--(130:1.3cm) arc (130:160:1.3cm) --(0,0); 
	\draw[shaded] (0,0)--(190:1.3cm) arc (190:220:1.3cm) --(0,0); 
	\draw[shaded] (0,0)--(55:1.3cm) arc (55:-20:1.3cm) --(0,0); 
	\draw[shaded] (0,0)--(-50:1.3cm) arc (-50:-80:1.3cm) --(0,0); 
	\draw[thick, unshaded] (0,0) circle (.4cm);
	\draw[unshaded] (20:1.2) circle (.3cm);
	\node at (90:.52) {$\star$};
	\node at (117:1.05) {$\star$};
	\node at (-.5,-.8) {$\cdot$};
	\node at (-.3,-.85) {$\cdot$};
	\node at (-.1,-.9) {$\cdot$};
	\node at (0,0) {$R$};
	\draw[ultra thick] (0,0) circle (1.2cm);
\end{tikzpicture}
}
\,,\,
\underset{\cup_{1,2}(R)}{
\begin{tikzpicture}[baseline = 0cm]
	\clip (0,0) circle (1.2cm);
	\draw[shaded] (0,0) circle (1.2cm);
	\draw[unshaded] (90:1.2) circle (.3cm);
	\draw[unshaded] (0,0)--(130:1.3cm) arc (130:160:1.3cm) --(0,0); 
	\draw[unshaded] (0,0)--(190:1.3cm) arc (190:220:1.3cm) --(0,0); 
	\draw[unshaded] (0,0)--(55:1.3cm) arc (55:-20:1.3cm) --(0,0); 
	\draw[unshaded] (0,0)--(-50:1.3cm) arc (-50:-80:1.3cm) --(0,0); 
	\draw[thick, unshaded] (0,0) circle (.4cm);
	\draw[shaded] (20:1.2) circle (.3cm);
	\node at (90:.52) {$\star$};
	\node at (145:1.05) {$\star$};
	\node at (-.5,-.8) {$\cdot$};
	\node at (-.3,-.85) {$\cdot$};
	\node at (-.1,-.9) {$\cdot$};
	\node at (0,0) {$\check{R}$};
	\draw[ultra thick] (0,0) circle (1.2cm);
\end{tikzpicture}
}
\,,\dots,\,
\underset{\cup_{1,2n+3}(R)}{
\begin{tikzpicture}[baseline = 0cm]
	\clip (0,0) circle (1.2cm);
	\draw[unshaded] (0,0) circle (1.2cm);
	\draw[shaded] (90:1.2) circle (.3cm);
	\draw[shaded] (0,0)--(130:1.3cm) arc (130:160:1.3cm) --(0,0); 
	\draw[shaded] (0,0)--(190:1.3cm) arc (190:220:1.3cm) --(0,0); 
	\draw[shaded] (0,0)--(55:1.3cm) arc (55:-20:1.3cm) --(0,0); 
	\draw[shaded] (0,0)--(-50:1.3cm) arc (-50:-80:1.3cm) --(0,0); 
	\draw[thick, unshaded] (0,0) circle (.4cm);
	\draw[unshaded] (20:1.2) circle (.3cm);
	\node at (90:.52) {$\star$};
	\node at (67:1.05) {$\star$};
	\node at (-.5,-.8) {$\cdot$};
	\node at (-.3,-.85) {$\cdot$};
	\node at (-.1,-.9) {$\cdot$};
	\node at (0,0) {$R$};
	\draw[ultra thick] (0,0) circle (1.2cm);
\end{tikzpicture}
}
\,,
\\
\vdots\hspace{2.5cm}\vdots\hspace{2.7cm}\vdots\hspace{3.4cm}\vdots\hspace{1.2cm}
\end{align*}
The index $i$ refers to how many through strings separate the two cups (counting clockwise from the cup at 12 o'clock in the above diagrams), where $i=-1$ denotes two nested cups.
The $j$ refers to the number of strings separating the external boundary interval at 12 o'clock from the interval for the external $\star$, counting counterclockwise (and subtract 1 for nested cups).
Note that $n+k$ strings separating the cups is the same as a rotation (up to switching the shading) of $n-k$ strings separating the cups.

The \underline{second annular basis} of $\gA_{n+2}(R)$ the set of $\cup_{i,j}(R)$ such that $-1\leq i\leq n$, and
$$
j\in 
\begin{cases}
\{-1,0,\dots, 2n+2\} &\text{if }i=-1\\
\{0,\dots, 2n+3\} &\text{if }-1<i<n\\
\{0,\dots, n+1\} &\text{if }i=n.
\end{cases}
$$
If $i=n$, the $n+2$ elements corresponding to $j=0,\dots,n+1$ are as follows
$$
\underset{\cup_{n,0}(R)}{
\begin{tikzpicture}[baseline = 0cm]
	\clip (0,0) circle (1.2cm);
	\draw[shaded] (0,0) circle (1.2cm);
	\draw[unshaded] (90:1.2) circle (.3cm);
	\draw[unshaded] (0,0)--(130:1.3cm) arc (130:410:1.3cm) --(0,0);  
	\draw (0,0)--(160:1.3cm); 
	\draw (0,0)--(210:1.3cm); 
	\draw (0,0)--(240:1.3cm);  
	\draw (0,0)--(20:1.3cm); 
	\draw (0,0)--(-30:1.3cm); 
	\draw (0,0)--(-60:1.3cm); 
	\draw[unshaded] (270:1.2) circle (.3cm);
	\draw[thick, unshaded] (0,0) circle (.4cm);
	\node at (90:.52) {$\star$};
	\node at (90:1.05) {$\star$};
	\node at (0:.85cm) {$\vdots$};
	\node at (180:.85cm) {$\vdots$};
	\node at (0,0) {$\check{R}$};
	\draw[ultra thick] (0,0) circle (1.2cm);
\end{tikzpicture}
}
\,,\,
\underset{\cup_{n,1}(R)}{
\begin{tikzpicture}[baseline = 0cm]
	\clip (0,0) circle (1.2cm);
	\draw[unshaded] (0,0) circle (1.2cm);
	\draw[shaded] (90:1.2) circle (.3cm);
	\draw[shaded] (0,0)--(160:1.3cm) arc (160:130:1.3cm) --(0,0); 
	\draw[shaded] (0,0)--(50:1.3cm) arc (50:20:1.3cm) --(0,0); 
	\draw (0,0)--(210:1.3cm); 
	\draw (0,0)--(240:1.3cm);  
	\draw (0,0)--(-30:1.3cm); 
	\draw (0,0)--(-60:1.3cm); 
	\draw[unshaded] (270:1.2) circle (.3cm);
	\draw[thick, unshaded] (0,0) circle (.4cm);
	\node at (90:.52) {$\star$};
	\node at (115:1.05) {$\star$};
	\node at (0:.85cm) {$\vdots$};
	\node at (180:.85cm) {$\vdots$};
	\node at (0,0) {$\check{R}$};
	\draw[ultra thick] (0,0) circle (1.2cm);
\end{tikzpicture}
}
\,,\,
\underset{\cup_{n,2}(R)}{
\begin{tikzpicture}[baseline = 0cm]
	\clip (0,0) circle (1.2cm);
	\draw[shaded] (0,0) circle (1.2cm);
	\draw[unshaded] (90:1.2) circle (.3cm);
	\draw[unshaded] (0,0)--(130:1.3cm) arc (130:410:1.3cm) --(0,0);  
	\draw (0,0)--(160:1.3cm); 
	\draw (0,0)--(210:1.3cm); 
	\draw (0,0)--(240:1.3cm);  
	\draw (0,0)--(20:1.3cm); 
	\draw (0,0)--(-30:1.3cm); 
	\draw (0,0)--(-60:1.3cm); 
	\draw[unshaded] (270:1.2) circle (.3cm);
	\draw[thick, unshaded] (0,0) circle (.4cm);
	\node at (90:.52) {$\star$};
	\node at (145:1.05) {$\star$};
	\node at (0:.85cm) {$\vdots$};
	\node at (180:.85cm) {$\vdots$};
	\node at (0,0) {$\check{R}$};
	\draw[ultra thick] (0,0) circle (1.2cm);
\end{tikzpicture}
}
\,,\dots,\,
\underset{\cup_{n,n+1}(R)}{
\begin{tikzpicture}[baseline = 0cm]
	\clip (0,0) circle (1.2cm);
	\draw[unshaded] (0,0) circle (1.2cm);
	\draw[unshaded] (90:1.2) circle (.3cm);
	\draw[unshaded] (0,0)--(130:1.3cm) arc (130:410:1.3cm) --(0,0);  
	\draw[shaded] (0,0)--(210:1.3cm) arc (210:240:1.3cm) --(0,0);
	\draw[shaded] (0,0)--(-60:1.3cm) arc (-60:-30:1.3cm) --(0,0);
	\draw (0,0)--(160:1.3cm); 
 	\draw (0,0)--(20:1.3cm); 
	\draw[shaded] (270:1.2) circle (.3cm);
	\draw[thick, unshaded] (0,0) circle (.4cm);
	\node at (90:.52) {$\star$};
	\node at (247.5:1.05) {$\star$};
	\node at (0:.85cm) {$\vdots$};
	\node at (180:.85cm) {$\vdots$};
	\node at (0,0) {$R'$};
	\draw[ultra thick] (0,0) circle (1.2cm);
\end{tikzpicture}
}
$$
where the shading on the bottom in the first 3 pictures depends on the parity of $n$, as does the shading on the top of the final picture, and whether $R'$ is $R$ or $\check{R}$.
\end{defn}

\begin{rem}
Note that $\cup_{-1,-1}(R)=j^2(R)=\JellyfishSquared{n}{R}$\,.
\end{rem}

Recall that the inner product is defined by $\langle x,y\rangle=\Tr(x^*y)$, which is the same as connecting all strings of $x^*$ and $y$. Computing inner products amongst the $\cup_{i,j}(R)$'s amounts to examining the relative positions of caps along the interface between the two diagrams. Since $R$ is uncappable, the entire diagram is zero if a cap from one of the $\cup_{i,j}(R)$'s reaches the other copy of $R$.

It is easy to see that pairing $\cup_{i,j}(R)$ with $\cup_{i',k}(R)$ is non-zero only if $|i-i'|< 3$. 
When the scalar is non-zero differs for the cases $i=-1$ and $i\geq 0$, and there are some exceptional cases when $i=n-1,n$.
\begin{itemize}
\item
When $i=-1$, there are exactly 5, 3, and 2 ways of getting a nonzero scalar when pairing $\cup_{-1,j}(R)$ with $\cup_{i',k}(R)$ for $i'=-1,0,1$ respectively, corresponding to the following relative positions of caps along the interface:
\begin{align*}
&
\begin{tikzpicture}[baseline = 0cm]
	\draw (-1.1,0)--(.8,0);
	\draw (-.6,0) arc (180:0:.15cm);
	\draw (-.9,0) arc (180:0:.45cm);
	\draw (.3,0) arc (0:-180:.15cm);
	\draw (.6,0) arc (0:-180:.45cm);
\end{tikzpicture}
\,,\,
\begin{tikzpicture}[baseline = 0cm]
	\draw (-1.1,0)--(.5,0);
	\draw (-.6,0) arc (180:0:.15cm);
	\draw (-.9,0) arc (180:0:.45cm);
	\draw (0,0) arc (0:-180:.15cm);
	\draw (.3,0) arc (0:-180:.45cm);
\end{tikzpicture}
\,,\,
\begin{tikzpicture}[baseline = 0cm]
	\draw (-1.1,0)--(.2,0);
	\draw (-.6,0) arc (180:0:.15cm);
	\draw (-.9,0) arc (180:0:.45cm);
	\draw (-.3,0) arc (0:-180:.15cm);
	\draw (0,0) arc (0:-180:.45cm);
\end{tikzpicture}
\,,\,
\begin{tikzpicture}[baseline = 0cm]
	\draw (-1.4,0)--(.2,0);
	\draw (-.6,0) arc (180:0:.15cm);
	\draw (-.9,0) arc (180:0:.45cm);
	\draw (-.6,0) arc (0:-180:.15cm);
	\draw (-.3,0) arc (0:-180:.45cm);
\end{tikzpicture}
\,,\,
\begin{tikzpicture}[baseline = 0cm]
	\draw (-1.7,0)--(.2,0);
	\draw (-.6,0) arc (180:0:.15cm);
	\draw (-.9,0) arc (180:0:.45cm);
	\draw (-.9,0) arc (0:-180:.15cm);
	\draw (-.6,0) arc (0:-180:.45cm);
\end{tikzpicture}
&&
i'=-1
\\
&
\begin{tikzpicture}[baseline = 0cm]
	\draw (-1.1,0)--(.5,0);
	\draw (-.9,0) arc (180:0:.15cm);
	\draw (-.3,0) arc (180:0:.15cm);
	\draw (0,0) arc (0:-180:.15cm);
	\draw (.3,0) arc (0:-180:.45cm);
\end{tikzpicture}
\,,\,
\begin{tikzpicture}[baseline = 0cm]
	\draw (-1.1,0)--(.2,0);
	\draw (-.9,0) arc (180:0:.15cm);
	\draw (-.3,0) arc (180:0:.15cm);
	\draw (-.3,0) arc (0:-180:.15cm);
	\draw (0,0) arc (0:-180:.45cm);
\end{tikzpicture}
\,,\,
\begin{tikzpicture}[baseline = 0cm]
	\draw (-1.4,0)--(.2,0);
	\draw (-.9,0) arc (180:0:.15cm);
	\draw (-.3,0) arc (180:0:.15cm);
	\draw (-.6,0) arc (0:-180:.15cm);
	\draw (-.3,0) arc (0:-180:.45cm);
\end{tikzpicture}
&&
i'=0
\\
&
\begin{tikzpicture}[baseline = 0cm]
	\draw (-1.1,0)--(.5,0);
	\draw (-.9,0) arc (180:0:.15cm);
	\draw (-.3,0)--(-.3,.3);
	\draw (0,0) arc (180:0:.15cm);
	\draw (0,0) arc (0:-180:.15cm);
	\draw (.3,0) arc (0:-180:.45cm);
\end{tikzpicture}
\,,\,
\begin{tikzpicture}[baseline = 0cm]
	\draw (-1.1,0)--(.5,0);
	\draw (-.9,0) arc (180:0:.15cm);
	\draw (-.3,0)--(-.3,.3);
	\draw (0,0) arc (180:0:.15cm);
	\draw (-.3,0) arc (0:-180:.15cm);
	\draw (0,0) arc (0:-180:.45cm);
\end{tikzpicture}
&&
i'=1.
\end{align*}
\item
For $0\leq i\leq n-2$, there are exactly 3, 2, and 1 ways of getting a nonzero scalar when pairing $\cup_{i,j}(R)$ with $\cup_{i',k}(R)$ for $i'=i,i+1,i+2$ respectively. The relative positions of caps corresponding to the case $i=0$ are as follows:
\begin{align*}
&
\begin{tikzpicture}[baseline = 0cm]
	\draw (-1.1,0)--(.5,0);
	\draw (-.9,0) arc (180:0:.15cm);
	\draw (-.3,0) arc (180:0:.15cm);
	\draw (-.3,0) arc (0:-180:.15cm);
	\draw (.3,0) arc (0:-180:.15cm);
\end{tikzpicture}
\,,\,
\begin{tikzpicture}[baseline = 0cm]
	\draw (-1.1,0)--(.2,0);
	\draw (-.9,0) arc (180:0:.15cm);
	\draw (-.3,0) arc (180:0:.15cm);
	\draw (-.6,0) arc (0:-180:.15cm);
	\draw (0,0) arc (0:-180:.15cm);
\end{tikzpicture}
\,,\,
\begin{tikzpicture}[baseline = 0cm]
	\draw (-1.4,0)--(.2,0);
	\draw (-.9,0) arc (180:0:.15cm);
	\draw (-.3,0) arc (180:0:.15cm);
	\draw (-.9,0) arc (0:-180:.15cm);
	\draw (-.3,0) arc (0:-180:.15cm);
\end{tikzpicture}
&&
i'=0
\\
&
\begin{tikzpicture}[baseline = 0cm]
	\draw (-1.1,0)--(.5,0);
	\draw (-.9,0) arc (180:0:.15cm);
	\draw (-.3,0)--(-.3,.3);
	\draw (0,0) arc (180:0:.15cm);
	\draw (-.3,0) arc (0:-180:.15cm);
	\draw (.3,0) arc (0:-180:.15cm);
\end{tikzpicture}
\,,\,
\begin{tikzpicture}[baseline = 0cm]
	\draw (-1.1,0)--(.5,0);
	\draw (-.9,0) arc (180:0:.15cm);
	\draw (-.3,0)--(-.3,.3);
	\draw (0,0) arc (180:0:.15cm);
	\draw (-.6,0) arc (0:-180:.15cm);
	\draw (0,0) arc (0:-180:.15cm);
\end{tikzpicture}
&&
i'=1
\\
&
\begin{tikzpicture}[baseline = 0cm]
	\draw (-1.1,0)--(.8,0);
	\draw (-.9,0) arc (180:0:.15cm);
	\draw (-.3,0)--(-.3,.3);
	\draw (0,0)--(0,.3);
	\draw (.3,0) arc (180:0:.15cm);
	\draw (-.3,0) arc (0:-180:.15cm);
	\draw (.3,0) arc (0:-180:.15cm);
\end{tikzpicture}
&&
i'=2.
\end{align*}
\item
For $i=n-1$, there is an additional way of getting a nonzero scalar when pairing $\cup_{n-1,j}(R)$ with $\cup_{n-1,k}(R)$, which makes up for the fact that there is no $\cup_{n+1,k}(R)$. The relative position of caps given by
$$
\begin{tikzpicture}[baseline = 0cm]
	\draw (-1.1,0)--(1.1,0);
	\draw (-.9,0) arc (180:0:.15cm);
	\draw (-.3,0)--(-.3,.3);
	\draw (.3,0)--(.3,.3);
	\draw (.6,0) arc (180:0:.15cm);
	\draw (-.3,0) arc (0:-180:.15cm);
	\draw (.6,0) arc (0:-180:.15cm);
	\draw[very thick] (0,-.3)--(0,.3);
	\node at (0,-.4) {{\scriptsize{$n-1$}}};
\end{tikzpicture}
$$
can be interpreted as $(j-k)\mod (2n+4)\equiv -1$ or $n+2$, depending on the location of the $\star$ above the line. In the former case, the diagram contributes $\sigma^{-1}$, and in the latter, $\sigma^n$.
\item
The case $i=n$ is more subtle. When $i'=n-2$, there are two ways of pairing $\cup_{n,j}(R)$ with $\cup_{n-2,k}(R)$ to get a nonzero scalar, which correspond to the $\star$ placement of
$$
\begin{tikzpicture}[baseline = 0cm]
	\draw (-1.1,0)--(1.1,0);
	\draw (-.9,0) arc (180:0:.15cm);
	\draw (-.3,0)--(-.3,.3);
	\draw (.3,0)--(.3,.3);
	\draw (.6,0) arc (180:0:.15cm);
	\draw (-.3,0) arc (0:-180:.15cm);
	\draw (.6,0) arc (0:-180:.15cm);
	\draw[very thick] (0,-.3)--(0,.3);
	\node at (0,-.4) {{\scriptsize{$n$}}};
\end{tikzpicture}\,,
$$
i.e., $(j-k)\mod (2n+4)\equiv -1$ or $n+1$. In the former case, the diagram contributes a scalar of $\sigma^{-1}$, and in the latter, $\sigma^n \sigma^{-1}$.

When $i'=n-1$, there are four ways to get a nonzero scalar, which correspond to the $\star$ placement of
$$
\begin{tikzpicture}[baseline = 0cm]
	\draw (-.8,0)--(1.1,0);
	\draw (-.6,0) arc (180:0:.15cm);
	\draw (.3,0)--(.3,.3);
	\draw (.6,0) arc (180:0:.15cm);
	\draw (-.3,0) arc (0:-180:.15cm);
	\draw (.6,0) arc (0:-180:.15cm);
	\draw[very thick] (0,-.3)--(0,.3);
	\node at (0,-.4) {{\scriptsize{$n$}}};
\end{tikzpicture}
\text{ and }
\begin{tikzpicture}[baseline = 0cm]
	\draw (-1.1,0)--(.8,0);
	\draw (-.9,0) arc (180:0:.15cm);
	\draw (-.3,0)--(-.3,.3);
	\draw (.3,0) arc (180:0:.15cm);
	\draw (-.3,0) arc (0:-180:.15cm);
	\draw (.6,0) arc (0:-180:.15cm);
	\draw[very thick] (0,-.3)--(0,.3);
	\node at (0,-.4) {{\scriptsize{$n$}}};
\end{tikzpicture}\,.
$$

Finally, when $i'=n$, there are three ways to get a non-zero scalar, corresponding to
$$
\begin{tikzpicture}[baseline = 0cm]
	\draw (-1.1,0)--(1.1,0);
	\draw (-.9,0) arc (180:0:.15cm);
	\draw (-.3,0)--(-.3,.3);
	\draw (.3,0)--(.3,-.3);
	\draw (.3,0) arc (180:0:.15cm);
	\draw (-.3,0) arc (0:-180:.15cm);
	\draw (.9,0) arc (0:-180:.15cm);
	\draw[very thick] (0,-.3)--(0,.3);
	\node at (0,-.4) {{\scriptsize{$n-1$}}};
\end{tikzpicture}
\,,\,
\begin{tikzpicture}[baseline = 0cm]
	\draw (-.8,0)--(.8,0);
	\draw (-.6,0) arc (180:0:.15cm);
	\draw (.3,0) arc (180:0:.15cm);
	\draw (-.3,0) arc (0:-180:.15cm);
	\draw (.6,0) arc (0:-180:.15cm);
	\draw[very thick] (0,-.3)--(0,.3);
	\node at (0,-.4) {{\scriptsize{$n$}}};
\end{tikzpicture}
\,,\,
\begin{tikzpicture}[baseline = 0cm]
	\draw (-1.1,0)--(1.1,0);
	\draw (-.6,0) arc (180:0:.15cm);
	\draw (-.3,0)--(-.3,-.3);
	\draw (.3,0)--(.3,.3);
	\draw (.6,0) arc (180:0:.15cm);
	\draw (-.6,0) arc (0:-180:.15cm);
	\draw (.6,0) arc (0:-180:.15cm);
	\draw[very thick] (0,-.3)--(0,.3);
	\node at (0,-.4) {{\scriptsize{$n-1$}}};
\end{tikzpicture}
$$
(note that the $\star$ placement is determined).
\end{itemize}

The following proposition now follows from the above discussion.

\begin{prop}\label{prop:AnnularInnerProducts}
Assuming $R=R^*$ and $\|R\|^2=\Tr(R^2)=1$, we have the following inner products (linear on the \underline{right}):
$$
\langle \cup_{i',k}(R),\cup_{-1,j}(R)\rangle
=
\begin{array}{c|c||c|c|c|c|c|}
\multicolumn{2}{c||}{} & \multicolumn{5}{|c|}{(j-k)\mod (2n+4)}
\\\cline{3-7}
\multicolumn{2}{c||}{}& -2& -1 &0 & 1 & 2 
\\\hline\hline
&-1 & \omega_R^{-1} & \sigma_R^{-1} & [2]^2 & \sigma_R & \omega_R
\\\cline{2-7}
i' &0 &0 & [2] \sigma_R^{-1} & [2] & [2] \sigma_R & 0
\\\cline{2-7}
&1 & 0 & 0 & 1 & \sigma_R & 0
\\\hline
\end{array}
$$
and is zero otherwise.

For $0\leq i',i\leq n-1$, we have
$$
\langle \cup_{i',k}(R),\cup_{i,j}(R)\rangle
=
\begin{array}{c|c||c|c|c|}
\multicolumn{2}{c||}{} & \multicolumn{3}{|c|}{(j-k)\mod (2n+4)}
\\\cline{3-5}
\multicolumn{2}{c||}{}&  -1 & 0 & 1
\\\hline\hline
&-2 & \sigma_R^{-1} & 0 &  0 
\\\cline{2-5}
&-1 & [2]\sigma_R^{-1} & [2] & 0
\\\cline{2-5}
i'-i&0 & \sigma_R^{-1} & [2]^2 &\sigma_R
\\\cline{2-5}
&1 & 0 & [2] & [2]\sigma_R
\\\cline{2-5}
&2 & 0 & 0 & \sigma_R
\\\hline
\end{array}
$$
and is zero otherwise, with the exception that
$$
\langle\cup_{n-1,k}(R),\cup_{n-1,j}\rangle = \sigma_R^n \text{ if }j-k\equiv n+2\mod 2n+4.
$$

For $i=n$ and $i'<n$, we have
$$
\langle \cup_{i',k}(R),\cup_{n-1,j}(R)\rangle
=
\begin{array}{c|c||c|c|c|c|}
\multicolumn{2}{c||}{} & \multicolumn{4}{|c|}{(j-k)\mod (2n+4)}
\\\cline{3-6}
\multicolumn{2}{c||}{}&  -1 & 0 & n+1 & n+2
\\\hline\hline
i'&n-2 & \sigma_R^{-1} & 0 &  \sigma_R^n\sigma_R^{-1} & 0
\\\cline{3-6}
&n-1 & [2]\sigma_R^{-1} & [2] & [2]\sigma_R^n\sigma_R^{-1} & [2]\sigma_R^n
\\\hline
\end{array}
$$
and is zero otherwise.

Finally, if $i=i'=n$, then we have
$$
\langle \cup_{n,k}(R),\cup_{n,j}(R)\rangle
=
\begin{cases}
\sigma_R^n\sigma_R^{-1} & \text{if }(j-k)\mod (n+2)\equiv -1 \text{ and } j=n+1\\
\sigma_R^{-1} & \text{if }(j-k)\mod (n+2)\equiv -1  \text{ and } j<n+1\\
[2]^2 & \text{if }(j-k)\mod (n+2)\equiv 0 \\
\sigma_R & \text{if }(j-k)\mod (n+2)\equiv 1\text{ and }j>0\\ 
\sigma_R^n\sigma_R& \text{if }(j-k)\mod (n+2)\equiv -1 \text{ and }j=0\\
0 &\text{else.}
\end{cases}
$$
\end{prop}

\begin{rem}
The concerned reader may wonder if we have missed a case or two amidst this muddle of indices.
Be reassured that we have checked these inner products numerically for the generators of each of our examples directly in the graph planar algebra. 
See Section \ref{sec:Checking} for more details. 
\end{rem}

\begin{rem}\label{rem:DualAnnularBasis}
In this article, we do not give a formula for the dual basis $\widehat{\cup}_{i,j}(R)$ in terms of the $\cup_{i,j}(R)$'s, i.e., the change of basis matrix from the annular basis to the dual annular basis. Instead, we find the dual annular basis for our examples by inverting the matrix of inner products given by Proposition \ref{prop:AnnularInnerProducts}. 

As in \cite[Remark 3.7]{1208.3637}, if $W$ is the matrix of inner products of the $\cup_{i,j}(R)$'s, then the change of basis matrix from the column vectors representing the annular basis $U$ to the column vectors representing the dual basis $\widehat{U}$ is $\overline{W^{-1}}$, i.e., $\overline{W^{-1}} U = \widehat{U}$. (The inner product is linear on the \underline{right}.) If $\widehat{c}$ is the row vector of coefficients in the dual basis for an annular consequence $x$, i.e., $x=\widehat{c}\cdot \widehat{U}$, then the row vector of coefficients in the annular basis is given by $c=\widehat{c}\,\overline{W^{-1}}$.

It would certainly be useful to have a general formula for the dual annular basis in terms of the annular basis. While such a computation is routine, it would be demanding, and we leave it for another time.
\end{rem}

\subsection{Towards computing the second dual annular basis}\label{sec:SecondDualBasis}

While we do not compute the change of basis matrix from the second annular basis to the second dual annular basis here, we record a formula for the two-cup Jones-Wenzl which will be highly instrumental in its future calculation. We use the first and second formulas below for calculations in Section \ref{sec:Formulas}.

\begin{defn}
The \underline{two-cup Jones-Wenzl} 
is the sum of all the terms in the Jones-Wenzl $\jw{k}$ with at most two cups on the top and bottom. 
\end{defn}

A formula for the two-cup Jones-Wenzl in terms of Temperley-Lieb diagrams can be deduced easily from \cite[Proposition 3.8]{MR1446615}, \cite[Introduction, Equations (3)]{MR2375712}, and \cite{morrison} (notice that our sign convention disagrees with that of \cite[Proposition 3.8]{MR1446615} and \cite{morrison}; to go from their equation to ours, all $[2k]$ should be replaced by $-[2k]$).

\begin{fact}\label{fact:2cup}
The coefficients of the diagrammatic basis elements with at most two cups on the top and bottom in the Jones-Wenzl idempotent are as follows.
\begin{align*}
\underset{\in \jw{k}}{\coeff}
\left(
\begin{tikzpicture}[baseline = -.1cm]
	\draw (0,-.8)--(0,.8);
	\node at (-.2,0) {{\scriptsize{$k$}}};
\end{tikzpicture}
\hspace{.2cm}
\right)
&=
1
\displaybreak[1]\\
\underset{\in \jw{k}}{\coeff}
\left(
\begin{tikzpicture}[baseline = -.1cm]
	\draw (-.5,-.8)--(-.5,.8);
	\draw (-.3,.8) arc (-180:0:.2cm);
	\draw (-.3,-.8)--(.3,.8);
	\draw (-.1,-.8) arc (180:0:.2cm);
	\draw (.5,-.8)--(.5,.8);
	\node at (-.7,0) {{\scriptsize{$a$}}};
	\node at (-.2,0) {{\scriptsize{$b$}}};
	\node at (.7,0) {{\scriptsize{$c$}}};
\end{tikzpicture}
\right)
&=
(-1)^{b+1}\frac{[a+1][c+1]}{[k]}
\displaybreak[1]\\
\underset{\in \jw{k}}{\coeff}
\left(
\begin{tikzpicture}[baseline = -.1cm]
	\draw (-1,-.8)--(-1,.8);
	\draw (-.8,.8) arc (-180:0:.2cm);
	\draw (-.8,-.8)--(-.2,.8);
	\draw (-.6,-.8) arc (180:0:.2cm);
	\draw (0,-.8)--(0,.8);
	\node at (-1.2,0) {{\scriptsize{$a$}}};
	\node at (-.7,0) {{\scriptsize{$b$}}};
	\node at (-.2,0) {{\scriptsize{$c$}}};
	\node at (.7,0) {{\scriptsize{$d$}}};
	\node at (1.2,0) {{\scriptsize{$e$}}};
	\draw (1,-.8)--(1,.8);
	\draw (.8,.8) arc (0:-180:.2cm);
	\draw (.8,-.8)--(.2,.8);
	\draw (.6,-.8) arc (0:180:.2cm);
\end{tikzpicture}
\right)
&=
\frac{(-1)^{b+d}[a+1][e+1]}{[k][k-1]}
\left(
[2][a+b+1][d+e+1]+[c+2][k-1]
\right)
\displaybreak[1]\\
\underset{\in \jw{k}}{\coeff}
\left(
\begin{tikzpicture}[baseline = -.1cm]
	\draw (-1,-.8)--(-1,.8);
	\draw (-.8,-.8) arc (180:0:.2cm);
	\draw (-.8,.8)--(-.2,-.8);
	\draw (-.6,.8) arc (-180:0:.2cm);
	\draw (0,-.8)--(0,.8);
	\node at (-1.2,0) {{\scriptsize{$a$}}};
	\node at (-.7,0) {{\scriptsize{$b$}}};
	\node at (-.2,0) {{\scriptsize{$c$}}};
	\node at (.7,0) {{\scriptsize{$d$}}};
	\node at (1.2,0) {{\scriptsize{$e$}}};
	\draw (1,-.8)--(1,.8);
	\draw (.8,.8) arc (0:-180:.2cm);
	\draw (.8,-.8)--(.2,.8);
	\draw (.6,-.8) arc (0:180:.2cm);
\end{tikzpicture}
\right)
&=
\text{ same as above}
\displaybreak[1]\\
\underset{\in \jw{k}}{\coeff}
\left(
\begin{tikzpicture}[baseline = -.1cm]
	\draw (-1,-.8)--(-1,.8);
	\draw (-.8,-.8) arc (180:0:.2cm);
	\draw (-.8,.8)--(-.2,-.8);
	\draw (-.4,.8) arc (-180:0:.2cm);
	\draw (-.6,.8)--(.6,-.8);
	\node at (-1.2,0) {{\scriptsize{$a$}}};
	\node at (-.7,0) {{\scriptsize{$b$}}};
	\node at (-.2,0) {{\scriptsize{$c$}}};
	\node at (.7,0) {{\scriptsize{$d$}}};
	\node at (1.2,0) {{\scriptsize{$e$}}};
	\draw (1,-.8)--(1,.8);
	\draw (.8,.8) arc (0:-180:.2cm);
	\draw (.8,-.8)--(.2,.8);
	\draw (.4,-.8) arc (0:180:.2cm);
\end{tikzpicture}
\right)
&=
\frac{(-1)^{b+d}[a+1][e+1]}{[k][k-1]}
\left(
[2][a+b+2][d+e+2]
\right)
\displaybreak[1]\\
\underset{\in \jw{k}}{\coeff}
\left(
\begin{tikzpicture}[baseline = -.1cm]
	\draw (-.9,-.8)--(-.9,.8);
	\draw (-.7,.8) arc (-180:0:.2cm);
	\draw (-.7,-.8)--(-.1,.8);
	\draw (-.2,-.8) arc (180:0:.2cm);
	\draw (-.3,-.8) arc (180:0:.3cm);
	\node at (-1.1,0) {{\scriptsize{$a$}}};
	\node at (-.6,0) {{\scriptsize{$b$}}};
	\node at (.6,0) {{\scriptsize{$d$}}};
	\node at (1.1,0) {{\scriptsize{$e$}}};
	\draw (.9,-.8)--(.9,.8);
	\draw (.7,.8) arc (0:-180:.2cm);
	\draw (.7,-.8)--(.1,.8);
\end{tikzpicture}
\right)
&=
\frac{(-1)^{b+d+1}[a+1][e+1]}{[k][k-1]}
\left(
[a+b+1][d+e+1]+[k-1]
\right)
\displaybreak[1]\\
\underset{\in \jw{k}}{\coeff}
\left(
\begin{tikzpicture}[baseline = -.1cm]
	\draw (-.8,-.8)--(-.8,.8);
	\draw (-.4,-.8) arc (180:0:.2cm);
	\draw (-.5,-.8) arc (180:0:.3cm);
	\draw (-.4,.8) arc (-180:0:.2cm);
	\draw (-.6,.8)--(.6,-.8);
	\node at (-1,0) {{\scriptsize{$a$}}};
	\node at (-.2,0) {{\scriptsize{$c$}}};
	\node at (.7,0) {{\scriptsize{$d$}}};
	\node at (1.2,0) {{\scriptsize{$e$}}};
	\draw (1,-.8)--(1,.8);
	\draw (.8,.8) arc (0:-180:.2cm);
	\draw (.8,-.8)--(.2,.8);
\end{tikzpicture}
\right)
&=
\frac{(-1)^{d+1}[a+1][e+1]}{[k][k-1]}
\left(
[a+2][d+e+2]
\right)
\displaybreak[1]\\
\underset{\in \jw{k}}{\coeff}
\left(
\begin{tikzpicture}[baseline = -.1cm]
	\draw (-.8,-.8)--(-.8,.8);
	\draw (-.5,-.8) arc (180:0:.2cm);
	\draw (-.6,-.8) arc (180:0:.3cm);
	\draw (-.2,.8) arc (-180:0:.2cm);
	\draw (-.3,.8) arc (-180:0:.3cm);
	\draw (-.6,.8)--(.3,-.8);
	\node at (-1,0) {{\scriptsize{$a$}}};
	\node at (-.3,0) {{\scriptsize{$c$}}};
	\node at (.7,0) {{\scriptsize{$e$}}};
	\draw (.5,-.8)--(.5,.8);
\end{tikzpicture}
\right)
&=
\frac{[a+1][e+1]}{[2][k][k-1]}
\left(
[a+2][e+2]
\right).
\end{align*} 
Note that coefficients for diagrams are invariant under the $\Z/2\oplus \Z/2$ symmetries of the rectangle (horizontal and vertical flipping).
\end{fact}

\section{Projections and inner products of trains}\label{sec:Formulas}

As in the previous section, we continue to use Assumptions \ref{assume:Generators}, \ref{assume:SpanAlgebras}, and \ref{assume:Tetrahedral}.

To derive two-strand jellyfish relations, we need to analyze all reduced $\fB$-trains in $\cP_{n+2,+}$, in particular their projections to $\TL_{n+2,+}$, their projections to the space of second annular consequence of $\fB$, and their pairwise inner products.

We express some projections to Temperley-Lieb and annular consequences in terms of dual bases.  Using our conventions, the formula for these projections is as below: 
\begin{rem}
Suppose $\{v_1,\dots,v_k\}\subset V$ is a basis for the finite dimensional Hilbert space $V$. Let $\{\widehat{v_1},\dots, \widehat{v_k}\}$ be the dual basis $V$, defined by $\langle \widehat{v_i}, v_j \rangle = \delta_{i,j}$, where the inner product is linear on the right. If $x\in V$, we have 
$$
x=\sum_{i=1}^k \langle v_i, x\rangle \widehat{v_i}. 
$$
\end{rem}

In what follows, $P,Q,R,S,T$ are always elements of  $\fB$. We will first need a few results about certain Temperley-Lieb dual basis elements.

\subsection{Some Temperley-Lieb dual basis elements}\label{sec:TLDualBasis}

We now discuss certain elements of the basis which is dual to the usual diagrammatic basis of $\TL_{k}$.

\begin{lem}\label{lem:QuantumIdentity}
If $ a,b \geq 0$ and $a+b=n$, then
$$
[a+2][b+1]-[a+1][b]=[n+2].
$$
\end{lem}
\begin{proof}
Immediate from the formula
$\D
[k][\ell] = 
%
\sum_{\substack{
|k-\ell|<j<k+\ell
\\ 
j\equiv |k-\ell|+1 \mod 2
}}
[j]
$.
%
\end{proof}

\begin{lem}\label{lem:TLDualBasis}
The element dual to 
$
\begin{tikzpicture}[baseline = -.1cm]
	\draw[thick] (-.6,-.4)--(-.6,.4)--(.6,.4)--(.6,-.4)--(-.6,-.4);
	\draw (-.4,-.4)--(-.4,.4);
	\draw (-.2,.4)--(.4,-.4);
	\draw (-.2,-.4) arc (180:0:.2cm);
	\draw (0,.4) arc (-180:0:.2cm);
	\node at (-.25,0) {{\scriptsize{$a$}}};
	\node at (.35,0) {{\scriptsize{$b$}}};
\end{tikzpicture}
\in \TL_{n+2,+}
$
is given by
$$
\begin{tikzpicture}[baseline = -.1cm]
	\draw[thick] (-.6,-.4)--(-.6,.4)--(.6,.4)--(.6,-.4)--(-.6,-.4);
	\draw (-.4,-.4)--(-.4,.4);
	\draw (-.2,.4)--(.4,-.4);
	\draw (-.2,-.4) arc (180:0:.2cm);
	\draw (0,.4) arc (-180:0:.2cm);
	\node at (-.25,0) {{\scriptsize{$a$}}};
	\node at (.35,0) {{\scriptsize{$b$}}};
\end{tikzpicture}
^{\widehat{\hs\hs}}
=
\frac{[a+1][b+1]}{[n+2]^2}
\underbrace{
\begin{tikzpicture}[baseline = -.7cm]
	\draw (0,1)--(0,.4);
	\draw (-.3,.4)--(-.3,-1);
	\draw (0,-1)--(0,-2.2);
	\draw (1.4,-1.4)--(1.4,-2.2);
	\draw (0,-.2) arc (180:270:.35cm) -- (1.2,-.55) arc (90:0:.45cm);
	\node at (-.4,.8) {{\scriptsize{$n+1$}}};
	\node at (-.4,-2) {{\scriptsize{$a+1$}}};
	\node at (1,-2) {{\scriptsize{$b+1$}}};
	\node at (-.5,-.6) {{\scriptsize{$a$}}};
	\node at (1.7,-.6) {{\scriptsize{$b$}}};
	\filldraw[unshaded,thick] (-.6,.6)--(.6,.6)--(.6,-.2)--(-.6,-.2)--(-.6,.6);
	\node at (0,.2) {$\jw{n+1}$};
	\filldraw[unshaded,thick] (-.6,-1)--(.6,-1)--(.6,-1.8)--(-.6,-1.8)--(-.6,-1);
	\filldraw[unshaded,thick] (.8,-1)--(2,-1)--(2,-1.8)--(.8,-1.8)--(.8,-1);
	\node at (0,-1.4) {$\jw{a+1}$};
	\node at (1.4,-1.4) {$\jw{b+1}$};
	\draw (.4,-.2) arc (180:360:.2cm) -- (.8,1);
	\draw (.4,-1) arc (180:0:.3cm);
\end{tikzpicture}
}_D
-
\frac{(-1)^b [a+1]}{[n+2][n+3]}
\begin{tikzpicture}[baseline = -.1cm]
	\draw (0,.8)--(0,-.8);
	\node at (-.4,.6) {{\scriptsize{$n+2$}}};
	\node at (-.4,-.6) {{\scriptsize{$n+2$}}};
	\filldraw[unshaded,thick] (-.6,.4)--(.6,.4)--(.6,-.4)--(-.6,-.4)--(-.6,.4);
	\node at (0,0) {$\jw{n+2}$};
\end{tikzpicture}.
$$ 
To find the element dual to
$
\begin{tikzpicture}[baseline = -.1cm,yscale=-1]
	\draw[thick] (-.6,-.4)--(-.6,.4)--(.6,.4)--(.6,-.4)--(-.6,-.4);
	\draw (-.4,-.4)--(-.4,.4);
	\draw (-.2,.4)--(.4,-.4);
	\draw (-.2,-.4) arc (180:0:.2cm);
	\draw (0,.4) arc (-180:0:.2cm);
	\node at (-.25,0) {{\scriptsize{$a$}}};
	\node at (.35,0) {{\scriptsize{$b$}}};
\end{tikzpicture}
\in \TL_{n+2,+}
$,
maintain the coefficients and vertically reflect the diagrams in the above formula.
\end{lem}
\begin{proof}
Note that the middle diagram $D$ in the above equation has non-zero inner product only with $1_{n+2}$ and
$
\begin{tikzpicture}[baseline = -.1cm]
	\draw[thick] (-.6,-.4)--(-.6,.4)--(.6,.4)--(.6,-.4)--(-.6,-.4);
	\draw (-.4,-.4)--(-.4,.4);
	\draw (-.2,.4)--(.4,-.4);
	\draw (-.2,-.4) arc (180:0:.2cm);
	\draw (0,.4) arc (-180:0:.2cm);
	\node at (-.25,0) {{\scriptsize{$a$}}};
	\node at (.35,0) {{\scriptsize{$b$}}};
\end{tikzpicture}
$. We already know that $\widehat{1_{n+2}}=\jw{n+2}/[n+3]$, so we have
$$
\begin{tikzpicture}[baseline = -.1cm]
	\draw[thick] (-.6,-.4)--(-.6,.4)--(.6,.4)--(.6,-.4)--(-.6,-.4);
	\draw (-.4,-.4)--(-.4,.4);
	\draw (-.2,.4)--(.4,-.4);
	\draw (-.2,-.4) arc (180:0:.2cm);
	\draw (0,.4) arc (-180:0:.2cm);
	\node at (-.25,0) {{\scriptsize{$a$}}};
	\node at (.35,0) {{\scriptsize{$b$}}};
\end{tikzpicture}
^{\widehat{\hs\hs}}
=
\frac{1}{\left\langle D, 
\begin{tikzpicture}[baseline = -.1cm]
	\draw[thick] (-.6,-.4)--(-.6,.4)--(.6,.4)--(.6,-.4)--(-.6,-.4);
	\draw (-.4,-.4)--(-.4,.4);
	\draw (-.2,.4)--(.4,-.4);
	\draw (-.2,-.4) arc (180:0:.2cm);
	\draw (0,.4) arc (-180:0:.2cm);
	\node at (-.25,0) {{\scriptsize{$a$}}};
	\node at (.35,0) {{\scriptsize{$b$}}};
\end{tikzpicture}\,
\right\rangle} 
\left(D
-
\langle D, 1_{n+2}\rangle
\frac{\jw{n+2}}{[n+3]}
\right).
$$
A routine calculation computes the necessary inner products. First,
$$
\left\langle D, 1_{n+2}\right\rangle
=
\begin{tikzpicture}[baseline = .5cm,yscale=-1]
	\draw (-.2,0)--(-.2,-1.2);
	\draw (-.4,-.2) arc (0:-180:.2cm) -- (-.8,.6) arc (180:0:.2cm);
	\draw (-1,-1.2)--(-1,.6) arc (180:0:.4cm);
	\draw (.2,0)--(.2,-1.8) arc (0:-180:.3cm);
	\draw (-.6,-1.8) .. controls ++(270:.8cm) and ++(270:1cm) .. (.8,-1.4)--(.8,.6) arc (0:180:.3cm);
	\node at (-.35,-.7) {{\scriptsize{$b$}}};
	\node at (-.6,-.5) {{\scriptsize{$a$}}};
	\node at (.65,-.6) {{\scriptsize{$b$}}};
	\filldraw[unshaded,thick] (-.6,.6)--(.6,.6)--(.6,-.2)--(-.6,-.2)--(-.6,.6);
	\node at (0,.2) {$\jw{n+1}$};
	\filldraw[unshaded,thick] (-1.2,-1)--(0,-1)--(0,-1.8)--(-1.2,-1.8)--(-1.2,-1);
	\node at (-.6,-1.4) {$\jw{b+1}$};
\end{tikzpicture}
=
\frac{(-1)^{b}[n+2]}{[b+1]},
$$
since the only diagram in the top $\jw{b+1}$ which contributes to the closed diagram is 
$
\begin{tikzpicture}[baseline = -.1cm]
	\draw[thick] (-.6,-.4)--(-.6,.4)--(.7,.4)--(.7,-.4)--(-.6,-.4);
	\draw (-.4,.4)--(.2,-.4);
	\draw (-.4,-.4) arc (180:0:.2cm);
	\draw (-.2,.4) arc (-180:0:.2cm);
	\node at (.35,0) {{\scriptsize{$b-1$}}};
\end{tikzpicture}
$ (the coefficient of this diagram in $\jw{b+1}$ is given in Fact \ref{fact:2cup}).
Next, we calculate
\begin{align}
\left\langle D, 
\begin{tikzpicture}[baseline = -.1cm]
	\draw[thick] (-.6,-.4)--(-.6,.4)--(.6,.4)--(.6,-.4)--(-.6,-.4);
	\draw (-.4,-.4)--(-.4,.4);
	\draw (-.2,.4)--(.4,-.4);
	\draw (-.2,-.4) arc (180:0:.2cm);
	\draw (0,.4) arc (-180:0:.2cm);
	\node at (-.25,0) {{\scriptsize{$a$}}};
	\node at (.35,0) {{\scriptsize{$b$}}};
\end{tikzpicture}\,
\right\rangle
&=
\begin{tikzpicture}[baseline = .6cm,yscale=-1]
	\draw (-.4,.6) arc (0:180:.25cm) -- (-.9,-1.8) arc (-180:0:.25cm);
	\draw (1.8,-1.8) arc (-180:0:.2cm) -- (2.2,-1.4) .. controls ++(90:1cm) and ++(0:1cm) .. (.6,1) arc (90:180:.4cm); 
	\draw (-.3,0)--(-.3,-1);
	\draw (0,-.2) arc (180:270:.35cm) -- (1.2,-.55) arc (90:0:.45cm);
	\node at (-.4,-.6) {{\scriptsize{$a$}}};
	\node at (-1,-.6) {{\scriptsize{$a$}}};
	\node at (1,-.4) {{\scriptsize{$b$}}};
	\node at (1.9,-.4) {{\scriptsize{$b$}}};
	\filldraw[unshaded,thick] (-.6,.6)--(.6,.6)--(.6,-.2)--(-.6,-.2)--(-.6,.6);
	\node at (0,.2) {$\jw{n+1}$};
	\filldraw[unshaded,thick] (-.6,-1)--(.6,-1)--(.6,-1.8)--(-.6,-1.8)--(-.6,-1);
	\filldraw[unshaded,thick] (.8,-1)--(2,-1)--(2,-1.8)--(.8,-1.8)--(.8,-1);
	\node at (0,-1.4) {$\jw{a+1}$};
	\node at (1.4,-1.4) {$\jw{b+1}$};
	\draw (.4,-.2) arc (180:360:.2cm) -- (.8,.6) arc (0:180:.2cm);
	\draw (.4,-1) arc (180:0:.3cm);
	\draw (.4,-1.8) arc (-180:0:.3cm);
\end{tikzpicture}
\label{eqn:HardTL1}\\
&=
[n+2]\left(\frac{[a+2]}{[a+1]}-\frac{[b]}{[b+1]}\right)
\label{eqn:HardTL2}\\
&=
\frac{[n+2]^2}{[a+1][b+1]},
\label{eqn:HardTL3}
\end{align}
where Equation \eqref{eqn:HardTL2} follows since the only two terms in the top $\jw{a+1}$ which contribute to the closed diagram are $1_{a+1}$ and 
$
\begin{tikzpicture}[baseline = -.1cm]
	\draw[thick] (-.6,-.4)--(-.6,.4)--(.6,.4)--(.6,-.4)--(-.6,-.4);
	\draw (0,.4) arc (-180:0:.2cm);
	\draw (-.4,-.4)--(-.4,.4);
	\draw (0,-.4) arc (180:0:.2cm);
	\node at (0,0) {{\scriptsize{$a-1$}}};
\end{tikzpicture}
$\,. Equation \eqref{eqn:HardTL3} now follows by Lemma \ref{lem:QuantumIdentity}.

(Note that the value of the diagram that appears in Equation \eqref{eqn:HardTL1} must be symmetric in $a$ and $b$, but the quantity in Equation \eqref{eqn:HardTL2} does not appear symmetric in $a$ and $b$. This gave a hint that some quantum number identity should hold, which motivated Lemma \ref{lem:QuantumIdentity}.)

The last claim is now immediate.
\end{proof}

\comment{
\begin{lem}\label{lem:HardTLDiagrams}
We have the following inner products.
\be
\item
$
\left\langle
\begin{tikzpicture}[baseline = -.6cm]
	\draw (0,1)--(0,.4);
	\draw (-.3,.4)--(-.3,-1);
	\draw (0,-1)--(0,-2.2);
	\draw (1.4,-1.4)--(1.4,-2.2);
	\draw (0,-.2) arc (180:270:.35cm) -- (1.2,-.55) arc (90:0:.45cm);
	\node at (-.4,.8) {{\scriptsize{$n+1$}}};
	\node at (-.4,-2) {{\scriptsize{$a+1$}}};
	\node at (1,-2) {{\scriptsize{$b+1$}}};
	\node at (-.5,-.6) {{\scriptsize{$a$}}};
	\node at (1.7,-.6) {{\scriptsize{$b$}}};
	\filldraw[unshaded,thick] (-.6,.6)--(.6,.6)--(.6,-.2)--(-.6,-.2)--(-.6,.6);
	\node at (0,.2) {$\jw{n+1}$};
	\filldraw[unshaded,thick] (-.6,-1)--(.6,-1)--(.6,-1.8)--(-.6,-1.8)--(-.6,-1);
	\filldraw[unshaded,thick] (.8,-1)--(2,-1)--(2,-1.8)--(.8,-1.8)--(.8,-1);
	\node at (0,-1.4) {$\jw{a+1}$};
	\node at (1.4,-1.4) {$\jw{b+1}$};
	\draw (.4,-.2) arc (180:360:.2cm) -- (.8,1);
	\draw (.4,-1) arc (180:0:.3cm);
\end{tikzpicture}
\,
,
\begin{tikzpicture}[baseline = -.6cm]
	\draw (0,1)--(0,.4);
	\draw (-.3,.4)--(-.3,-1);
	\draw (0,-1)--(0,-2.2);
	\draw (1.4,-1.4)--(1.4,-2.2);
	\draw (0,-.2) arc (180:270:.35cm) -- (1.2,-.55) arc (90:0:.45cm);
	\node at (-.4,.8) {{\scriptsize{$n+1$}}};
	\node at (-.4,-2) {{\scriptsize{$c+1$}}};
	\node at (1,-2) {{\scriptsize{$d+1$}}};
	\node at (-.5,-.6) {{\scriptsize{$c$}}};
	\node at (1.7,-.6) {{\scriptsize{$d$}}};
	\filldraw[unshaded,thick] (-.6,.6)--(.6,.6)--(.6,-.2)--(-.6,-.2)--(-.6,.6);
	\node at (0,.2) {$\jw{n+1}$};
	\filldraw[unshaded,thick] (-.6,-1)--(.6,-1)--(.6,-1.8)--(-.6,-1.8)--(-.6,-1);
	\filldraw[unshaded,thick] (.8,-1)--(2,-1)--(2,-1.8)--(.8,-1.8)--(.8,-1);
	\node at (0,-1.4) {$\jw{c+1}$};
	\node at (1.4,-1.4) {$\jw{d+1}$};
	\draw (.4,-.2) arc (180:360:.2cm) -- (.8,1);
	\draw (.4,-1) arc (180:0:.3cm);
\end{tikzpicture}
\right\rangle
=
\begin{cases}
(-1)^{c-a}\frac{[n+2]^2}{[b+1][c+1]}&\text{if }a<c \\ 
\frac{[n+2]^2}{[a+1][b+1]}&\text{if }a=c\\
(-1)^{a-c}\frac{[n+2]^2}{[a+1][d+1]}&\text{if }a>c \\ 
\end{cases}
$
\item
$
\D
\left\langle
\begin{tikzpicture}[baseline = -.6cm]
	\draw (0,1)--(0,.4);
	\draw (-.3,.4)--(-.3,-1);
	\draw (0,-1)--(0,-2.2);
	\draw (1.4,-1.4)--(1.4,-2.2);
	\draw (0,-.2) arc (180:270:.35cm) -- (1.2,-.55) arc (90:0:.45cm);
	\node at (-.4,.8) {{\scriptsize{$n+1$}}};
	\node at (-.4,-2) {{\scriptsize{$a+1$}}};
	\node at (1,-2) {{\scriptsize{$b+1$}}};
	\node at (-.5,-.6) {{\scriptsize{$a$}}};
	\node at (1.7,-.6) {{\scriptsize{$b$}}};
	\filldraw[unshaded,thick] (-.6,.6)--(.6,.6)--(.6,-.2)--(-.6,-.2)--(-.6,.6);
	\node at (0,.2) {$\jw{n+1}$};
	\filldraw[unshaded,thick] (-.6,-1)--(.6,-1)--(.6,-1.8)--(-.6,-1.8)--(-.6,-1);
	\filldraw[unshaded,thick] (.8,-1)--(2,-1)--(2,-1.8)--(.8,-1.8)--(.8,-1);
	\node at (0,-1.4) {$\jw{a+1}$};
	\node at (1.4,-1.4) {$\jw{b+1}$};
	\draw (.4,-.2) arc (180:360:.2cm) -- (.8,1);
	\draw (.4,-1) arc (180:0:.3cm);
\end{tikzpicture}
\,
,
\begin{tikzpicture}[baseline = .6cm,yscale=-1]
	\draw (0,1)--(0,.4);
	\draw (-.3,.4)--(-.3,-1);
	\draw (0,-1)--(0,-2.2);
	\draw (1.4,-1.4)--(1.4,-2.2);
	\draw (0,-.2) arc (180:270:.35cm) -- (1.2,-.55) arc (90:0:.45cm);
	\node at (-.4,.8) {{\scriptsize{$n+1$}}};
	\node at (-.4,-2) {{\scriptsize{$c+1$}}};
	\node at (1,-2) {{\scriptsize{$d+1$}}};
	\node at (-.5,-.6) {{\scriptsize{$c$}}};
	\node at (1.7,-.6) {{\scriptsize{$d$}}};
	\filldraw[unshaded,thick] (-.6,.6)--(.6,.6)--(.6,-.2)--(-.6,-.2)--(-.6,.6);
	\node at (0,.2) {$\jw{n+1}$};
	\filldraw[unshaded,thick] (-.6,-1)--(.6,-1)--(.6,-1.8)--(-.6,-1.8)--(-.6,-1);
	\filldraw[unshaded,thick] (.8,-1)--(2,-1)--(2,-1.8)--(.8,-1.8)--(.8,-1);
	\node at (0,-1.4) {$\jw{c+1}$};
	\node at (1.4,-1.4) {$\jw{d+1}$};
	\draw (.4,-.2) arc (180:360:.2cm) -- (.8,1);
	\draw (.4,-1) arc (180:0:.3cm);
\end{tikzpicture}
\right\rangle
=
\frac{(-1)^{b+d}[n+2]^2}{[b+1][d+1][n+1]}
$
\ee
\end{lem}
\begin{proof}
\mbox{}
\be
\item
First, if $a=c$, then the inner product is the same as that in Equation \eqref{eqn:HardTL3}. (We get the reflection of the closed diagram in Equation \eqref{eqn:HardTL1}, but since the inner product in question is real, it doesn't matter.)

Now if $a<c$, we also have $d<b$, and the inner product in question is equal to the following closed diagram:
$$
\begin{tikzpicture}[baseline = .6cm,yscale=-1]
	\draw (-.4,.6) arc (0:180:.25cm) -- (-.9,-3.2) arc (-180:0:.25cm);
	\draw (1.4,-1.8) arc (-180:0:.2cm) -- (1.8,-1.4) .. controls ++(90:1cm) and ++(0:1cm) .. (.6,1) arc (90:180:.4cm); 
	\draw (-.3,0)--(-.3,-2.4);
	\draw (.4,-3.2) .. controls ++(270:.4cm) and ++(270:.4cm) .. (1.3,-3.2)-- (1.3,-1.4);
	\draw (0,-.2) arc (180:270:.35cm) -- (.8,-.55) arc (90:0:.45cm);
	\node at (-.4,-1.4) {{\scriptsize{$a$}}};
	\node at (-1,-1.4) {{\scriptsize{$c$}}};
	\node at (1.2,-.5) {{\scriptsize{$b$}}};
	\node at (1.9,-.4) {{\scriptsize{$d$}}};
	\node at (.85,-2.2) {{\scriptsize{$c-a$}}};
	\node at (1.4,-2.8) {{\scriptsize{$1$}}};
	\filldraw[unshaded,thick] (-.6,.6)--(.6,.6)--(.6,-.2)--(-.6,-.2)--(-.6,.6);
	\node at (0,.2) {$\jw{n+1}$};
	\filldraw[unshaded,thick] (-.6,-2.4)--(.6,-2.4)--(.6,-3.2)--(-.6,-3.2)--(-.6,-2.4);
	\filldraw[unshaded,thick] (.4,-1)--(1.6,-1)--(1.6,-1.8)--(.4,-1.8)--(.4,-1);
	\node at (0,-2.8) {$\jw{c+1}$};
	\node at (1,-1.4) {$\jw{b+1}$};
	\draw (.4,-.2) arc (180:360:.2cm) -- (.8,.6) arc (0:180:.2cm);
	\draw (.6,-1) arc (0:180:.25cm) -- (.1,-2.4); 
	\draw (.5,-1.8)--(.5,-2.4);
\end{tikzpicture}.
$$
If $c>a$, there are exactly two diagrams in the top $\jw{c+1}$ which contribute to the inner product:
$$
\begin{tikzpicture}[baseline = -.1cm]
	\draw[thick] (-.6,-.4)--(-.6,.4)--(.6,.4)--(.6,-.4)--(-.6,-.4);
	\draw (-.4,-.4)--(-.4,.4);
	\draw[very thick] (-.2,.4)--(.4,-.4);
	\draw (-.2,-.4) arc (180:0:.2cm);
	\draw (0,.4) arc (-180:0:.2cm);
	\node at (-.25,0) {{\scriptsize{$a$}}};
\end{tikzpicture}
\text{ and }
\begin{tikzpicture}[baseline = -.1cm]
	\draw[thick] (-.6,-.4)--(-.6,.4)--(1,.4)--(1,-.4)--(-.6,-.4);
	\draw (-.4,-.4)--(-.4,.4);
	\draw[very thick] (.2,.4)--(.8,-.4);
	\draw (.2,-.4) arc (180:0:.2cm);
	\draw (.4,.4) arc (-180:0:.2cm);
	\node at (0,0) {{\scriptsize{$a-1$}}};
\end{tikzpicture}\,.
$$
Here, the thick line represents $c-a-1$ and $c-a$ through strings respectively. The first diagram contributes a factor of 
$$
(-1)^{c-a}\frac{[a+1][n+2][b+2]}{[b+1][c+1]},
$$
while the second diagram contributes a factor of
$$
(-1)^{c-a+1}\frac{[a][n+2]}{[c+1]}.
$$
Combined, by Lemma \ref{lem:QuantumIdentity} we get
\begin{align*}
(-1)^{c-a}\frac{[n+2]}{[c+1]}\left(\frac{[a+1][b+2]}{[b+1]}-[a]\right)
&=(-1)^{c-a}\frac{[n+2]}{[c+1]}\left(\frac{[a+1][b+2]-[a][b+1]}{[b+1]}\right)\\
&=(-1)^{c-a}\frac{[n+2]^2}{[b+1][c+1]}.
\end{align*}
Now if $c<a$, then since these inner products are real, just take the conjugate, which switches $a$ with $c$ and $b$ with $d$.
\item
The inner product in question is equal to the following closed diagram
$$
\begin{tikzpicture}[baseline = -.5cm]
	\draw (-.4,1.8)--(-.4,.2);
	\draw (0,-.2)--(-.6,-1);
	\draw (-.4,-1.4)--(-.4,-3);
	\draw (-.2,2.2) arc (180:0:.2cm) -- (.2,.2);
	\draw (-.2,-1) arc (180:0:.2cm) -- (.2,-3);
	\draw (-.4,-.2) arc (0:-180:.2cm) -- (-.8,1.8);
	\draw (-.4,-3.4) arc (0:-180:.2cm) -- (-.8,-1.4);
	\draw (-.6,2.2) .. controls ++(90:1cm) and ++(90:1cm) .. (.8,1.4) -- (.8,-3) .. controls ++(270:1cm) and ++(270:.4cm) .. (0,-3.4);
	\node at (-.6,-.7) {{\scriptsize{$d$}}};
	\node at (.9,-.6) {{\scriptsize{$b$}}};
	\node at (-.3,1) {{\scriptsize{$d$}}};
	\node at (-.3,-2.2) {{\scriptsize{$b$}}};
	\filldraw[unshaded,thick] (-1.2,1.4)--(0,1.4)--(0,2.2)--(-1.2,2.2)--(-1.2,1.4);
	\node at (-.6,1.8) {$\jw{n+1}$};
	\filldraw[unshaded,thick] (-.6,.6)--(.6,.6)--(.6,-.2)--(-.6,-.2)--(-.6,.6);
	\node at (0,.2) {$\jw{d+1}$};
	\filldraw[unshaded,thick] (-1.2,-1)--(0,-1)--(0,-1.8)--(-1.2,-1.8)--(-1.2,-1);
	\node at (-.6,-1.4) {$\jw{n+1}$};
	\filldraw[unshaded,thick] (-.6,-2.6)--(.6,-2.6)--(.6,-3.4)--(-.6,-3.4)--(-.6,-2.6);
	\node at (0,-3) {$\jw{b+1}$};
\end{tikzpicture}\,.
$$
The only diagrams in the $\jw{b+1}$ and the $\jw{d+1}$ which contribute are 
$
\begin{tikzpicture}[baseline = -.1cm]
	\draw[thick] (-.6,-.4)--(-.6,.4)--(.7,.4)--(.7,-.4)--(-.6,-.4);
	\draw (-.4,.4)--(.2,-.4);
	\draw (-.4,-.4) arc (180:0:.2cm);
	\draw (-.2,.4) arc (-180:0:.2cm);
	\node at (.35,0) {{\scriptsize{$b-1$}}};
\end{tikzpicture}
$ 
and
$
\begin{tikzpicture}[baseline = -.1cm]
	\draw[thick] (-.6,-.4)--(-.6,.4)--(.7,.4)--(.7,-.4)--(-.6,-.4);
	\draw (-.4,.4)--(.2,-.4);
	\draw (-.4,-.4) arc (180:0:.2cm);
	\draw (-.2,.4) arc (-180:0:.2cm);
	\node at (.35,0) {{\scriptsize{$d-1$}}};
\end{tikzpicture}
$
respectively, and the entire closed diagram is equal to
$$
\frac{(-1)^{b+d}[n+2]^2}{[b+1][d+1][n+1]}.
$$
\ee
\end{proof}
}

\begin{lem}\label{lem:IPofDualTL}
Suppose $a,b\geq 0$ with $a+b=n$. Let $D_a,D_a^*,\widehat{1_{n+2}}$ be the Temperley-Lieb dual basis elements as follows: 
$$
D_a
=
\begin{tikzpicture}[baseline = -.1cm]
	\draw[thick] (-.6,-.4)--(-.6,.4)--(.6,.4)--(.6,-.4)--(-.6,-.4);
	\draw (-.4,-.4)--(-.4,.4);
	\draw (-.2,.4)--(.4,-.4);
	\draw (-.2,-.4) arc (180:0:.2cm);
	\draw (0,.4) arc (-180:0:.2cm);
	\node at (-.25,0) {{\scriptsize{$a$}}};
	\node at (.35,0) {{\scriptsize{$b$}}};
\end{tikzpicture}
^{\widehat{\hs\hs}}
,\,
D_a^*
=
\begin{tikzpicture}[baseline = -.1cm,yscale=-1]
	\draw[thick] (-.6,-.4)--(-.6,.4)--(.6,.4)--(.6,-.4)--(-.6,-.4);
	\draw (-.4,-.4)--(-.4,.4);
	\draw (-.2,.4)--(.4,-.4);
	\draw (-.2,-.4) arc (180:0:.2cm);
	\draw (0,.4) arc (-180:0:.2cm);
	\node at (-.25,0) {{\scriptsize{$a$}}};
	\node at (.35,0) {{\scriptsize{$b$}}};
\end{tikzpicture}
^{\widehat{\hs\hs}}
,\,\text{ and }
\widehat{1_{n+2}}=
\frac{\jw{n+2}}{[n+3]}\,.
$$
\be
\item
$\D \langle C_i[P\circ Q],\widehat{1_{n+2}} \rangle = 
\begin{cases}
\Tr(PQ)[n+2]^{-1} & \text{if }i=n+2\\
0 & \text{else.}
\end{cases}
$
\item
$\D \langle C_i[P\circ Q],D_a \rangle = 
\begin{cases}
\Tr(PQ)[n+2]^{-1} & \text{if }i-1=a\\
0 & \text{if }i=n+2\\
\frac{(-1)^b[a+1]}{[n+1][n+2]}\Tr(PQ) & \text{if }i=n+3\\
0 & \text{else.}
\end{cases}
$
\item
$\D 
\langle C_i[P\circ Q],D_a^* \rangle 
= 
\langle D_a , C_{2n+4-i}[Q\circ P] \rangle
= 
\langle C_{2n+4-i}[P\circ Q] , D_a\rangle.
$
\ee
\end{lem}
\begin{proof}
\mbox{}
\be
\item
First, we have
$$\D
\langle C_i[P\circ Q],\widehat{1_{n+2}} \rangle
=\frac{1}{[n+3]}\langle C_i[P\circ Q],\jw{n+2} \rangle,
$$
which is clearly zero if $i\neq n+2$. When $i=n+2$, it is easy to see we get $\D\frac{\Tr(PQ)}{[n+2]}$.
\item
First, suppose $1\leq i\leq n+1$. Then the inner product in question is given by
$$
\frac{[a+1][b+1]}{[n+2]^2} \langle C_i[P\circ Q] , D\rangle,
$$
where $D$ is the diagram in Lemma \ref{lem:TLDualBasis}. If $i-1\neq a$, then the resulting closed diagram is clearly zero. If $i-1=a$, then we have
$$
\frac{[a+1][b+1]}{[n+2]^2} \langle C_i[P\circ Q],D \rangle
= 
\begin{tikzpicture}[baseline = .6cm]
	\draw (-.2,1.8) .. controls ++(90:1cm) and ++(90:1cm) .. (3.4,1.8);
	\draw (1.6,1.8) .. controls ++(90:.6cm) and ++(90:.6cm) .. (3,1.8);
	\draw (.4,1.8) arc (180:0:.3cm);	
	\draw (.4,1) arc (-180:0:.3cm);	
	\draw (2.6,1.8) arc (0:180:.2cm) -- (2.2,1) .. controls ++(270:.6cm) and ++(30:1cm) .. (.7,0);	
	\draw (.7,0)--(3,0);
	\draw (.7,0)-- (-.2,1);
	\draw (.7,0)-- (1.6,1);
	\draw (3,0)--(3,1);
	\filldraw[unshaded,thick] (2.4,1)--(3.6,1)--(3.6,1.8)--(2.4,1.8)--(2.4,1);
	\filldraw[unshaded,thick] (-.6,1)--(.6,1)--(.6,1.8)--(-.6,1.8)--(-.6,1);
	\filldraw[unshaded,thick] (.8,1)--(2,1)--(2,1.8)--(.8,1.8)--(.8,1);
	\filldraw[unshaded,thick] (.7,0) circle (.4cm);	
	\filldraw[unshaded,thick] (3,0) circle (.4cm);	
	\node at (0,1.4) {$\jw{a+1}$};
	\node at (1.4,1.4) {$\jw{b+1}$};
	\node at (3,1.4) {$\jw{n+1}$};
	\node at (.7,0) {$P$};
	\node at (3,0) {$Q$};
	\node at (1.85,-.15) {{\scriptsize{$n-1$}}};
	\node at (1.5,2.7) {{\scriptsize{$a$}}};
	\node at (1.6,2.2) {{\scriptsize{$b$}}};
	\node at (-.1,.7) {{\scriptsize{$a$}}};
	\node at (1.5,.7) {{\scriptsize{$b$}}};
	\node at (3.4,.7) {{\scriptsize{$n+1$}}};
	\node at (.7,-.55) {$\star$};
	\node at (3,-.55) {$\star$};
\end{tikzpicture}
$$
and the only terms in the $\jw{a+1}$ which contribute to the value are $1_{a+1}$ and 
$
\begin{tikzpicture}[baseline = -.1cm]
	\draw[thick] (-.6,-.4)--(-.6,.4)--(.6,.4)--(.6,-.4)--(-.6,-.4);
	\draw (0,.4) arc (-180:0:.2cm);
	\draw (-.4,-.4)--(-.4,.4);
	\draw (0,-.4) arc (180:0:.2cm);
	\node at (0,0) {{\scriptsize{$a-1$}}};
\end{tikzpicture}
$\,.
This yields, using Lemma \ref{lem:QuantumIdentity} and the formulas in Fact \ref{fact:2cup}, 
$$
\frac{[a+1][b+1]}{[n+2]^2}\left(\frac{[b+2]}{[b+1]}-\frac{[a]}{[a+1]}\right)\Tr(PQ) = \frac{\Tr(PQ)}{[n+2]}.
$$

Second, if $i=n+2$, then both diagrams in the formula for $D_a$ from Lemma \ref{lem:TLDualBasis} contribute to the inner product, and we have
\begin{align*}
\langle C_{n+2}&[P\circ Q] , D_a\rangle\\
& = \frac{[a+1][b+1]}{[n+2]^2} \langle C_{n+2}[P\circ Q] , D\rangle - \frac{(-1)^b[a+1]}{[n+2][n+3]} \langle C_{n+2}[P\circ Q] , \jw{n+2}\rangle\\
& =  \frac{[a+1][b+1]}{[n+2]^2} \langle C_{n+2}[P\circ Q] , D\rangle - \frac{(-1)^b[a+1]}{[n+2]^2}\Tr(PQ)
\end{align*}
by part (1) of this lemma. Now by drawing diagrams, we get 
$$
\langle C_{n+2}[P\circ Q] , D\rangle
=
\begin{tikzpicture}[baseline = .6cm]
	\draw (-.2,1.8) .. controls ++(90:1cm) and ++(90:1cm) .. (3.4,1.8);
	\draw (1.6,1.8) .. controls ++(90:.6cm) and ++(90:.6cm) .. (3,1.8);
	\draw (.4,1.8) arc (180:0:.3cm);	
	\draw (2.6,1.8) arc (0:180:.2cm) -- (2.2,1) arc (0:-180:.2cm);	
	\draw (.7,0)--(3,0);
	\draw (.7,0)-- (-.2,1);
	\draw (.7,0)-- (1.2,1);
	\draw (3,0)--(3,1);
	\filldraw[unshaded,thick] (2.4,1)--(3.6,1)--(3.6,1.8)--(2.4,1.8)--(2.4,1);
	\filldraw[unshaded,thick] (-.6,1)--(.6,1)--(.6,1.8)--(-.6,1.8)--(-.6,1);
	\filldraw[unshaded,thick] (.8,1)--(2,1)--(2,1.8)--(.8,1.8)--(.8,1);
	\filldraw[unshaded,thick] (.7,0) circle (.4cm);	
	\filldraw[unshaded,thick] (3,0) circle (.4cm);	
	\node at (0,1.4) {$\jw{a+1}$};
	\node at (1.4,1.4) {$\jw{b+1}$};
	\node at (3,1.4) {$\jw{n+1}$};
	\node at (.7,0) {$P$};
	\node at (3,0) {$Q$};
	\node at (1.85,-.15) {{\scriptsize{$n-1$}}};
	\node at (1.5,2.7) {{\scriptsize{$a$}}};
	\node at (1.6,2.2) {{\scriptsize{$b$}}};
	\node at (-.3,.6) {{\scriptsize{$a+1$}}};
	\node at (1.2,.6) {{\scriptsize{$b$}}};
	\node at (3.4,.7) {{\scriptsize{$n+1$}}};
	\node at (.7,-.55) {$\star$};
	\node at (3,-.55) {$\star$};
\end{tikzpicture}.
$$
The only diagram in $\jw{b+1}$ which contributes is
$
\begin{tikzpicture}[baseline = -.1cm,xscale=-1]
	\draw[thick] (-.6,-.4)--(-.6,.4)--(.7,.4)--(.7,-.4)--(-.6,-.4);
	\draw (-.4,.4)--(.2,-.4);
	\draw (-.4,-.4) arc (180:0:.2cm);
	\draw (-.2,.4) arc (-180:0:.2cm);
	\node at (.35,0) {{\scriptsize{$b-1$}}};
\end{tikzpicture}
$\,, which yields
$$
\frac{(-1)^b}{[b+1]}\Tr(PQ).
$$
The inner product in question is thus zero.

Third, if $i=n+3$, then as in the case $1\leq i\leq n+1$, we have
$$
\frac{[a+1][b+1]}{[n+2]^2} \langle C_{n+3}[P\circ Q] , D\rangle
=
\begin{tikzpicture}[baseline = .6cm]
	\draw (-.2,1.8) .. controls ++(90:1cm) and ++(90:1cm) .. (3.4,1.8);
	\draw (1.6,1.8) .. controls ++(90:.6cm) and ++(90:.6cm) .. (3,1.8);
	\draw (.4,1.8) arc (180:0:.3cm);	
	\draw (2.6,1.8) arc (0:180:.2cm) -- (2.2,1) arc (-180:0:.2cm);	
	\draw (.7,0)--(3,0);
	\draw (.7,0)-- (-.2,1);
	\draw (.7,0)-- (1.2,1);
	\draw (3,0)--(3,1);
	\draw (3,0)--(1.6,1);
	\filldraw[unshaded,thick] (2.4,1)--(3.6,1)--(3.6,1.8)--(2.4,1.8)--(2.4,1);
	\filldraw[unshaded,thick] (-.6,1)--(.6,1)--(.6,1.8)--(-.6,1.8)--(-.6,1);
	\filldraw[unshaded,thick] (.8,1)--(2,1)--(2,1.8)--(.8,1.8)--(.8,1);
	\filldraw[unshaded,thick] (.7,0) circle (.4cm);	
	\filldraw[unshaded,thick] (3,0) circle (.4cm);	
	\node at (0,1.4) {$\jw{a+1}$};
	\node at (1.4,1.4) {$\jw{b+1}$};
	\node at (3,1.4) {$\jw{n+1}$};
	\node at (.7,0) {$P$};
	\node at (3,0) {$Q$};
	\node at (1.85,-.15) {{\scriptsize{$n-1$}}};
	\node at (1.5,2.7) {{\scriptsize{$a$}}};
	\node at (1.6,2.2) {{\scriptsize{$b$}}};
	\node at (-.3,.6) {{\scriptsize{$a+1$}}};
	\node at (1.2,.6) {{\scriptsize{$b$}}};
	\node at (3.2,.7) {{\scriptsize{$n$}}};
	\node at (.7,-.55) {$\star$};
	\node at (3,-.55) {$\star$};
\end{tikzpicture}.
$$
Again, the only diagram in $\jw{b+1}$ which contributes is
$
\begin{tikzpicture}[baseline = -.1cm,xscale=-1]
	\draw[thick] (-.6,-.4)--(-.6,.4)--(.7,.4)--(.7,-.4)--(-.6,-.4);
	\draw (-.4,.4)--(.2,-.4);
	\draw (-.4,-.4) arc (180:0:.2cm);
	\draw (-.2,.4) arc (-180:0:.2cm);
	\node at (.35,0) {{\scriptsize{$b-1$}}};
\end{tikzpicture}
$\,, which yields
$$
\frac{[a+1][b+1]}{[n+1][n+2]}\left(\frac{(-1)^b}{[b+1]}\right)=\frac{(-1)^b[a+1]}{[n+1][n+2]}.
$$

Finally, if $i>n+3$, the result is once again zero, since both diagrams in the formula for $D_a$ from Lemma \ref{lem:TLDualBasis} have zero inner product with $C_i[P\circ Q]$.

\item
The first equality follows since both sides give the same closed diagram. Note that the quantity in the middle is equal to its conjugate by part (2) of this lemma. The second equality now follows since $\Tr(QP)=\Tr(PQ)$.
\ee
\end{proof}

\subsection{Projections to Temperley-Lieb}\label{sec:ProjectToTL}

The first lemma below is similar to \cite[Proposition 4.5.2]{MR2972458}.
\begin{lem}\label{lem:WenzlIdentity}
\mbox{}
\be
\item
If $k=0,\dots, 2n$, then
$\D P_{\TL_{k,+}}\left(
\begin{tikzpicture}[baseline = -.7cm]
	\draw (0,.8)--(0,-2);
	\node at (-.4,.6) {{\scriptsize{$k$}}};
	\node at (-.4,-1.8) {{\scriptsize{$k$}}};
	\node at (-.5,-.6) {{\scriptsize{$2n-k$}}};
	\draw[thick, unshaded] (0,0) circle (.4);
	\node at (0,0) {$Q$};
	\node at (-.55,0) {$\star$};
	\draw[thick, unshaded] (0,-1.2) circle (.4);
	\node at (0,-1.2) {$P$};
	\node at (-.55,-1.2) {$\star$};
\end{tikzpicture}
\right)=\frac{\Tr(PQ)}{[k+1]}\jw{k}$. 
\item
If $k=0,\dots,n-1$, then 
$\D
\begin{tikzpicture}[baseline = -.7cm]
	\draw (0,.8)--(0,-2);
	\node at (-.4,.6) {{\scriptsize{$k$}}};
	\node at (-.4,-1.8) {{\scriptsize{$k$}}};
	\node at (-.5,-.6) {{\scriptsize{$2n-k$}}};
	\draw[thick, unshaded] (0,0) circle (.4);
	\node at (0,0) {$Q$};
	\node at (-.55,0) {$\star$};
	\draw[thick, unshaded] (0,-1.2) circle (.4);
	\node at (0,-1.2) {$P$};
	\node at (-.55,-1.2) {$\star$};
\end{tikzpicture}
=\frac{\Tr(PQ)}{[k+1]}\jw{k} 
$.
\ee
\end{lem}
\begin{proof}
For (1), notice that adding a cap to the top or bottom of 
$$
\begin{tikzpicture}[baseline = -.7cm]
	\draw (0,.8)--(0,-2);
	\node at (-.4,.6) {{\scriptsize{$k$}}};
	\node at (-.4,-1.8) {{\scriptsize{$k$}}};
	\node at (-.5,-.6) {{\scriptsize{$2n-k$}}};
	\draw[thick, unshaded] (0,0) circle (.4);
	\node at (0,0) {$Q$};
	\node at (-.55,0) {$\star$};
	\draw[thick, unshaded] (0,-1.2) circle (.4);
	\node at (0,-1.2) {$P$};
	\node at (-.55,-1.2) {$\star$};
\end{tikzpicture}
$$
gives zero, so its projection to $\TL_{k,+}$ must be a constant times $\jw{k}$. Taking traces gives the constant.

For (2), notice that the diagram is already in Temperley-Lieb since $\fB\cup\{\jw{n}\}$ spans an algebra.
\end{proof}

\begin{prop}\label{prop:ProjectToTL}
\mbox{}
\be
\item
$\D P_{\TL_{n+2,+}}\left(P\traincirc{n-2} Q\right) = \frac{\Tr(PQ)}{[n+3]}\jw{n+2}$
\item
$\D P_{\TL_{n+2,+}}\left(P\traincirc{n-1} Q \traincirc{n-1} R\right) 
=
a_R^{PQ} \left(
\frac{[n+1]}{[n+2]^2}
\begin{tikzpicture}[baseline = -.7cm]
	\clip (-.9,-2)--(-.9,1)--(.9,1)--(.9,-2);
	\draw (0,1)--(0,-2);
	\node at (-.4,.8) {{\scriptsize{$n+1$}}};
	\node at (-.4,-1.8) {{\scriptsize{$n+1$}}};
	\node at (-.2,-.5) {{\scriptsize{$n$}}};
	\filldraw[unshaded,thick] (-.6,.6)--(.6,.6)--(.6,-.2)--(-.6,-.2)--(-.6,.6);
	\node at (0,.2) {$\jw{n+1}$};
	\filldraw[unshaded,thick] (-.6,-.8)--(.6,-.8)--(.6,-1.6)--(-.6,-1.6)--(-.6,-.8);
	\node at (0,-1.2) {$\jw{n+1}$};
	\draw (.4,-.2) arc (180:360:.2cm) -- (.8,1);
	\draw (.4,-.8) arc (180:0:.2cm) -- (.8,-2);
\end{tikzpicture}
-
\frac{[n+1]}{[n+2][n+3]}  
\begin{tikzpicture}[baseline = -.1cm]
	\draw (0,.8)--(0,-.8);
	\node at (-.4,.6) {{\scriptsize{$n+2$}}};
	\node at (-.4,-.6) {{\scriptsize{$n+2$}}};
	\filldraw[unshaded,thick] (-.6,.4)--(.6,.4)--(.6,-.4)--(-.6,-.4)--(-.6,.4);
	\node at (0,0) {$\jw{n+2}$};
\end{tikzpicture}
\right)$.
\ee
\end{prop}
\begin{proof}
(1) is immediate from Lemma \ref{lem:WenzlIdentity}. 
For (2), for $T$ a diagrammatic basis element of $\TL_{n+2,+}$, it is clear that
$$
\left\langle T, P_{\TL_{n+2,+}}(P\traincirc{n-1} Q \traincirc{n-1} R)\right\rangle=
\begin{cases}
a_R^{PQ} & \text{if }T=E_{n+1}=
\begin{tikzpicture}[baseline = -.1cm]
	\draw[thick] (-.4,-.4)--(-.4,.4)--(.6,.4)--(.6,-.4)--(-.4,-.4);
	\draw (0,.4) arc (-180:0:.2cm);
	\draw (-.2,-.4)--(-.2,.4);
	\draw (0,-.4) arc (180:0:.2cm);
	\node at (-.05,0) {{\scriptsize{$n$}}};
\end{tikzpicture}
\\
0 & \text{else.}
\end{cases}
$$ 
Hence $P_{\TL_{n+2,+}}\left(P\traincirc{n-1} Q \traincirc{n-1} R\right)=a_R^{PQ} \widehat{E_{n+1}}$, where $\widehat{E_{n+1}}$ is the dual basis element of $E_{n+1}$ in $\TL_{n+2,+}$. The result now follows by Lemma \ref{lem:TLDualBasis} (using $b=0$, $a=n$. In particular $[b]=0$ and $[b+1]=1$).
\end{proof}

\begin{prop}\label{prop:ProjectToDualTL}
$P_{\TL_{n+2,+}}\left(C_i[P\traincirc{n-1} Q]\right)=\Tr(PQ)X$
where $X$ is a linear combination of Temperley-Lieb dual basis elements $D_a,D_a^*,\widehat{1_{n+2}}$ (as in Lemma \ref{lem:IPofDualTL}).  The exact linear combination is given in the table below.
$$
\begin{array}{|c|c|}
\hline
\text{i} & {X}
\\
\hline\hline
1 & [2]D_{0}+D_1
\\
\hline
1< i< n+1 & D_{i-2}+[2]D_{i-1} + D_{i}
\\
\hline
n+1 & D_{n-1}+[2]D_{n} + \widehat{1_{n+2}}
\\
\hline
n+2 & D_n+[2] \widehat{1_{n+2}}
\\
\hline
n+3 &\widehat{1_{n+2}}+[2]D_{n}^* + D_{n-1}^*
\\
\hline
n+3< i< 2n+3 & D_{2n+2-i}^*+[2]D_{2n+3-i}^* + D_{2n+4-i}^*
\\
\hline
2n+3 & [2]D_{0}^*+D_1^*
\\
\hline
\end{array}
$$
\end{prop}
\begin{proof}
The only diagrammatic basis elements $T$ in Temperley-Lieb which pair nontrivially with $C_i[P\traincirc{n-1}Q]$ are those whose dual basis elements $\widehat{T}$ appear in the linear combination. The coefficients are given by $\langle T, C_i[P\traincirc{n-1}Q]\rangle$.
\end{proof}

\subsection{Projections to annular consequences}

\begin{defn}
Let $\gA_{n+2}$ denote the space of second annular consequences of $\fB$ in $\cP_{n+2,+}$.
\end{defn}

The proofs of the following propositions are parallel to the proof of \cite[Proposition 4.4.1]{MR2972458}. The inner products are only non-zero for the given annular consequences, and they are easily worked out by drawing pictures and using Lemma \ref{lem:WenzlIdentity}.

\begin{prop}\label{prop:ProjectTrainsToAC}
\mbox{}
\be
\item
$\D
P_{\gA_{n+2}}\left(P\underset{n-2}{\circ} Q\right)
=
\sum_{R\in\fB}
a^{PQ}_R
\omega_P\omega_Q^{-1}
\widehat{\cup}_{-1,-1}(R)
+
a^{PQ}_R
\sigma_R^n 
\widehat{\cup}_{-1,n+1}(R)
+
b^{PQ}_R
\sigma_P\sigma_Q^{-1}
\widehat{\cup}_{n,0}(R)
$
where the coefficients of the $\widehat{\cup}_{i,j}(R)$ are given by $\langle \cup_{i,j}(R), P\traincirc{n-2} Q\rangle$.

\item
\begin{align*}
P_{\gA_{n+2}}&\left(P\traincirc{n-1} Q\traincirc{n-1} R\right)
=
\\
&\sum_{S\in\fB}
\Delta_{n-1,2}(P,Q,R\mid S)
\widehat{\cup}_{-1,-1}(S)
+
\frac{\sigma_S^{n+1}}{[n]}
\Tr(SP)\Tr(QR)
\widehat{\cup}_{-1,n}(S)\\
&+
\frac{\sigma_S^{n-1}}{[n]}
\Tr(PQ)\Tr(RS)
\widehat{\cup}_{-1,n+2}(S)
+
\sigma_S^n\Tr(PQRS)
\widehat{\cup}_{0,n+1}(S)\\
&+
\Delta_{n-1,1}(P,Q,R\mid S)
\widehat{\cup}_{n-1,0}(S)
+
\sigma_S^{n-1}\Delta_{n,1}(P,Q,R\mid S)
\widehat{\cup}_{n-1,n+3}(S)
\end{align*}
where the coefficients of the $\widehat{\cup}_{i,j}(S)$ are given by $\langle \cup_{i,j}(S),P\traincirc{n-1} Q\traincirc{n-1}R\rangle$. 

Note that in the above formula, the quartic moment and two of the three tetrahedral constants were computed in terms of the moments and chiralities of $\fB$ in Remark \ref{rem:SpanAlgebras} and Example \ref{ex:ReduceTetrahedral}.
\ee
\end{prop}

\begin{prop}\label{prop:ProjectACofTrainsToAC}
$P_{\gA_{n+2}}\left(C_i[P\traincirc{n-1}Q]\right)=\sum_{R\in \fB} X_R$ where $X_R$ is given below:
$$
\begin{array}{|c|c|}
\hline
i & X_R\in \gA_{n+2}
\\\hline\hline
1 & 
\beta \widehat{\cup}_{0,2n+2}
+
\alpha \widehat{\cup}_{n-1,n+1}
+
[2]\beta \widehat{\cup}_{-1,2n+2}
+
[2]\alpha \widehat{\cup}_{n,n+1}
+
\sigma_R^{-1}a^{PQ}_R \widehat{\cup}_{n-1,0}
\\\hline
2 &
\sigma_R\beta\widehat{\cup}_{1,2n+1}
+
\alpha \widehat{\cup}_{n-2,n+1}
+
[2]\beta\widehat{\cup}_{0,2n+2}
\\
&
+
[2]\alpha \widehat{\cup}_{n-1,n+1}
+
\beta\widehat{\cup}_{-1,2n+2}
+
\alpha \widehat{\cup}_{n,n+1}
+
\sigma_R^{-1}\beta \widehat{\cup}_{-1,-1}
\\\hline
2< i < n+1 &
\sigma_R^{i-1}\beta \widehat{\cup}_{i-1,2n-i+3}
+
\alpha  \widehat{\cup}_{n-i,n+1}
+
[2]\sigma_R^{i-2}\beta \widehat{\cup}_{i-2,2n-i+4}
\\
&
+
[2]\alpha\widehat{\cup}_{n-i+1,n+1}
+
\sigma_R^{i-3} \beta \widehat{\cup}_{i-3,2n-i+5}
+
\alpha \widehat{\cup}_{n-i+2,n+1}
\\\hline
n+1 &
\beta \widehat{\cup}_{n,0}
+
\alpha \widehat{\cup}_{-1,n+1}
+
[2]\sigma_R^{n-1}\beta \widehat{\cup}_{n-1,n+3}
\\
&
+
[2]\alpha \widehat{\cup}_{0,n+1}
+
\sigma_R^{n-2}\beta \widehat{\cup}_{n-2,n+4}
+
\alpha \widehat{\cup}_{1,n+1}
+
\sigma_R^{n+1}a^{PQ}_R \widehat{\cup}_{-1,n}
\\\hline
n+2 &
\beta \widehat{\cup}_{n-1,0}
+
\alpha \widehat{\cup}_{0,n+1}
+
[2]\beta \widehat{\cup}_{n,0}
+
[2]\alpha \widehat{\cup}_{-1,n+1}
+
\sigma_R^{n-1}\beta \widehat{\cup}_{n-1,n+3}
\\\hline
n+3 &
\beta \widehat{\cup}_{n-2,0}
+
\sigma_R^{n+1}a^{PQ}_R \widehat{\cup}_{1,n}
+
[2]\beta \widehat{\cup}_{n-1,0}
\\
&
+
[2]\alpha \widehat{\cup}_{0,n+1}
+
\beta \widehat{\cup}_{n,0}
+
\alpha \widehat{\cup}_{-1,n+1}
+
\sigma_R^{n-1}a^{PQ}_R \widehat{\cup}_{-1,n+2}
\\\hline
n+3<i<2n+2 &
\beta \widehat{\cup}_{2n+1-i,0}
+
\sigma_R^{i-2}a^{PQ}_R\widehat{\cup}_{i-n-2,2n+3-i}
+
[2]\beta \widehat{\cup}_{2n+2-i,0}
\\
&
+
[2]\sigma_R^{i-3}a^{PQ}_R\widehat{\cup}_{i-n-3,2n+4-i}
+
\beta \widehat{\cup}_{2n+3-i,0}
+
\sigma_R^{i-4}a^{PQ}_R\widehat{\cup}_{i-n-4,2n+5-i}
\\\hline
2n+2 &
\beta \widehat{\cup}_{-1,0}
+
a^{PQ}_R \widehat{\cup}_{n,1}
+
[2]\beta \widehat{\cup}_{0,0}
\\
&
+
[2]\sigma_R^{-1}a^{PQ}_R \widehat{\cup}_{n-1,2}
+
\beta \widehat{\cup}_{1,0}
+
\sigma_R^{-2}a^{PQ}_R \widehat{\cup}_{n-2,3}
+
\sigma_R\beta \widehat{\cup}_{-1,-1}
\\\hline
2n+3 &
\alpha \widehat{\cup}_{n-1,n+3}
+
[2]\beta \widehat{\cup}_{-1,0}
+
[2]a^{PQ}_R \widehat{\cup}_{n,1}
+
\beta \widehat{\cup}_{0,0}
+
\sigma_R^{-1}a^{PQ}_R \widehat{\cup}_{n-1,2}
\\
\hline
\end{array}
$$
where $\alpha=\sigma_R^na^{PQ}_R$, $\beta=\sigma_Q^{-1}\sigma_P b^{PQ}_R$, and $\widehat{\cup}_{i,j}=\widehat{\cup}_{i,j}(R)$.
\end{prop}

\begin{rem}\label{rem:CheckInnerProducts}
We check the formulas given in Propositions \ref{prop:ProjectTrainsToAC} and \ref{prop:ProjectACofTrainsToAC} by taking inner products directly in the graph planar algebra. 
See Subsection \ref{sec:Checking} for more details.
Note that $3^{\Z/3}$ (Haagerup) has $a^{AA}_A=0\neq b^{AA}_A$ (see Appendix \ref{moments:333Haagerup}), and $3^{\Z/2\times \Z/2}$ has $a^{AA}_A\neq 0=b^{AA}_A$ (see Appendix \ref{moments:3333}) which helps us isolate particular constants above. 

However, the best evidence that these formulas are correct is the fact that we can actually compute the two-strand jellyfish relations for our subfactor planar algebras!
\end{rem}

\subsection{Inner products amongst trains and their projections}

\begin{prop}\label{prop:InnerProduct}
\mbox{}
\be
\item
$\D\left\langle P\traincirc{n-2} Q,R\traincirc{n-2} S\right\rangle =\D \frac{\Tr(PR)\Tr(SQ)}{[n-1]}$,
\item
$\D\left\langle P\traincirc{n-1} Q\traincirc{n-1} R,P'\traincirc{n-1} Q'\traincirc{n-1} R'\right\rangle =\D \frac{\Tr(PP')\Tr(QQ')\Tr(RR')}{[n]^2}$,
\item
$\D\left\langle P\traincirc{n-1} Q\traincirc{n-1} R,S\traincirc{n-2} T\right\rangle = 0$,
\ee
\end{prop}
\begin{proof}
For (1), the left hand side equals
$$
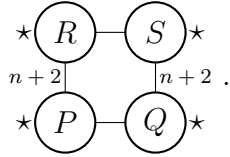
\begin{tikzpicture}[baseline = -.7cm]
	\draw (0,0)--(1.2,0) -- (1.2,-1.2)--(0,-1.2)--(0,0);
	\node at (-.4,-.6) {{\scriptsize{$n+2$}}};
	\node at (1.6,-.6) {{\scriptsize{$n+2$}}};
	\draw[thick, unshaded] (0,0) circle (.4);
	\node at (0,0) {$R$};
	\node at (-.55,0) {$\star$};
	\draw[thick, unshaded] (0,-1.2) circle (.4);
	\node at (0,-1.2) {$P$};
	\node at (-.55,-1.2) {$\star$};
	\draw[thick, unshaded] (1.2,0) circle (.4);
	\node at (1.2,0) {$S$};
	\node at (1.75,0) {$\star$};
	\draw[thick, unshaded] (1.2,-1.2) circle (.4);
	\node at (1.2,-1.2) {$Q$};
	\node at (1.75,-1.2) {$\star$};
\end{tikzpicture}.
$$
The result now follows by (2) of Lemma \ref{lem:WenzlIdentity}. 
We omit the proof of (2), which is similar to the proof of (1). 
For (3), again using (2) of Lemma \ref{lem:WenzlIdentity}, we see that the left hand side is equal to
$$
\begin{tikzpicture}[baseline = .5cm, yscale=-1]
	\draw (0,0)--(3.2,0);
	\draw (2.4,-.8)--(.8,-.8);
	\draw (0,0)--(.8,-.8)--(1.6,0)--(2.4,-.8)--(3.2,0);
	\node at (.8,.15) {{\scriptsize{$n-1$}}};
	\node at (2.4,.15) {{\scriptsize{$n-1$}}};
	\node at (0,-.6) {{\scriptsize{$n+1$}}};
	\node at (3.2,-.6) {{\scriptsize{$n+1$}}};
	\node at (1.6,-1) {{\scriptsize{$n-2$}}};
	\draw[thick, unshaded] (0,0) circle (.4);
	\node at (0,0) {$P$};
	\node at (0,.55) {$\star$};
	\draw[thick, unshaded] (1.6,0) circle (.4);
	\node at (1.6,0) {$Q$};
	\node at (1.6,.55) {$\star$};
	\draw[thick, unshaded] (3.2,0) circle (.4);
	\node at (3.2,0) {$R$};
	\node at (3.2,.55) {$\star$};
	\draw[thick, unshaded] (.8,-.8) circle (.4);
	\node at (.8,-.8) {$S$};
	\node at (.8,-1.35) {$\star$};
	\draw[thick, unshaded] (2.4,-.8) circle (.4);
	\node at (2.4,-.8) {$T$};
	\node at (2.4,-1.35) {$\star$};
\end{tikzpicture}
=
\frac{\Tr(PS)\Tr(RT)}{[n]^2}
\begin{tikzpicture}[baseline = .5cm,yscale=-1]
	\draw (2.4,-.8)--(.8,-.8);
	\draw (.8,-.8)--(1.6,0)--(2.4,-.8);
	\node at (1.1,-.3) {{\scriptsize{$n$}}};
	\node at (2.1,-.3) {{\scriptsize{$n$}}};
	\node at (1.6,-1) {{\scriptsize{$n-2$}}};
	\draw[thick, unshaded] (1.6,0) circle (.4);
	\node at (1.6,0) {$Q$};
	\node at (1.6,.55) {$\star$};
	\draw[thick, unshaded] (.8,-.8) circle (.4);
	\node at (.8,-.8) {$f$};
	\node at (.8,-1.35) {$\star$};
	\draw[thick, unshaded] (2.4,-.8) circle (.4);
	\node at (2.4,-.8) {$f$};
	\node at (2.4,-1.35) {$\star$};
\end{tikzpicture}
$$
where $f=\jw{n-1}$. The right hand side of the above equation is zero, since it is a linear combination of closed diagrams containing only one generator.
\end{proof}

\begin{prop}\label{prop:InnerProduct2}
\mbox{}
\be
\item
\begin{multline*}
\left\langle C_i[P\traincirc{n-1} Q], C_j[R\traincirc{n-1} S],\right\rangle\\
=
\begin{cases}
\Tr(PR)\Tr(SQ)[2][n]^{-1} &\text{if }i=j\\
\Tr(PR)\Tr(SQ)[n]^{-1} & \text{if }|i-j|=1\\
\Tr(PRSQ) & \text{if }(i,j)\in\{(n{+}1,n{+}3),(n{+}3,n{+}1)\}\\
0 & \text{else,}
\end{cases}
\end{multline*}
\item
$\D \left\langle C_i[P\traincirc{n-1} Q], R\traincirc{n-2} S\right\rangle=
\begin{cases}
\D \Tr(PR)\Tr(SQ)[n]^{-1} & \text{if }i=n+2\\
0 & \text{else,}
\end{cases}
$
\item
$\D \left\langle C_i[P\traincirc{n-1} Q], R\traincirc{n-1} S\traincirc{n-1}T\right\rangle=
\begin{cases}
a^{PR}_S\Tr(QT)[n]^{-1} & \text{if }i=n+1\\
a^{ST}_Q\Tr(RP)[n]^{-1} & \text{if }i=n+3\\
0 & \text{else.}
\end{cases}
$
\ee
\end{prop}
\begin{proof}
The proofs are all relatively straightforward drawing the necessary diagrams. 
The case in part (1) which is easiest to miss is when $(i,j)\in \{(n+1,n+3),(n+3,n+1)\}$. 
In this case we get the following diagrams:
$$
\begin{tikzpicture}[baseline = -.7cm]
	\draw[dashed] (-.2,-.6)--(1.8,-.6);
	\draw (0,0)--(1.6,0) -- (1.6,-1.2)--(0,-1.2)--(0,0);
	\draw  (0,0) .. controls ++(-80:1.6cm) and ++(200:1cm) .. (1.6,0);
	\draw  (0,-1.2) .. controls ++(20:1cm) and ++(100:1.6cm) .. (1.6,-1.2);
	\node at (-.4,-.6) {{\scriptsize{$n$}}};
	\node at (2,-.6) {{\scriptsize{$n$}}};
	\node at (.8,.2) {{\scriptsize{$n-1$}}};
	\node at (.8,-1.4) {{\scriptsize{$n-1$}}};
	\draw[thick, unshaded] (0,0) circle (.4);
	\node at (0,0) {$R$};
	\node at (-.55,0) {$\star$};
	\draw[thick, unshaded] (0,-1.2) circle (.4);
	\node at (0,-1.2) {$P$};
	\node at (-.55,-1.2) {$\star$};
	\draw[thick, unshaded] (1.6,0) circle (.4);
	\node at (1.6,0) {$S$};
	\node at (2.15,0) {$\star$};
	\draw[thick, unshaded] (1.6,-1.2) circle (.4);
	\node at (1.6,-1.2) {$Q$};
	\node at (2.15,-1.2) {$\star$};
\end{tikzpicture}
=
\begin{tikzpicture}[baseline = -.7cm]
	\draw[dashed] (-.2,-.6)--(1.8,-.6);
	\draw (0,0)--(1.6,0) -- (1.6,-1.2)--(0,-1.2)--(0,0);
	\draw  (0,0) .. controls ++(-20:1cm) and ++(-100:1.6cm) .. (1.6,0);
	\draw  (0,-1.2) .. controls ++(80:1.6cm) and ++(160:1cm) .. (1.6,-1.2);
	\node at (-.4,-.6) {{\scriptsize{$n$}}};
	\node at (2,-.6) {{\scriptsize{$n$}}};
	\node at (.8,.2) {{\scriptsize{$n-1$}}};
	\node at (.8,-1.4) {{\scriptsize{$n-1$}}};
	\draw[thick, unshaded] (0,0) circle (.4);
	\node at (0,0) {$R$};
	\node at (-.55,0) {$\star$};
	\draw[thick, unshaded] (0,-1.2) circle (.4);
	\node at (0,-1.2) {$P$};
	\node at (-.55,-1.2) {$\star$};
	\draw[thick, unshaded] (1.6,0) circle (.4);
	\node at (1.6,0) {$S$};
	\node at (2.15,0) {$\star$};
	\draw[thick, unshaded] (1.6,-1.2) circle (.4);
	\node at (1.6,-1.2) {$Q$};
	\node at (2.15,-1.2) {$\star$};
\end{tikzpicture}
=
\Tr(PRSQ).
$$
\end{proof}

\begin{prop}\label{prop:InnerProductInTL}
\mbox{}
\be
\item
$\D \left\langle P\traincirc{n-2} Q ,  P_{\TL_{n+2,+}}\left(R\traincirc{n-2} S\right)\right\rangle =\frac{\Tr(PQ)\Tr(RS)}{[n+3]}$
\item
$\D \left\langle P\traincirc{n-1} Q \traincirc{n-1} R ,  P_{\TL_{n+2,+}}\left(P'\traincirc{n-1} Q' \traincirc{n-1} R'\right)\right\rangle = a_R^{QP}a_{R'}^{P'Q'}\frac{[2][n+1]}{[n+2][n+3]}$
\item
$\D\left\langle P\traincirc{n-1} Q \traincirc{n-1} R  , P_{\TL_{n+2,+}}\left(S\traincirc{n-2} T\right)\right\rangle=-\frac{\Tr(ST)a_R^{QP}[n+1]}{[n+2][n+3]}$
\ee
\end{prop}
\begin{proof}
This follows quickly from Proposition \ref{prop:ProjectToTL}. For part (2), using Proposition \ref{prop:ProjectToTL}, the inner product in question is equal to 
\begin{align*}
a_R^{QP}a_{R'}^{P'Q'}\left(\frac{[n+1]^2}{[n+2]^2[n+3]}+\frac{[n+1]}{[n+2]^2}\right)
&=a_R^{QP}a_{R'}^{P'Q'}\frac{[n+1]}{[n+2]^2}\left(\frac{[n+1]+[n+3]}{[n+3]}\right)\\
&=a_R^{QP}a_{R'}^{P'Q'}\frac{[2][n+1]}{[n+2][n+3]}.
\end{align*}
\end{proof}

\begin{prop}\label{prop:UglyIPInTL}
\mbox{}
\be
\item
\begin{multline*}
\D \left\langle C_i[P\traincirc{n-1} Q] , P_{\TL_{n+2,+}}\left(C_j[R\traincirc{n-1} S]\right)\right\rangle
\\
=
\begin{cases}
\Tr(PQ)\Tr(RS)[2][n+2]^{-1} & \text{if }i=j\\
\Tr(PQ)\Tr(RS)[n+2]^{-1} & \text{if }|i-j|=1\\
\Tr(PQ)\Tr(RS)[n+1]^{-1} & \text{if }(i,j)\in \{(n{+}1,n{+}3), (n{+}3,n{+}1)\}\\
0& \text{else.}
\end{cases}
\end{multline*}
\item
$\D \left\langle C_i[P\traincirc{n-2} Q] ,  P_{\TL_{n+2,+}}\left(R\traincirc{n-2} S\right)\right\rangle
=
\begin{cases}
\Tr(PQ)\Tr(RS)[n+2]^{-1} & \text{if }i=n+2\\
0& \text{else.}
\end{cases}
$
\item
$\D \left\langle C_i[P\traincirc{n-2} Q] ,  P_{\TL_{n+2,+}}\left(R\traincirc{n-1} S\traincirc{n-1} T\right)\right\rangle
=
\begin{cases}
\Tr(PQ)a^{RS}_T[n+2]^{-1} & \text{if }i=n+1,n+3\\
0 & \text{else.}
\end{cases}
$
\ee
\end{prop}
\begin{proof}
\mbox{}
\be
\item
The formulas can be worked out easily from Lemma \ref{lem:IPofDualTL} and Proposition \ref{prop:ProjectToDualTL}. 
We will work out a few interesting cases. 

If $i=n+1$ and $j=n+3$, then
\begin{align*}
\left\langle C_i[P\traincirc{n-1} Q] \right., &\left. P_{\TL_{n+2,+}}\left(  C_j[R\traincirc{n-1} S]\right)\right\rangle\\
& =
\left\langle C_{n+1}[P\traincirc{n-1} Q] ,\widehat{1_{n+2}}+[2]D_n^*+D_{n-1}^*\right\rangle\Tr(RS)\\
& =
\left\langle C_{n+3}[P\traincirc{n-1} Q] ,[2]D_n+D_{n-1}\right\rangle\Tr(RS)\\
& =
\left(\frac{[2]}{[n+2]}-\frac{[n]}{[n+1][n+2]}\right)\Tr(PQ)\Tr(RS)\\
& =
\frac{\Tr(PQ)\Tr(RS)}{[n+1]}.
\end{align*}

If $i=n+1$ and $n+3<j<2n+3$, then 
\begin{align*}
\left\langle C_i[P\traincirc{n-1} Q] ,  \right.&\left. P_{\TL_{n+2,+}}\left(  C_j[R\traincirc{n-1} S]\right)\right\rangle\\
& =
\left\langle C_{n+1}[P\traincirc{n-1} Q] ,D_{2n+2-j}^*+[2]D_{2n+3-j}^*+D_{2n+4-j}^*\right\rangle\Tr(RS)\\
& =
\left\langle C_{n+3}[P\traincirc{n-1} Q] ,D_{j}+[2]D_{j+1}+D_{j+2}\right\rangle\Tr(RS)\\
& =
\frac{(-1)^{n-j}}{[n+1][n+2]}\left([j+1]-[2][j+2]+[j+3]\right)\Tr(PQ)\Tr(RS)\\
& =
0.
\end{align*}

\item
By Proposition \ref{prop:ProjectToTL}, we have
$$
\left\langle C_i[P\traincirc{n-1} Q] ,  P_{\TL_{n+2,+}}\left(R\traincirc{n-2} S\right)\right\rangle
=
\frac{\Tr(RS)}{[n+3]} \left\langle C_i[P\traincirc{n-1} Q] ,  \jw{n+2}\right\rangle,
$$
which is zero unless $i=n+2$. Now by Lemma \ref{lem:TLDualBasis} and Proposition \ref{prop:ProjectToDualTL}, the right hand side is equal to 
\begin{align*}
\frac{\Tr(RS)\Tr(PQ)}{[n+3]} \langle D_n +[2]\widehat{1_{n+2}}, \jw{n+2}\rangle 
&= 
\frac{\Tr(PQ)\Tr(RS)}{[n+3]}\left([2]-\frac{[n+1]}{[n+2]}\right)\\
&=
\frac{\Tr(PQ)\Tr(RS)}{[n+2]}.
\end{align*}

\item
By Proposition \ref{prop:ProjectToTL}, we have
$$
\left\langle C_i[P\traincirc{n-1} Q] ,  P_{\TL_{n+2,+}}\left(R\traincirc{n-2} S\traincirc{n-1} T\right)\right\rangle
=
a^{RS}_{T}\left\langle C_i[P\traincirc{n-1} Q] ,  \widehat{E_{n+1}}\right\rangle,
$$
which is clearly zero unless $n+1\leq i\leq n+3$ (use the formula for $\widehat{E_{n+1}}$).

If $i=n+1$ (and similarly for $i=n+3$), then only the first diagram in Proposition \ref{prop:ProjectToTL} (2) contributes to the inner product, and the value is given by
$$
a^{RS}_T\frac{[n+1]}{[n+2]^2}
\begin{tikzpicture}[baseline = .6cm]
	\draw (.6,1.8) .. controls ++(90:.6cm) and ++(90:.6cm) .. (3.2,1.8);
	\draw (2.6,1.8) arc (0:180:.2cm) -- (2.2,1) .. controls ++(270:.6cm) and ++(30:1cm) .. (1,0);
	\draw (1.4,1.8) arc (180:0:.2cm) -- (1.8,1) arc (0:-180:.2cm);	
	\draw (1,0)--(3,0);
	\draw (1,0)-- (1,1);
	\draw (3,0)--(3,1);
	\filldraw[unshaded,thick] (2.4,1)--(3.6,1)--(3.6,1.8)--(2.4,1.8)--(2.4,1);
	\filldraw[unshaded,thick] (.4,1)--(1.6,1)--(1.6,1.8)--(.4,1.8)--(.4,1);
	\filldraw[unshaded,thick] (1,0) circle (.4cm);	
	\filldraw[unshaded,thick] (3,0) circle (.4cm);	
	\node at (1,1.4) {$\jw{n+1}$};
	\node at (3,1.4) {$\jw{n+1}$};
	\node at (1,0) {$P$};
	\node at (3,0) {$Q$};
	\node at (1,-.55) {$\star$};
	\node at (3,-.55) {$\star$};
	\node at (2,-.15) {{\scriptsize{$n-1$}}};
	\node at (2,2.4) {{\scriptsize{$n$}}};
	\node at (3.2,.7) {{\scriptsize{$n$}}};
	\node at (.8,.7) {{\scriptsize{$n$}}};
\end{tikzpicture}
=
\frac{\Tr(PQ)a^{RS}_T}{[n+2]}.
$$
If $i=n+2$, by drawing similar diagrams, we see the inner product in question is equal to
$$
a^{RS}_T \left(\frac{[n+1]}{[n+2]^2}- \frac{[n+3][n+1]}{[n+2][n+2][n+3]}\right)\Tr(PQ)=0.
$$ 
\ee
\end{proof}

\begin{rem}\label{rem:InnerProductInA}
We now explain how to obtain the inner products
\begin{itemize}
\item
$\left\langle P\traincirc{n-2} Q, P_{\gA_{n+2}}\left(R\traincirc{n-2} S\right)\right\rangle$
\item
$\left\langle P\traincirc{n-1} Q\traincirc{n-1} R, P_{\gA_{n+2}}\left(P'\traincirc{n-1} Q'\traincirc{n-1} R'\right)\right\rangle$
\item
$\left\langle P\traincirc{n-1} Q\traincirc{n-1} R, P_{\gA_{n+2}}\left(S\traincirc{n-2} T\right)\right\rangle$
\item
$\D \left\langle C_i[P\traincirc{n-2} Q], P_{\gA_{n+2}}\left(C_j[R\traincirc{n-2} S]\right)\right\rangle$
\item
$\D \left\langle C_i[P\traincirc{n-2} Q],  P_{\gA_{n+2}}\left(R\traincirc{n-1} S\traincirc{n-1} T\right)\right\rangle$
\item
$\D \left\langle C_i[P\traincirc{n-2} Q],  P_{\gA_{n+2}}\left(R\traincirc{n-2} S\right)\right\rangle$.
\end{itemize}
First, we use the formulas for 
$$P_{\gA_{n+2}}\left(P\traincirc{n-2} Q\right),\, P_{\gA_{n+2}}\left(P\traincirc{n-1} Q\traincirc{n-1} R\right)\,\text{ and }P_{\gA_{n+2}}\left(C_i[P\traincirc{n-2} Q]\right)
$$ 
obtained in Propositions \ref{prop:ProjectTrainsToAC} and \ref{prop:ProjectACofTrainsToAC} to express each side as a linear combination of the $\widehat{\cup}_{i,j}(S)$'s. 
Next, we use the change of basis matrix discussed in Remark \ref{rem:DualAnnularBasis} to write the $\widehat{\cup}_{i,j}(S)$ on the right hand side in terms of the $\cup_{i,j}(S)$.
Finally, we expand the inner product in the usual way to obtain the answer.
\end{rem}

\section{Deriving formulas for two-strand box jellyfish relations}\label{sec:ComputingJellyfishRelations}

As in the previous sections, we continue Assumptions \ref{assume:Generators}, \ref{assume:SpanAlgebras}, and \ref{assume:Tetrahedral}.

We now go through our algorithm for determining two-strand jellyfish relations. 
We would like to follow the method of \cite[Section 3]{1208.3637}, which consisted of three parts:
\be
\item
Find the quadratic tangles in annular consequences,
\item
Find the jellyfish matrix, and
\item
Invert the jellyfish matrix.
\ee
The steps in our algorithm will be clearly marked in the following three subsections.

\subsection{Reduced trains in annular consequences}
In \cite{1208.3637}, the first step was to obtain a basis for the quadratic tangles in annular consequences. 
Since we have quadratic and cubic trains, we call this step obtaining a basis for the the reduced trains in annular consequences. 

\begin{defn}
Recall from Definition \ref{defn:ReducedTrain} that a `reduced train' is one where no generator connects to itself, and no pair are connected by more than $n-1$ strands.
Starting with our set of minimal generators $\fB$ satisfying Assumptions \ref{assume:Generators}, \ref{assume:SpanAlgebras}, and \ref{assume:Tetrahedral}, we have the reduced trains
$$
\set{C_i[P\traincirc{n-1}Q]}{P,Q\in\fB \text{ and }i=1,\dots, 2n+3}\subset \cP_{n+2,+}
$$
which are annular consequences of trains in $\cP_{n+1,+}$, and we have the reduced trains
$$
\set{
\TwoTrain{2}{P}{Q}
\,,\,
\ThreeTrain{P}{Q}{R}
}{P,Q,R\in\fB}\subset \cP_{n+2,+}
$$
which are non-zero when placing a Jones-Wenzl underneath. We let $\rt$ be the union of the above two sets.
\end{defn}

Since we hope that our generators generate a subfactor planar algebra whose principal graph is that from which we started, we hope that some linear combination of these reduced trains lie in annular consequences. 

\begin{defn}
We set 
$$
\rtac = \left(\TL_{n+2,+}\oplus \gA_{n+2}\right)\cap \spann(\rt),
$$
where $\rtac$ stands for \underline{reduced trains in annular consequences}.
\end{defn}

Step \ref{step:RTAC} of our algorithm finds a basis for $\rtac$. 
Since we are trying to derive box jellyfish relations, we are only interested in basis elements which are not sent to zero when we put a $\jw{2n+4}$ underneath. 
Thus we make the following definition.

\begin{defn}
An element of $\rtac$ is called \underline{essential} if at least one of the coefficients of the $P\traincirc{n-2}Q$'s or the $P\traincirc{n-1}Q\traincirc{n-1}R$'s does not vanish.
\end{defn}

\begin{rem}
If we've chosen $k$ generators in a graph planar algebra and are hoping that they give us a subfactor planar algebra with one spoke principal graph, we expect to have at least $k$ essential basis elements of $\rtac$ -- ie, one two-strand jellyfish relation for each generator.
\end{rem}

\begin{step}[A basis for $\rtac$]\label{step:RTAC}
Consider the matrix of inner products modulo Temperley-Lieb and annular consequences, i.e.,
$$
\left(
\langle
\cX-P_{\TL_{n+2,+}}(\cX)-P_{\gA_{n+2}}(\cX),
\cY
\rangle
\right)_{\cX,\cY\in \rt},
$$
where the necessary inner products were derived in Propositions \ref{prop:InnerProduct} and \ref{prop:InnerProductInTL} and Remark \ref{rem:InnerProductInA}.
\be
\item
Taking a basis for the null space of this matrix gives us a basis for $\rtac$.
\item
From this basis, we keep only the essential elements, which we call $X_1,\dots, X_k$.
\ee
\end{step}


\subsection{Compute the jellyfish matrix}

From Step \ref{step:RTAC}, we have an expression for each essential basis element of $\rtac$ as follows:
\begin{align*}
X_i
&= \sum_{P,Q\in\fB}\alpha_{P,Q}^i \TwoTrain{2}{P}{Q} + \sum_{P,Q,R\in\fB}\beta_{P,Q,R}^i \ThreeTrain{P}{Q}{R}+W_i
\end{align*}
where $W_i\in \spann\set{C_i[P\circ Q]\,}{\,P,Q\in\fB\text{ and }i=1,\dots,2n+3}$. 
We also have an expression for $X_i$ as an element of $\TL_{n+2,+}\oplus \gA_{n+2}$.

\begin{step}[Expression in the annular basis]\label{step:AnnularBasis}
Using Proposition \ref{prop:ProjectTrainsToAC}, we express the $X_i$ in terms of the dual annular basis $\widehat{\cup}_{r,s}(S)$ for $S\in\fB$.
We then use our change of basis matrix discussed in Remark \ref{rem:DualAnnularBasis} to write the $\widehat{\cup}_{r,s}(S)$ in terms of the $\cup_{j,\ell}(S)$. 
Hence we may write each $X_i$ as
\begin{align*}
X_i&=\left(\sum_{S\in\fB} \gamma_S^i \cup_{-1,-1}(S) \right) + Y_i +Z_i
=
\left(\sum_{S\in\fB} \gamma_S^i\, \JellyfishSquared{n}{S} \right) + Y_i +Z_i
\end{align*}
where $Y_i$ is a linear combination of the $\cup_{j,\ell}(S)$ for $S\in\fB$ with $(j,\ell)\neq (-1,-1)$, and $Z_i\in \TL_{n+2,+}$.
\end{step}

\begin{nota}
For $P,Q,R,S\in \fB$, we use the following notation:
\begin{align*}
f(P\traincirc{n-2}Q)&= \BoxTwoTrain{P}{Q} \\
f(P\traincirc{n-1}Q\traincirc{n-1}R)&= \BoxThreeTrain{P}{Q}{R}\\
f\cdot j^2(S)&=\BoxJellyfishSquared{n}{S} 
\end{align*}
We will also use the notation $f\cdot X$ to denote $X\in \cP_{n+2,+}$ in jellyfish form with a $\jw{2n+4}$ underneath.
\end{nota}

\begin{step}[Box jellyfish equations]\label{step:BoxJelliesEquations}
Put an $\jw{2n+4}$ underneath the two formulas for $X_i$ obtained in Steps \ref{step:RTAC} and \ref{step:AnnularBasis} to get the following equations for $i=1,\dots, k$:
$$
f\cdot X_i = \sum_{P,Q\in\fB}\alpha_{P,Q}^i f (P\traincirc{n-2}Q) + \sum_{P,Q,R\in\fB}\beta_{P,Q,R}^i  f (P\traincirc{n-1}Q\traincirc{n-1} R)
= \sum_{S\in\fB} \gamma_S^i f\cdot j^2(S).
$$
\end{step}

\begin{rem}\label{rem:NoCheck}
In \cite[Section 3.2]{1208.3637}, similar formulas to those obtained in Step \ref{step:BoxJelliesEquations} were checked by wrapping a Jones-Wenzl around the top of $P\traincirc{n-1} Q$. 
In our case, we can not use this check, since wrapping a Jones-Wenzl around the top of a 3-train does not give another box-train.
\end{rem}

We can now define the jellyfish matrix and the reduced trains matrix from the equations obtained in Step \ref{step:BoxJelliesEquations}.

\begin{defn}\label{defn:JellyfishMatrices}
The \underline{jellyfish matrix} is the matrix $J$ whose $i$-th row is $(\gamma_S^i)_{S\in\fB}$.
The \underline{reduced trains matrix} is the matrix $K$ whose $i$-th row is given by concatenating the lists $(\alpha_{P,Q}^i)_{P,Q\in\fB}$ and $(\beta_{R,S,T}^i)_{R,S,T\in\fB}$.
\end{defn}

\begin{rem}
Note that
$$
K 
\begin{pmatrix}
f (P\traincirc{n-2}Q)\\
\vdots\\
f (R\traincirc{n-1}S\traincirc{n-1}T)\\
\vdots
\end{pmatrix}_{P,Q,R,S,T\in\fB}
=
J
\begin{pmatrix}
f\cdot j^2(S)\\
\vdots
\end{pmatrix}_{S\in \fB}.
$$
\end{rem}

\subsection{Invert the jellyfish matrix}

At this point, we have accomplished most of the difficult work. 
Two easy steps remain.

\begin{step}[Invert $J$]\label{step:Invert}
Given the matrix $J$ from Definition \ref{defn:JellyfishMatrices} obtained via Step \ref{step:BoxJelliesEquations}, we check if it has rank $|\fB|$. 
If it does (and we know that it should by \cite{1208.1564}), we find a left inverse for $J$ by the formula
$$
J^L=(J^*J)^{-1}J^*
$$
since $J$ and $J^*J$ have the same rank. 
\end{step}

\begin{step}[Box jellyfish relations]\label{step:BoxJelliesFormulas}
Get the \underline{box jellyfish relations} by multiplying by $J^L$ from Step \ref{step:Invert} 
$$
\begin{pmatrix}
f\cdot j^2(S)\\
\vdots
\end{pmatrix}_{S\in \fB}
=
J^LK 
\begin{pmatrix}
f (P\traincirc{n-2}Q)\\
\vdots\\
f (P\traincirc{n-1}Q\traincirc{n-1}R)\\
\vdots
\end{pmatrix}_{P,Q,R\in\fB},
$$
which express the $f\cdot j^2(S)$ as linear combinations of reduced trains.
\end{step}

\begin{rem}
See \cite[Section 2.5]{1208.3637} for the discussion of how the existence of box jellyfish relations is equivalent to the existence of jellyfish relations.
\end{rem}

\subsection{Checking our calculations}\label{sec:Checking}

Since the computer is doing all the arithmetic, it is good to check that our formulas are consistent with other methods of calculation. 
The computations in this section are redundant; hence we freely take shortcuts and perform spot checks when more thorough checks would be too time consuming.

The checks we perform in this subsection are done directly in the graph planar algebra.
As such computations are computationally expensive, we use the following shortcut, which is known to experts.
We do not prove it here as it would take us too far afield.

\begin{prop}
Suppose $\cP_\bullet$ is a subfactor planar algebra.
Choose an embedding of $\cP_\bullet$ into $\cG\cP\cA(\Gamma_+)_\bullet$, the graph planar algebra of its principal graph, and identify $\cP_\bullet$ with its image.
Define the map $\Phi:\cP_{k,\pm}\to\cG\cP\cA(\Gamma_+)_{k,\pm}$ by cutting down at the zero box $\star$ (the distinguished vertex of $\Gamma_+$), i.e., forgetting all loops of length $2k$ which do not start at $\star$.
$$
\begin{tikzpicture}[baseline=-.1cm]
	\draw (0,-.8) -- (0,.8);
	\nbox{unshaded}{(0,0)}{0}{0}{x}
	\node at (.2,-.6) {\scriptsize{$k$}};
	\node at (.2,.6) {\scriptsize{$k$}};
\end{tikzpicture}
\overset{\Phi}{\longmapsto}
\begin{tikzpicture}[baseline=-.1cm]
	\draw[thick] (-1,-.2) rectangle (-.6,.2);
	\node at (-.8,0) {$\star$};
	\draw (0,-.8) -- (0,.8);
	\nbox{unshaded}{(0,0)}{0}{0}{x}
	\node at (.2,-.6) {\scriptsize{$k$}};
	\node at (.2,.6) {\scriptsize{$k$}};
\end{tikzpicture}
$$
Then $\Phi$ is a $*$-algebra isomorphism under the usual multiplication, and $\Phi$ commutes with taking (partial) traces.
\end{prop}
We remark that $\dim(\cP_{k,\pm})$ is equal to the number of loops of length $2k$ starting at $\star$ on the principal graph, so one only needs to prove this map is injective.

To simplify calculations in the graph planar algebra, we can compute the inner product by first cutting down at star and then taking the inner product of the cut down elements in the graph planar algebra. 
Note that this simplification assumes we are working in the image of a subfactor planar algebra, so it cannot be used to prove that formulas hold. 
However, it can be used as a check for our calculations.

Using this shortcut, we check the propositions listed in the following table.
The calculations are performed in the notebook {\tt TwoStrandJellyfish} in subsections called {\tt Checking directly in the GPA} for each of our examples.
Many of the computations are exact, but two are numerical.
For the checks for Propositions \ref{prop:ProjectTrainsToAC} and \ref{prop:ProjectACofTrainsToAC}, we don't check all the coefficients in the graph planar algebra; rather we only check the coefficients that our formulas tell us are nonzero.
\\

\begin{tabular}{|c|c|c|}
\hline
Proposition & Checking functions & Numerical?
\\\hline
\ref{prop:AnnularInnerProducts} 
&
{\tt CheckPairwiseInnerProductsOfSecondAC}
&
Yes
\\\hline
\ref{prop:ProjectTrainsToAC}
&
{\tt CheckCoefficientsOf2TrainsInSecondAC}
&
No
\\
&{\tt CheckCoefficientsOf3TrainsInSecondAC} &
\\\hline
\ref{prop:ProjectACofTrainsToAC} 
&
{\tt CheckCoefficientsOfCiQTCircsInSecondAC}
&
No
\\\hline
\ref{prop:InnerProduct} 
&
{\tt CheckInnerProductBetweenTrains}
&
No
\\\hline
\ref{prop:InnerProduct2}
&
{\tt CheckInnerProductWithCiQTCircs}
&
Yes
\\\hline
\end{tabular}
\\

We note that the formula for the two-strand jellyfish relation for $3^{\Z/3}$ (Haagerup) agrees with that obtained in \cite{MR2979509}.
We have not checked that our two-strand relations for $3^{\Z/2\times\Z/2}$ are consistent with the one-strand relations found in \cite{1208.3637}, since we are using different generators.
Unfortunately, as noted in Remark \ref{rem:NoCheck}, we do not get another check for the entries of the jellyfish matrices from wrapping the Jones-Wenzl idempotent around the top of a 3-train.

In \cite{1208.3637}, the authors were able to check the one-strand jellyfish relations for 2221 directly in the graph planar algebra using a clever trick due to Bigelow.
We cannot do these computations for our graphs. 
Not only are our graphs 3-supertransitive, but we also use two-strand relations, making the preparation of the two-cup Jones-Wenzl too computationally expensive.

\section{Examples}\label{sec:relations}

We now have a subsection for each of our examples. 
Each has an orthogonal set of minimal generators $\fB$ which lives in the appropriate graph planar algebra. 
Formulas for these generators are given in Appendix \ref{sec:Generators}. 
We first check that Assumptions \ref{assume:Generators}, \ref{assume:SpanAlgebras}, and \ref{assume:Tetrahedral} hold for these generators, i.e., 
\begin{itemize}
\item
The elements $R\in \fB$ are self-adjoint low-weight rotational eigenvectors with corresponding chiralities $\sigma_R$ given in Appendix \ref{sec:Generators}. 
Moreover, $\fB$ is linearly independent and orthogonal and has scalar moments. The moments are given in Appendix \ref{sec:MomentsAndTetrahedrals}.
\item
$\fB\cup\{\jw{n}\}$ and $\check{\fB}\cup\{\jw{n}\}$ span complex algebras under the usual multiplication. We check this using the program {\texttt{VerifyClosedUnderMultiplication}} in the notebook {\texttt{TwoStrandJellyfish.nb}}.
\item
The tetrahedral structure constants $\Delta(P,Q,R\mid S)$ are scalars for all $P,Q,R,S\in\fB$. The tetrahedral constants are given in Appendix \ref{sec:MomentsAndTetrahedrals}.
\end{itemize} 

Each of the subsections below consists of three lemmas, which consist of performing the calculations described in Section \ref{sec:ComputingJellyfishRelations} for each of our examples. 
The proofs are simply substituting in the appropriate quantities (moments, tetrahedral structure constants) where applicable, and executing the functions in the \texttt{Mathematica} notebooks included with the \texttt{arXiv} sources of this article.

Throughout, the notation  $\lambda_{a_n, \ldots, a_0}^{(z)}$ denotes the root of the polynomial $\sum_i a_i x^i$ which is closest to the approximate real number $z$. 
(The digits of precision of $z$ are in each case chosen so that this unambiguously identifies the root.) 
For example, $\lambda^{(0.3278)}_{1024,0,-864,0,81}$ denotes the root of $1024 x^4 - 864 x^2 + 81$ which is closest to $0.3278$.

For each subsection, $A,B$ are the generators given in the subsection of Appendix \ref{sec:Generators} for each respective graph.

\subsection{$3^{\Z/3}$: Haagerup}

\begin{lem}\label{25781EDEA61320E6E26689FBED470E90-rtac}
The following linear combination $X$ of reduced trains lies in annular consequences.  The column marked $X$ gives the coefficients of the reduced trains for $X$.

{\scriptsize{
\begin{longtable}{|c|c|}
 \hline 
 & $X_1$\\[4pt]
\hline $A\traincirc{n-2}A$ & $0$ \\[4pt]
\hline $A\traincirc{n-1}A\traincirc{n-1}A$ & $1$ \\[4pt]
\hline $C_{1}[A\traincirc{n-1}A]$ & $0$ \\[4pt]
\hline $C_{2}[A\traincirc{n-1}A]$ & $0$ \\[4pt]
\hline $C_{3}[A\traincirc{n-1}A]$ & $0$ \\[4pt]
\hline $C_{4}[A\traincirc{n-1}A]$ & $0$ \\[4pt]
\hline $C_{5}[A\traincirc{n-1}A]$ & $0$ \\[4pt]
\hline $C_{6}[A\traincirc{n-1}A]$ & $0$ \\[4pt]
\hline $C_{7}[A\traincirc{n-1}A]$ & $0$ \\[4pt]
\hline $C_{8}[A\traincirc{n-1}A]$ & $0$ \\[4pt]
\hline $C_{9}[A\traincirc{n-1}A]$ & $0$ \\[4pt]
\hline $C_{10}[A\traincirc{n-1}A]$ & $0$ \\[4pt]
\hline $C_{11}[A\traincirc{n-1}A]$ & $0$ \\[4pt]
\hline
\end{longtable}
}}
\end{lem}

\begin{lem}
Let $K$ be the transpose of the $2\times 1$ matrix whose entries are given by the first 2 rows and the 1 column of the table in Lemma \ref{25781EDEA61320E6E26689FBED470E90-rtac}. Then we have
$$
K
\begin{pmatrix}
f(A\traincirc{n-2}A) \\[6pt]
f(A\traincirc{n-1}A\traincirc{n-1}A)\end{pmatrix}=J
\begin{pmatrix}
f\cdot j^2(A)\end{pmatrix}$$
where
\begin{align*}
J &= \begin{pmatrix}
\lambda_{243,0,-132,0,16}^{(0.6005)}\end{pmatrix}.\\
\end{align*}
\end{lem}

\begin{lem}
The elements $A$ satisfy the box jellyfish relations
$$
\begin{pmatrix}
f\cdot j^2(A)\end{pmatrix}=J^{L}K
\begin{pmatrix}
f(A\traincirc{n-2}A) \\[6pt]
f(A\traincirc{n-1}A\traincirc{n-1}A)\end{pmatrix}$$
where
\begin{align*}
(J^{L}K)^T &= \begin{pmatrix}
0 \\[6pt]
\lambda_{16,0,-132,0,243}^{(1.665)}\end{pmatrix}.\\
\end{align*}
\end{lem}

\subsection{$3^{\Z/2\times \Z/2}$}

\begin{lem}\label{772136FAF425B059772136FAF425B059-rtac}
The following linear combinations $X_i$ of reduced trains lie in annular consequences.  The column marked $X_i$ gives the coefficients of the reduced trains for $X_i$).

{\scriptsize{
\begin{longtable}{|c|c|c|}
 \hline 
 & $X_1$ & $X_2$\\[4pt]
\hline $A\traincirc{n-2}A$ & $1$ & $0$ \\[4pt]
\hline $A\traincirc{n-2}B$ & $0$ & $1$ \\[4pt]
\hline $B\traincirc{n-2}A$ & $0$ & $1$ \\[4pt]
\hline $B\traincirc{n-2}B$ & $-\frac{1}{3}$ & $0$ \\[4pt]
\hline $A\traincirc{n-1}A\traincirc{n-1}A$ & $\frac{3 \sqrt{\frac{5}{2}}}{4}$ & $\frac{27}{4 \sqrt{2}}$ \\[4pt]
\hline $A\traincirc{n-1}A\traincirc{n-1}B$ & $\lambda_{1024,0,-576,0,1}^{(0.74884)}$ & $\lambda_{1024,0,-16704,0,9801}^{(-0.781)}$ \\[4pt]
\hline $A\traincirc{n-1}B\traincirc{n-1}A$ & $\lambda_{1024,0,-2880,0,25}^{(0.09331)}$ & $\lambda_{1024,0,-46656,0,6561}^{(0.3756)}$ \\[4pt]
\hline $A\traincirc{n-1}B\traincirc{n-1}B$ & $\lambda_{1024,0,-1856,0,121}^{(0.2602)}$ & $\lambda_{1024,0,-576,0,1}^{(0.74884)}$ \\[4pt]
\hline $B\traincirc{n-1}A\traincirc{n-1}A$ & $\lambda_{1024,0,-576,0,1}^{(0.74884)}$ & $\lambda_{1024,0,-16704,0,9801}^{(-0.781)}$ \\[4pt]
\hline $B\traincirc{n-1}A\traincirc{n-1}B$ & $\lambda_{1024,0,-5184,0,81}^{(-0.1252)}$ & $\lambda_{1024,0,-2880,0,25}^{(0.09331)}$ \\[4pt]
\hline $B\traincirc{n-1}B\traincirc{n-1}A$ & $\lambda_{1024,0,-1856,0,121}^{(0.2602)}$ & $\lambda_{1024,0,-576,0,1}^{(0.74884)}$ \\[4pt]
\hline $B\traincirc{n-1}B\traincirc{n-1}B$ & $\frac{3}{4 \sqrt{2}}$ & $-\frac{\sqrt{\frac{5}{2}}}{4}$ \\[4pt]
\hline $C_{1}[A\traincirc{n-1}A]$ & $0$ & $0$ \\[4pt]
\hline $C_{1}[A\traincirc{n-1}B]$ & $0$ & $0$ \\[4pt]
\hline $C_{1}[B\traincirc{n-1}A]$ & $0$ & $0$ \\[4pt]
\hline $C_{1}[B\traincirc{n-1}B]$ & $0$ & $0$ \\[4pt]
\hline $C_{2}[A\traincirc{n-1}A]$ & $0$ & $0$ \\[4pt]
\hline $C_{2}[A\traincirc{n-1}B]$ & $0$ & $0$ \\[4pt]
\hline $C_{2}[B\traincirc{n-1}A]$ & $0$ & $0$ \\[4pt]
\hline $C_{2}[B\traincirc{n-1}B]$ & $0$ & $0$ \\[4pt]
\hline $C_{3}[A\traincirc{n-1}A]$ & $0$ & $0$ \\[4pt]
\hline $C_{3}[A\traincirc{n-1}B]$ & $0$ & $0$ \\[4pt]
\hline $C_{3}[B\traincirc{n-1}A]$ & $0$ & $0$ \\[4pt]
\hline $C_{3}[B\traincirc{n-1}B]$ & $0$ & $0$ \\[4pt]
\hline $C_{4}[A\traincirc{n-1}A]$ & $0$ & $0$ \\[4pt]
\hline $C_{4}[A\traincirc{n-1}B]$ & $0$ & $0$ \\[4pt]
\hline $C_{4}[B\traincirc{n-1}A]$ & $0$ & $0$ \\[4pt]
\hline $C_{4}[B\traincirc{n-1}B]$ & $0$ & $0$ \\[4pt]
\hline $C_{5}[A\traincirc{n-1}A]$ & $0$ & $0$ \\[4pt]
\hline $C_{5}[A\traincirc{n-1}B]$ & $0$ & $0$ \\[4pt]
\hline $C_{5}[B\traincirc{n-1}A]$ & $0$ & $0$ \\[4pt]
\hline $C_{5}[B\traincirc{n-1}B]$ & $0$ & $0$ \\[4pt]
\hline $C_{6}[A\traincirc{n-1}A]$ & $0$ & $0$ \\[4pt]
\hline $C_{6}[A\traincirc{n-1}B]$ & $0$ & $0$ \\[4pt]
\hline $C_{6}[B\traincirc{n-1}A]$ & $0$ & $0$ \\[4pt]
\hline $C_{6}[B\traincirc{n-1}B]$ & $0$ & $0$ \\[4pt]
\hline $C_{7}[A\traincirc{n-1}A]$ & $0$ & $0$ \\[4pt]
\hline $C_{7}[A\traincirc{n-1}B]$ & $0$ & $0$ \\[4pt]
\hline $C_{7}[B\traincirc{n-1}A]$ & $0$ & $0$ \\[4pt]
\hline $C_{7}[B\traincirc{n-1}B]$ & $0$ & $0$ \\[4pt]
\hline $C_{8}[A\traincirc{n-1}A]$ & $0$ & $0$ \\[4pt]
\hline $C_{8}[A\traincirc{n-1}B]$ & $0$ & $0$ \\[4pt]
\hline $C_{8}[B\traincirc{n-1}A]$ & $0$ & $0$ \\[4pt]
\hline $C_{8}[B\traincirc{n-1}B]$ & $0$ & $0$ \\[4pt]
\hline $C_{9}[A\traincirc{n-1}A]$ & $0$ & $0$ \\[4pt]
\hline $C_{9}[A\traincirc{n-1}B]$ & $0$ & $0$ \\[4pt]
\hline $C_{9}[B\traincirc{n-1}A]$ & $0$ & $0$ \\[4pt]
\hline $C_{9}[B\traincirc{n-1}B]$ & $0$ & $0$ \\[4pt]
\hline $C_{10}[A\traincirc{n-1}A]$ & $0$ & $0$ \\[4pt]
\hline $C_{10}[A\traincirc{n-1}B]$ & $0$ & $0$ \\[4pt]
\hline $C_{10}[B\traincirc{n-1}A]$ & $0$ & $0$ \\[4pt]
\hline $C_{10}[B\traincirc{n-1}B]$ & $0$ & $0$ \\[4pt]
\hline $C_{11}[A\traincirc{n-1}A]$ & $0$ & $0$ \\[4pt]
\hline $C_{11}[A\traincirc{n-1}B]$ & $0$ & $0$ \\[4pt]
\hline $C_{11}[B\traincirc{n-1}A]$ & $0$ & $0$ \\[4pt]
\hline $C_{11}[B\traincirc{n-1}B]$ & $0$ & $0$ \\[4pt]
\hline
\end{longtable}
}}
\end{lem}

\begin{lem}
Let $K$ be the transpose of the $12\times 2$ matrix whose entries are given by the first 12 rows and the 2 columns of the table in Lemma \ref{772136FAF425B059772136FAF425B059-rtac}. Then we have
$$
K
\begin{pmatrix}
f(A\traincirc{n-2}A) \\[6pt]
f(A\traincirc{n-2}B) \\[6pt]
f(B\traincirc{n-2}A) \\[6pt]
f(B\traincirc{n-2}B) \\[6pt]
f(A\traincirc{n-1}A\traincirc{n-1}A) \\[6pt]
f(A\traincirc{n-1}A\traincirc{n-1}B) \\[6pt]
f(A\traincirc{n-1}B\traincirc{n-1}A) \\[6pt]
f(A\traincirc{n-1}B\traincirc{n-1}B) \\[6pt]
f(B\traincirc{n-1}A\traincirc{n-1}A) \\[6pt]
f(B\traincirc{n-1}A\traincirc{n-1}B) \\[6pt]
f(B\traincirc{n-1}B\traincirc{n-1}A) \\[6pt]
f(B\traincirc{n-1}B\traincirc{n-1}B)\end{pmatrix}=J
\begin{pmatrix}
f\cdot j^2(A) \\[6pt]
f\cdot j^2(B)\end{pmatrix}$$
where
\begin{align*}
J &= \begin{pmatrix}
\frac{1}{4} \left(6-3 \sqrt{5}\right) & -\frac{3}{4} \\[6pt]
-\frac{9}{4} & \frac{1}{4} \left(3 \sqrt{5}-6\right)\end{pmatrix}.\\
\end{align*}
\end{lem}

\begin{lem}
The elements $A$,$B$ satisfy the box jellyfish relations
$$
\begin{pmatrix}
f\cdot j^2(A) \\[6pt]
f\cdot j^2(B)\end{pmatrix}=J^{L}K
\begin{pmatrix}
f(A\traincirc{n-2}A) \\[6pt]
f(A\traincirc{n-2}B) \\[6pt]
f(B\traincirc{n-2}A) \\[6pt]
f(B\traincirc{n-2}B) \\[6pt]
f(A\traincirc{n-1}A\traincirc{n-1}A) \\[6pt]
f(A\traincirc{n-1}A\traincirc{n-1}B) \\[6pt]
f(A\traincirc{n-1}B\traincirc{n-1}A) \\[6pt]
f(A\traincirc{n-1}B\traincirc{n-1}B) \\[6pt]
f(B\traincirc{n-1}A\traincirc{n-1}A) \\[6pt]
f(B\traincirc{n-1}A\traincirc{n-1}B) \\[6pt]
f(B\traincirc{n-1}B\traincirc{n-1}A) \\[6pt]
f(B\traincirc{n-1}B\traincirc{n-1}B)\end{pmatrix}$$
where
\begin{align*}
(J^{L}K)^T &= \begin{pmatrix}
\frac{1}{12} \left(1-\sqrt{5}\right) & \frac{1}{4} \left(-3-\sqrt{5}\right) \\[6pt]
\frac{1}{12} \left(-3-\sqrt{5}\right) & \frac{1}{12} \left(\sqrt{5}-1\right) \\[6pt]
\frac{1}{12} \left(-3-\sqrt{5}\right) & \frac{1}{12} \left(\sqrt{5}-1\right) \\[6pt]
\frac{1}{36} \left(\sqrt{5}-1\right) & \frac{1}{12} \left(3+\sqrt{5}\right) \\[6pt]
\lambda_{64,0,-336,0,121}^{(-2.20)} & -\frac{3}{2 \sqrt{2}} \\[6pt]
\frac{\sqrt{\frac{5}{2}}}{6} & -\frac{3}{2 \sqrt{2}} \\[6pt]
\lambda_{5184,0,-4176,0,121}^{(-0.1735)} & \lambda_{64,0,-144,0,1}^{(-0.08346)} \\[6pt]
-\frac{1}{2 \sqrt{2}} & -\frac{\sqrt{\frac{5}{2}}}{6} \\[6pt]
\frac{\sqrt{\frac{5}{2}}}{6} & -\frac{3}{2 \sqrt{2}} \\[6pt]
\lambda_{5184,0,-1296,0,1}^{(-0.027821)} & \lambda_{5184,0,-4176,0,121}^{(0.1735)} \\[6pt]
-\frac{1}{2 \sqrt{2}} & -\frac{\sqrt{\frac{5}{2}}}{6} \\[6pt]
\frac{1}{6 \sqrt{2}} & \lambda_{5184,0,-3024,0,121}^{(-0.7349)}\end{pmatrix}.\\
\end{align*}
\end{lem}

\subsection{$3^{\Z/4}$}

\begin{lem}\label{B5425B05363851E6DB2C121504201C9D-rtac}
The following linear combinations $X_i$ of reduced trains lie in annular consequences.  The column marked $X_i$ gives the coefficients of the reduced trains for $X_i$).

{\scriptsize{
\begin{longtable}{|c|c|c|}
 \hline 
 & $X_1$ & $X_2$\\[4pt]
\hline $A\traincirc{n-2}A$ & $1$ & $0$ \\[4pt]
\hline $A\traincirc{n-2}B$ & $0$ & $1$ \\[4pt]
\hline $B\traincirc{n-2}A$ & $\lambda_{2025,0,-720,0,-16}^{(0.1449 i)}$ & $1$ \\[4pt]
\hline $B\traincirc{n-2}B$ & $\frac{1}{45} \left(-10-3 \sqrt{5}\right)$ & $0$ \\[4pt]
\hline $A\traincirc{n-1}A\traincirc{n-1}A$ & $\lambda_{100,0,-2610,0,-81}^{(0.1761 i)}$ & $\lambda_{4,0,-1134,0,6561}^{(2.43)}$ \\[4pt]
\hline $A\traincirc{n-1}A\traincirc{n-1}B$ & $\lambda_{100,0,-1030,0,121}^{(0.3447)}$ & $\lambda_{4,0,-198,0,-81}^{(0.637 i)}$ \\[4pt]
\hline $A\traincirc{n-1}B\traincirc{n-1}A$ & $\lambda_{25,0,-180,0,4}^{(0.1493)}$ & $0$ \\[4pt]
\hline $A\traincirc{n-1}B\traincirc{n-1}B$ & $\lambda_{4,0,-8,0,-1}^{(-0.3436 i)}$ & $\lambda_{4,0,-180,0,25}^{(0.3733)}$ \\[4pt]
\hline $B\traincirc{n-1}A\traincirc{n-1}A$ & $\lambda_{4,0,-6,0,1}^{(0.4370)}$ & $\lambda_{4,0,-198,0,-81}^{(-0.637 i)}$ \\[4pt]
\hline $B\traincirc{n-1}A\traincirc{n-1}B$ & $\lambda_{100,0,-290,0,-1}^{(0.05869 i)}$ & $\lambda_{4,0,-126,0,81}^{(0.810)}$ \\[4pt]
\hline $B\traincirc{n-1}B\traincirc{n-1}A$ & $\lambda_{324,0,-2232,0,-361}^{(0.3976 i)}$ & $\lambda_{4,0,-180,0,25}^{(0.3733)}$ \\[4pt]
\hline $B\traincirc{n-1}B\traincirc{n-1}B$ & $\lambda_{164025,0,-34020,0,484}^{(0.4382)}$ & $0$ \\[4pt]
\hline $C_{1}[A\traincirc{n-1}A]$ & $0$ & $0$ \\[4pt]
\hline $C_{1}[A\traincirc{n-1}B]$ & $0$ & $0$ \\[4pt]
\hline $C_{1}[B\traincirc{n-1}A]$ & $\lambda_{2025,0,-3420,0,-1}^{(-0.017098 i)}$ & $\sqrt{5}-2$ \\[4pt]
\hline $C_{1}[B\traincirc{n-1}B]$ & $\frac{1}{45} \left(5-3 \sqrt{5}\right)$ & $\lambda_{81,0,-1044,0,-16}^{(-0.1237 i)}$ \\[4pt]
\hline $C_{2}[A\traincirc{n-1}A]$ & $0$ & $0$ \\[4pt]
\hline $C_{2}[A\traincirc{n-1}B]$ & $0$ & $0$ \\[4pt]
\hline $C_{2}[B\traincirc{n-1}A]$ & $\lambda_{2025,0,-2610,0,-4}^{(0.039125 i)}$ & $\lambda_{1,0,-14,0,4}^{(-0.540)}$ \\[4pt]
\hline $C_{2}[B\traincirc{n-1}B]$ & $\lambda_{164025,0,-9720,0,64}^{(0.08686)}$ & $\lambda_{81,0,-792,0,-64}^{(0.2831 i)}$ \\[4pt]
\hline $C_{3}[A\traincirc{n-1}A]$ & $0$ & $0$ \\[4pt]
\hline $C_{3}[A\traincirc{n-1}B]$ & $0$ & $0$ \\[4pt]
\hline $C_{3}[B\traincirc{n-1}A]$ & $\lambda_{2025,0,-180,0,-1}^{(-0.07243 i)}$ & $1$ \\[4pt]
\hline $C_{3}[B\traincirc{n-1}B]$ & $\frac{1}{45} \left(-5-\sqrt{5}\right)$ & $\lambda_{81,0,-36,0,-16}^{(-0.5241 i)}$ \\[4pt]
\hline $C_{4}[A\traincirc{n-1}A]$ & $0$ & $0$ \\[4pt]
\hline $C_{4}[A\traincirc{n-1}B]$ & $0$ & $0$ \\[4pt]
\hline $C_{4}[B\traincirc{n-1}A]$ & $\lambda_{2025,0,-3960,0,-64}^{(0.1266 i)}$ & $\lambda_{1,0,-24,0,64}^{(-1.75)}$ \\[4pt]
\hline $C_{4}[B\traincirc{n-1}B]$ & $\frac{4 \sqrt{\frac{2}{5}}}{9}$ & $\lambda_{81,0,-1152,0,-1024}^{(0.916 i)}$ \\[4pt]
\hline $C_{5}[A\traincirc{n-1}A]$ & $0$ & $0$ \\[4pt]
\hline $C_{5}[A\traincirc{n-1}B]$ & $0$ & $0$ \\[4pt]
\hline $C_{5}[B\traincirc{n-1}A]$ & $\lambda_{25,0,-20,0,-1}^{(-0.2173 i)}$ & $3$ \\[4pt]
\hline $C_{5}[B\traincirc{n-1}B]$ & $\frac{1}{15} \left(-5-\sqrt{5}\right)$ & $\lambda_{1,0,-4,0,-16}^{(-1.57 i)}$ \\[4pt]
\hline $C_{6}[A\traincirc{n-1}A]$ & $0$ & $0$ \\[4pt]
\hline $C_{6}[A\traincirc{n-1}B]$ & $0$ & $0$ \\[4pt]
\hline $C_{6}[B\traincirc{n-1}A]$ & $\lambda_{2025,0,-15840,0,-1024}^{(-0.2532 i)}$ & $\lambda_{1,0,-96,0,1024}^{(-3.50)}$ \\[4pt]
\hline $C_{6}[B\traincirc{n-1}B]$ & $\lambda_{164025,0,-87480,0,7744}^{(0.3348)}$ & $\lambda_{81,0,-360,0,-1600}^{(1.66 i)}$ \\[4pt]
\hline $C_{7}[A\traincirc{n-1}A]$ & $0$ & $0$ \\[4pt]
\hline $C_{7}[A\traincirc{n-1}B]$ & $0$ & $0$ \\[4pt]
\hline $C_{7}[B\traincirc{n-1}A]$ & $\lambda_{25,0,-180,0,-81}^{(0.652 i)}$ & $3$ \\[4pt]
\hline $C_{7}[B\traincirc{n-1}B]$ & $\frac{1}{30} \left(\sqrt{5}-5\right)$ & $\lambda_{1,0,1,0,-1}^{(-1.27 i)}$ \\[4pt]
\hline $C_{8}[A\traincirc{n-1}A]$ & $0$ & $0$ \\[4pt]
\hline $C_{8}[A\traincirc{n-1}B]$ & $0$ & $0$ \\[4pt]
\hline $C_{8}[B\traincirc{n-1}A]$ & $\lambda_{25,0,-440,0,-64}^{(-0.3798 i)}$ & $\lambda_{1,0,-24,0,64}^{(-1.75)}$ \\[4pt]
\hline $C_{8}[B\traincirc{n-1}B]$ & $\lambda_{164025,0,-22680,0,64}^{(0.05368)}$ & $\lambda_{81,0,-72,0,-64}^{(0.741 i)}$ \\[4pt]
\hline $C_{9}[A\traincirc{n-1}A]$ & $0$ & $0$ \\[4pt]
\hline $C_{9}[A\traincirc{n-1}B]$ & $0$ & $0$ \\[4pt]
\hline $C_{9}[B\traincirc{n-1}A]$ & $\lambda_{25,0,-20,0,-1}^{(0.2173 i)}$ & $1$ \\[4pt]
\hline $C_{9}[B\traincirc{n-1}B]$ & $\frac{1}{90} \left(\sqrt{5}-5\right)$ & $\lambda_{81,0,9,0,-1}^{(-0.4240 i)}$ \\[4pt]
\hline $C_{10}[A\traincirc{n-1}A]$ & $0$ & $0$ \\[4pt]
\hline $C_{10}[A\traincirc{n-1}B]$ & $0$ & $0$ \\[4pt]
\hline $C_{10}[B\traincirc{n-1}A]$ & $\lambda_{25,0,-290,0,-4}^{(-0.1174 i)}$ & $\lambda_{1,0,-14,0,4}^{(-0.540)}$ \\[4pt]
\hline $C_{10}[B\traincirc{n-1}B]$ & $\lambda_{164025,0,-14580,0,4}^{(0.016589)}$ & $\lambda_{81,0,-72,0,-4}^{(0.2290 i)}$ \\[4pt]
\hline $C_{11}[A\traincirc{n-1}A]$ & $0$ & $0$ \\[4pt]
\hline $C_{11}[A\traincirc{n-1}B]$ & $0$ & $0$ \\[4pt]
\hline $C_{11}[B\traincirc{n-1}A]$ & $\lambda_{25,0,-380,0,-1}^{(0.05129 i)}$ & $\sqrt{5}-2$ \\[4pt]
\hline $C_{11}[B\traincirc{n-1}B]$ & $\frac{1}{90} \left(15-7 \sqrt{5}\right)$ & $\lambda_{81,0,-99,0,-1}^{(-0.1001 i)}$ \\[4pt]
\hline
\end{longtable}
}}
\end{lem}

\begin{lem}
Let $K$ be the transpose of the $12\times 2$ matrix whose entries are given by the first 12 rows and the 2 columns of the table in Lemma \ref{B5425B05363851E6DB2C121504201C9D-rtac}. Then we have
$$
K
\begin{pmatrix}
f(A\traincirc{n-2}A) \\[6pt]
f(A\traincirc{n-2}B) \\[6pt]
f(B\traincirc{n-2}A) \\[6pt]
f(B\traincirc{n-2}B) \\[6pt]
f(A\traincirc{n-1}A\traincirc{n-1}A) \\[6pt]
f(A\traincirc{n-1}A\traincirc{n-1}B) \\[6pt]
f(A\traincirc{n-1}B\traincirc{n-1}A) \\[6pt]
f(A\traincirc{n-1}B\traincirc{n-1}B) \\[6pt]
f(B\traincirc{n-1}A\traincirc{n-1}A) \\[6pt]
f(B\traincirc{n-1}A\traincirc{n-1}B) \\[6pt]
f(B\traincirc{n-1}B\traincirc{n-1}A) \\[6pt]
f(B\traincirc{n-1}B\traincirc{n-1}B)\end{pmatrix}=J
\begin{pmatrix}
f\cdot j^2(A) \\[6pt]
f\cdot j^2(B)\end{pmatrix}$$
where
\begin{align*}
J &= \begin{pmatrix}
\lambda_{400,0,-5220,0,-81}^{(0.1245 i)} & \frac{1}{10} \left(-5-\sqrt{5}\right) \\[6pt]
\frac{1}{4} \left(27-9 \sqrt{5}\right) & 0\end{pmatrix}.\\
\end{align*}
\end{lem}

\begin{lem}
The elements $A$,$B$ satisfy the box jellyfish relations
$$
\begin{pmatrix}
f\cdot j^2(A) \\[6pt]
f\cdot j^2(B)\end{pmatrix}=J^{L}K
\begin{pmatrix}
f(A\traincirc{n-2}A) \\[6pt]
f(A\traincirc{n-2}B) \\[6pt]
f(B\traincirc{n-2}A) \\[6pt]
f(B\traincirc{n-2}B) \\[6pt]
f(A\traincirc{n-1}A\traincirc{n-1}A) \\[6pt]
f(A\traincirc{n-1}A\traincirc{n-1}B) \\[6pt]
f(A\traincirc{n-1}B\traincirc{n-1}A) \\[6pt]
f(A\traincirc{n-1}B\traincirc{n-1}B) \\[6pt]
f(B\traincirc{n-1}A\traincirc{n-1}A) \\[6pt]
f(B\traincirc{n-1}A\traincirc{n-1}B) \\[6pt]
f(B\traincirc{n-1}B\traincirc{n-1}A) \\[6pt]
f(B\traincirc{n-1}B\traincirc{n-1}B)\end{pmatrix}$$
where
\begin{align*}
(J^{L}K)^T &= \begin{pmatrix}
0 & \frac{1}{2} \left(\sqrt{5}-5\right) \\[6pt]
\frac{1}{9} \left(3+\sqrt{5}\right) & \lambda_{81,0,-99,0,-1}^{(0.1001 i)} \\[6pt]
\frac{1}{9} \left(3+\sqrt{5}\right) & \lambda_{81,0,-99,0,-1}^{(-0.1001 i)} \\[6pt]
0 & \frac{1}{18} \left(7+\sqrt{5}\right) \\[6pt]
\sqrt{2} & 0 \\[6pt]
\lambda_{81,0,-18,0,-4}^{(0.3706 i)} & \lambda_{1,0,-14,0,4}^{(-0.540)} \\[6pt]
0 & \lambda_{1,0,-94,0,4}^{(-0.2063)} \\[6pt]
\lambda_{6561,0,-2430,0,100}^{(0.2172)} & \lambda_{81,0,-360,0,-100}^{(0.512 i)} \\[6pt]
\lambda_{81,0,-18,0,-4}^{(-0.3706 i)} & \lambda_{1,0,-14,0,4}^{(-0.540)} \\[6pt]
\frac{\sqrt{2}}{3} & 0 \\[6pt]
\lambda_{6561,0,-2430,0,100}^{(0.2172)} & \lambda_{81,0,-360,0,-100}^{(-0.512 i)} \\[6pt]
0 & \lambda_{6561,0,-3726,0,484}^{(-0.6056)}\end{pmatrix}.\\
\end{align*}
\end{lem}

\section{Calculating principal graphs}\label{sec:PrincipalGraphs}

We now know that for each of our examples $G\in \{\Z/3,\Z/2\times\Z/2,\Z/4\}$, the corresponding set of minimal generators given in Appendix \ref{sec:Generators} generates an evaluable, subfactor planar algebra $\cP_\bullet^G$. 
We must now determine the principal graphs of the $\cP_\bullet^G$.
By the next lemma, we know that the principal graphs have the desired supertransitivity since we have two-strand jellyfish relations.

\begin{lem}\label{lem:Supertransitive}
Suppose a planar algebra $\cP_\bullet$ is generated by uncappable elements $A_1,\dots, A_k\in \cP_{n,+}$ such that
\begin{enumerate}[(1)]
\item
the $A_j$'s satisfy two-strand jellyfish relations, and
\item
the complex span of $\{ A_1,\dots, A_k,f^{(n)}\}$ forms an algebra under the usual multiplication. 
\end{enumerate}
Then $\cP_\bullet$ is $(n-1)$ supertransitive.
\end{lem}
\begin{proof}
Similar to \cite[Lemma 5.1]{1208.3637}.
\end{proof}

We now determine the principal graphs of the $\cP_\bullet^{G}$.
These arguments are similar to those in \cite[Section 5]{1208.3637}.

\subsection{$3^{\Z/3}$: Haagerup}

\begin{thm}\label{thm:Z3graph}
The principal graphs of $\cP_\bullet^{\Z/3}$ are 
$$
\left(
\bigraph{bwd1v1v1v1p1v1x0p0x1v1x0p0x1duals1v1v1x2v2x1},
\bigraph{bwd1v1v1v1p1v1x0p1x0duals1v1v1x2}
\right).
$$
\end{thm}
\begin{proof}
The modulus is $\sqrt{(5+\sqrt{13})/2}\simeq 2.07431$, and we find that the minimal projections one past the branch are given by $\frac{1}{2}(\jw{4}+A)$ and $\frac{1}{2}(\jw{4}-A)$.
Since $\Tr(f^{(4)})=3+\sqrt{13}$, both minimal projections have trace $\frac{1}{2} (3 + \sqrt{13})$, which agree with the Frobenius-Perron dimensions of the vertices at depth 4 of
$$
\bigraph{gbg1v1v1v1p1v1x0p0x1v1x0p0x1}\,.
$$
Hence each vertex at depth 4 on the principal graph must be attached to a vertex at depth 5. Since
\begin{align*}
\left\|
\bigraph{gbg1v1v1v1p1v1x1}
\right\|
&\simeq
2.1889
\\
\left\|
\bigraph{gbg1v1v1v1p1v2x0}
\right\|
&\simeq
2.38098
\\
\left\|
\bigraph{gbg1v1v1v1p1v1x0p0x1p0x1}
\right\|
&\simeq
2.11917,
\end{align*}
we know each branch must continue simply. Again by analyzing the Frobenius-Perron dimensions, we know each branch must continue, and since
\begin{align*}
\left\|
\bigraph{gbg1v1v1v1p1v1x0p0x1v1x1}
\right\|
&\simeq
2.14896
\\
\left\|
\bigraph{gbg1v1v1v1p1v1x0p0x1v2x0}
\right\|
&\simeq
2.34554
\\
\left\|
\bigraph{gbg1v1v1v1p1v1x0p0x1v0x1p0x1}
\right\|
&\simeq
2.101,
\end{align*}
we know each branch must continue simply. We conclude that the principal graph is as claimed. 

The dual graph is correct by a number of arguments, including Ocneanu's triple point obstruction \cite{MR1317352,MR2902285}, or Jones' quadratic tangles formula \cite{MR2972458}.

To determine the dual data, note that the projections at depth 4 on the principal graph are self-dual since $\rho^2=\id$ on $\spann\{A,\jw{4}\}$. The projections at depth 6 must be dual to each other since the dimension 1 bimodules form a group.
Note that there is only one possibility for the dual data for the dual graph. 
\end{proof}

\subsection{$3^{\Z/2\times \Z/2}$}

\begin{thm}\label{thm:Z2Z2graph}
The principal graphs of $\cP_\bullet^{\Z/2\times \Z/2}$ are 
$$
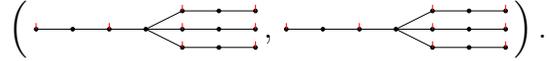
\left(
\bigraph{bwd1v1v1v1p1p1v1x0x0p0x1x0p0x0x1v1x0x0p0x1x0p0x0x1duals1v1v1x2x3v1x2x3},
\bigraph{bwd1v1v1v1p1p1v1x0x0p0x1x0p0x0x1v1x0x0p0x1x0p0x0x1duals1v1v1x2x3v1x2x3}
\right).
$$
\end{thm}
\begin{proof}
Our generators $A,B$ are obtained from the generators $A_0,B_0$ from \cite{1208.3637} (see Appendix \ref{generators:3333}), for which the proof is given by \cite[Theorems 5.3 and 5.9]{1208.3637}.
\end{proof}

\subsection{$3^{\Z/4}$}

\begin{thm}\label{thm:Z4graph}
The principal graphs of $\cP_\bullet^{\Z/4}$ are 
$$
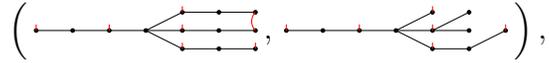
\left(
\bigraph{bwd1v1v1v1p1p1v1x0x0p0x1x0p0x0x1v1x0x0p0x1x0p0x0x1duals1v1v1x2x3v1x3x2},
\bigraph{bwd1v1v1v1p1p1v1x0x0p0x1x0p0x1x0v1x0x0duals1v1v1x2x3v1}
\right),
$$
\end{thm}
\begin{proof}
The modulus is $\sqrt{3+\sqrt{5}}\simeq2.28825$, and we find that the minimal projections one past the branch from bottom to top are given by $aA +bB+c\jw{4}$ where
$$
(a,b,c)=
\begin{cases}
\displaystyle\left(0, \frac{1}{3},\frac{1}{3}\right)\\
\displaystyle\left( \frac{1}{2}, -\frac{1}{6},\frac{1}{3}\right)\\
\displaystyle\left(-\frac{1}{2}, -\frac{1}{6},\frac{1}{3}\right).
\end{cases}
$$
Since $\Tr(f^{(4)})=6+3 \sqrt{5}$, all the minimal projections have trace $2+\sqrt{5}$, and the proof of \cite[Theorem 5.9]{1208.3637} shows the principal graph is correct.

To see that the dual graph is correct, we first find that the minimal projections one past the branch from bottom to top are given by $a\check{A} +b\check{B}+c\jw{4}$ where
$$
(a,b,c)=
\begin{cases}
\displaystyle\left(\lambda_{4,0,2,0,-1}^{(-0.556)},\lambda_{324,0,-126,0,1}^{(0.09003)},\frac{1}{3}\right)\\
\displaystyle\left(\lambda_{4,0,22,0,-1}^{(0.2123)},\lambda_{324,0,-270,0,25}^{(-0.3257)},\frac{1}{3} \left(\sqrt{5}-1\right)\right)\\
\displaystyle\left(\lambda_{4,0,8,0,-1}^{(0.3436)},\frac{1}{3 \sqrt{2}},\frac{1}{3} \left(3-\sqrt{5}\right)\right)
\end{cases}
$$
which have traces $2+\sqrt{5},3+\sqrt{5},1+\sqrt{5}$ respectively. Hence there is a univalent vertex at depth $4$ on the dual graph. We now run the {\texttt{FusionAtlas}} program {\tt{FindGraphPartners}} on the 3333 graph and we see there are only two possibilities where the dual graph has a univalent vertex at depth 4:
$$
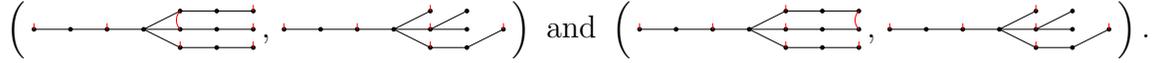
\left(
\bigraph{bwd1v1v1v1p1p1v1x0x0p0x1x0p0x0x1v1x0x0p0x1x0p0x0x1duals1v1v1x3x2v1x2x3},
\bigraph{bwd1v1v1v1p1p1v1x0x0p0x1x0p0x1x0v1x0x0duals1v1v1x2x3v1}
\right)\text{ and }
\left(
\bigraph{bwd1v1v1v1p1p1v1x0x0p0x1x0p0x0x1v1x0x0p0x1x0p0x0x1duals1v1v1x2x3v1x3x2},
\bigraph{bwd1v1v1v1p1p1v1x0x0p0x1x0p0x1x0v1x0x0duals1v1v1x2x3v1}
\right).
$$
Now the projections at depth 4 on the principal graph are self-dual since $\rho^2=\id$ on $\spann\{A,B,\jw{4}\}$, so the only possibility is the one claimed.
\end{proof}

\appendix
\section{Generators}\label{sec:Generators}

Suppose $\Gamma$ is a simply laced graph with a distinguished subgraph $\Lambda\subset \Gamma$ such that $\Gamma$ is obtained from $\Lambda$ by adding $A_{finite}$ tails to $\Lambda$. 
For example, when $\Gamma$ is a spoke graph, we can choose $\Lambda$ to be the central vertex. 
When $\Gamma=$2D2 (see Section \ref{generators:3333ZMod4}), we can choose $\Lambda$ to be the central diamond.

By the proof of \cite[Lemma A.1]{1208.3637}, a low-weight generator $A$ is completely determined by its values on loops which stay within distance 1 of $\Lambda$. 
Furthermore, if $\Gamma$ is obtained from $\Lambda$ by adding $A_{finite}$ tails to \underline{distinct} vertices of $\Lambda$, then $A$ is completely determined by its values on loops which stay inside $\Lambda$. 
So when $\Gamma$ is a spoke graph with $n$ spokes, we can choose $\Lambda$ to be an $(n-1)$-star.

Moreover, as $A$ is a rotational eigenvector, $A$ is completely determined by its values on a set of rotation orbit representatives which stay in $\Lambda$.

We now describe an algorithm to recover our low-weight generator $A$ from its values on such loops. 

\begin{rem}
It should seem plausible, but not at all obvious, that the recovered generator is in fact a low-weight rotational eigenvector. 
Proposition \ref{prop:WellDefinedExtension} gives a well-defined element of the graph planar algebra. 
For our examples, we check using the programs  {\tt CheckLowestWeightCondition} and {\tt CheckRotationalEigenvector} in the notebook {\tt Generators.nb} that the low-weight and rotational eigenvector conditions hold respectively.
\end{rem}

\begin{defn}
For a vertex $v\in \Gamma$, we define $d(v,\Lambda)$ to be the minimal distance of $v$ to $\Lambda$.
For a loop $\gamma$ whose $i$-th vertex is denoted $\gamma(i)$, we define $d(\gamma,\Lambda)=\max_{i} d(\gamma(i),\Lambda)$.
\end{defn}

\paragraph{2-valent folding relation.\\}

Suppose $A$ is an $n$-box.
We start with a loop $\gamma$ on $\Gamma$ of length $2n$. 
If $d(\gamma,\Lambda)>1$, we can use the 2-valent relation first considered in \cite{MR2679382,MR2979509} to fold $\gamma$ inward by analyzing the capping action on 2-valent vertices as follows.
We use the notation of \cite{1208.3637}.

\begin{nota}
Suppose $s = \gamma(i)$ is a vertex on $\gamma$ which is distance at least 2 from $\Lambda$.
Let $t$ be the vertex on the same tail 2 closer to $\Lambda$ than $s$ (possibly $t$ is in $\Lambda$ itself).
Let $\gamma'$ be the loop modified from $\gamma$ by replacing $s$ at position $i$ with $t$.
Let $\pi$ be the `snipped' loop of length $2n-2$ obtained from $\gamma$ or $\gamma'$ by removing the $i$-th and $i+1$-th positions.
For convenience, we let $r = \gamma(i\pm1)=\gamma'(i\pm 1)$.
\end{nota}

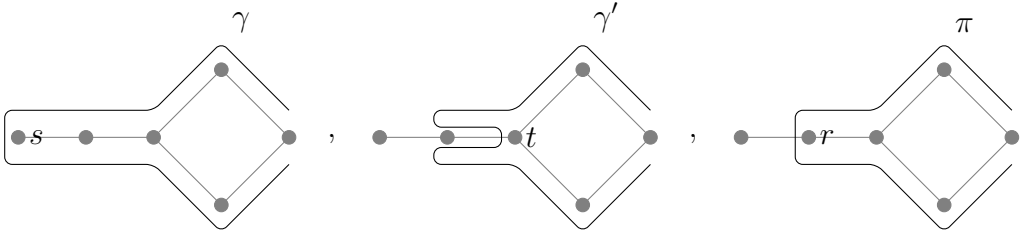
\begin{figure}[!htb]
$$
\begin{tikzpicture}[baseline=0,scale=.9]
	\filldraw[gray] (0,0) circle (1mm);
	\filldraw[gray] (1,0) circle (1mm);
	\filldraw[gray] (2,0) circle (1mm);
	\filldraw[gray] (3,-1) circle (1mm);
	\filldraw[gray] (3,1) circle (1mm);
	\filldraw[gray] (4,0) circle (1mm);
	\draw[gray] (0,0)--(2,0)--(3,1)--(4,0);
	\draw[gray] (2,0)--(3,-1)--(4,0);
	
	\draw[rounded corners=1mm] (4,.4)--(3,1.4)--(2,.4)--(-.2,.4)--(-.2,-.4)--(2,-.4)--(3,-1.4)--(4,-.4);

	\node[above right] at (3,1.4) {$\gamma$};
	\node[right] at (0,0) {$s$};

\end{tikzpicture} \quad , \quad
\begin{tikzpicture}[baseline=0,scale=.9]
	\filldraw[gray] (0,0) circle (1mm);
	\filldraw[gray] (1,0) circle (1mm);
	\filldraw[gray] (2,0) circle (1mm);
	\filldraw[gray] (3,-1) circle (1mm);
	\filldraw[gray] (3,1) circle (1mm);
	\filldraw[gray] (4,0) circle (1mm);
	\draw[gray] (0,0)--(2,0)--(3,1)--(4,0);
	\draw[gray] (2,0)--(3,-1)--(4,0);
	
	\draw[rounded corners=1mm] (4,.4)--(3,1.4)--(2,.4)--(.8,.4) -- (.8,.15)--(1.8,.15)--(1.8,-.15)--(.8,-.15)--(.8,-.4)--(2,-.4)--(3,-1.4)--(4,-.4);

	\node[above right] at (3,1.4) {$\gamma'$};
	\node[right] at (2,0) {$t$};

\end{tikzpicture}
 \quad , \quad
\begin{tikzpicture}[baseline=0,scale=.9]
	\filldraw[gray] (0,0) circle (1mm);
	\filldraw[gray] (1,0) circle (1mm);
	\filldraw[gray] (2,0) circle (1mm);
	\filldraw[gray] (3,-1) circle (1mm);
	\filldraw[gray] (3,1) circle (1mm);
	\filldraw[gray] (4,0) circle (1mm);
	\draw[gray] (0,0)--(2,0)--(3,1)--(4,0);
	\draw[gray] (2,0)--(3,-1)--(4,0);
	
	\draw[rounded corners=1mm] (4,.4)--(3,1.4)--(2,.4)--(.8,.4) --(.8,-.4)--(2,-.4)--(3,-1.4)--(4,-.4);

	\node[above right] at (3,1.4) {$\pi$};
	\node[right] at (1,0) {$r$};

\end{tikzpicture}
$$
\caption{Example of loops and vertices appearing in the 2-valent folding relation.}
\label{fig:SnippedLoops}
\end{figure}

\begin{defn}
Applying a cap at position $i$ to $A$, we have $\cap_i (A) = 0$. Evaluating this at $\pi$ gives the \underline{2-valent folding relation}
\begin{align*}
0 = \sqrt{\dim(r)}^{k_i} \;{\cap_i(A)(\pi)} & = \sqrt{\dim(s)}^{k_i} A(\gamma) + \sqrt{\dim(t)}^{k_i} A(\gamma').
\end{align*}
Here $k_i$ is the number of critical points in the cap strand, either $1$ or $2$ depending on the position of the point $i$ around the boundary of the rectangular box, as follows
\begin{equation*}
k_i = \begin{cases}
1 & \text{when we have $\tikz[baseline=7]{\draw[fill=gray] (0,0) rectangle (1,0.5); \draw (0.3,0.5) arc (180:90:0.2) node[above left] {$i$} arc(90:0:0.2);}$ or $\tikz[baseline=0]{\draw[fill=gray] (0,0) rectangle (1,0.5); \draw (0.3,0) arc (180:270:0.2) node[right=6pt] {$i$} arc(270:360:0.2);}$} \\
2 & \text{when we have $\tikz[baseline=4]{\draw[fill=gray] (0,0) rectangle (1,0.5); \draw (0.8,0.5) arc (180:0:0.2) -- node[right] {$i$} (1.2,0) arc(0:-180:0.2);}$ or $\tikz[baseline=4]{\draw[fill=gray] (0,0) rectangle (1,0.5); \draw (0.2,0.5) arc (0:180:0.2) -- node[left] {$i$} (-0.2,0) arc(180:360:0.2);}$}
\end{cases}
\end{equation*}
(The case $k_i=3$, which occurs for the following caps
$$
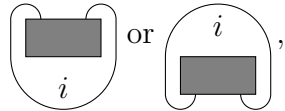
\tikz[baseline=5]{\path (0,0) -- (0,1); \draw[fill=gray] (0,0) rectangle (1,0.5); \draw (0.2,0.5) arc (0:180:0.2) -- (-0.2,0) arc(180:270:0.7) node[above] {$i$} arc (270:360:0.7) -- (1.2,0.5) arc (0:180:0.2);}
\text{ or }
\tikz[y=-1cm,baseline=5]{\path (0,0) -- (0,1); \draw[fill=gray] (0,0) rectangle (1,0.5); \draw (0.2,0.5) arc (0:180:0.2) -- (-0.2,0) arc(180:270:0.7) node[below] {$i$} arc (270:360:0.7) -- (1.2,0.5) arc (0:180:0.2);}\,,
$$
never occurs for us as we always consider boxes with equal numbers of strands above and below.)
\end{defn}

\begin{lem}\label{lem:collapsing-spherical}
If $\hat{\gamma}$ is the loop of length $2n$ with $d(\hat{\gamma},\Lambda)=1$ obtained from $\gamma$ by the 2-valent folding relation described above, we have
\begin{equation}
\label{eq:collapsing-spherical}
A(\gamma) = (-1)^{(\norm{\gamma}-\norm{\hat{\gamma}})/2} \left( \prod_{i} \sqrt{\frac{\dim(\hat{\gamma}(i))}{\dim(\gamma(i))}}^{k_i} \right) A(\hat{\gamma}).
\end{equation}
where $\norm{\gamma}=\sum_{i} d(\gamma(i),\Lambda)$.
\end{lem}

\begin{rem}\label{rem:collapsing-lopsided}
In the lopsided convention, this formula is given by
\begin{equation}
\label{eq:collapsing-lopsided}
A(\gamma) = (-1)^{(\norm{\gamma}-\norm{\hat{\gamma}})/2} \left( \prod_{i} \left(\frac{\dim(\hat{\gamma}(i))}{\dim(\gamma(i))}\right)^{\ell_i} \right) A(\hat{\gamma})
\end{equation}
where $\ell_i$ is the number of minima on the cap:
\begin{equation*}
\ell_i = \begin{cases}
0 & \text{when we have $\tikz[baseline=7]{\draw[fill=gray] (0,0) rectangle (1,0.5); \draw (0.3,0.5) arc (180:90:0.2) node[above left] {$i$} arc(90:0:0.2);}$} \\
1 & \text{when we have $\tikz[baseline=0]{\draw[fill=gray] (0,0) rectangle (1,0.5); \draw (0.3,0) arc (180:270:0.2) node[right=6pt] {$i$} arc(270:360:0.2);}$, $\tikz[baseline=4]{\draw[fill=gray] (0,0) rectangle (1,0.5); \draw (0.8,0.5) arc (180:0:0.2) -- node[right] {$i$} (1.2,0) arc(0:-180:0.2);}$, or $\tikz[baseline=4]{\draw[fill=gray] (0,0) rectangle (1,0.5); \draw (0.2,0.5) arc (0:180:0.2) -- node[left] {$i$} (-0.2,0) arc(180:360:0.2);}$}\,.
\end{cases}
\end{equation*}
\end{rem}

\paragraph{Tail avoiding relation.\\}
Now suppose $\Gamma$ is obtained from $\Lambda$ by adding $A_{finite}$ tails to distinct vertices of $\Lambda$.
Further suppose $\gamma$ is a loop of length $2n$ with $d(\gamma,\Lambda)=1$.

\begin{nota}
Suppose $s = \gamma(i)$ is a vertex on $\gamma$ which is distance 1 from $\Lambda$, and let $r=\gamma(i+1)$ which is necessarily in $\Lambda$.
Let $\{t\}$ be the set of vertices in $\Lambda$ incident to $r$.
Let $\gamma_{i,t}$ be the loop modified from $\gamma$ by replacing $s$ at position $i$ with $t$.
Let $\pi$ be the `snipped' loop of length $2n-2$ obtained from $\gamma$ or $\gamma_{i,t}$ by removing the $i$-th and $i+1$-th positions.
\end{nota}

\begin{defn}
The \underline{tail avoiding relation} is given by:
$$
0 = \sqrt{\dim(r)}^{k_i} \;{\cap_i(A)(\pi)}  =  \sqrt{\dim(s)}^{k_i} A(\gamma)+\sum_{t} \sqrt{\dim(t)}^{k_i} A(\gamma_{i,t}).
$$
\end{defn}

\begin{lem}\label{lem:TailAvoiding-spherical}
If $\gamma$ has $d(\gamma,\Lambda)=1$, and $\hat{\gamma}$ has $d(\hat{\gamma},\Lambda)=0$ and is obtained from $\gamma$ by the tail avoiding relation described above, then
\begin{equation}
\label{eq:TailAvoiding-spherical}
A(\gamma)
=
(-1)^{\norm{\gamma}}
\sum_{\set{i}{\gamma(i)\notin\Lambda}}
\sum_{\set{t_i}{{\substack{t_i\sim \gamma(i\pm 1)\\ t_i\in\Lambda}}}} 
\sqrt{\frac{\dim(t_i)}{\dim(\gamma(i))}}^{k_i} A(\gamma_{i,t_i})
\end{equation}
where $v\sim w$ means $v$ is incident to $w$ (note $\gamma(i+1)=\gamma(i-1)$ if $\gamma(i)\notin \Lambda$), and $k_i$ is as in Lemma \ref{lem:collapsing-spherical}.
\end{lem}

\begin{rem}
In the lopsided convention, this formula is given by
\begin{equation}
\label{eq:TailAvoiding-lopsided}
A(\gamma)
=
(-1)^{\norm{\gamma}}
\sum_{\set{i}{\gamma(i)\notin\Lambda}}
\sum_{\set{t_i}{{\substack{t_i\sim \gamma(i\pm 1)\\ t_i\in\Lambda}}}} 
\left(\frac{\dim(t_i)}{\dim(\gamma(i))}\right)^{\ell_i} A(\gamma_{i,t_i})
\end{equation}
using similar notation from Remark \ref{rem:collapsing-lopsided} and Lemma \ref{lem:TailAvoiding-spherical}.
\end{rem}

\paragraph{Rotation.\\}
We still assume $\Gamma$ is obtained from $\Lambda$ by adding $A_{finite}$ tails to distinct vertices of $\Lambda$.

Rotation acts on the set of loops which stay in $\Lambda$, so if we are trying to specify a lowest weight vector $A$ which is also a rotational eigenvector corresponding to eigenvalue $\omega$, then it suffices to specify $A$ only on a representative of each such orbit.

\begin{prop}\label{prop:WellDefinedExtension}
Let $S$ be a set of representatives of each rotation orbit of loops of length $2n$ in $\Lambda$.
Let $A_0:S\to \C$. For a loop $\gamma$ of length $2n$ in $\Lambda$, let $[\gamma]$ be its representative in $S$.
Suppose that whenever $\gamma'\in S$ is fixed by the $k$-fold rotation, and $\omega^k \neq 1$, then $A_0(\gamma')=0$. 
Then there is a well-defined function $A_1$ on the loops of length $2n$ in $\Lambda$ such that $A_1|_S=A_0$.

Moreover, there is a well-defined element $A \in\cP\cA(\Gamma)_n$ such that the values of $A$ on the loops of length $2n$ on $\Lambda$ is equal to $A_1$.
\end{prop}
\begin{proof}
Suppose $\gamma$ is a loop of length $2n$ which stays in $\Lambda$, and $\rho^{-j}(\gamma)=[\gamma]$ for some $j=0,\dots, n-1$. If $j\leq n/2$,
$$
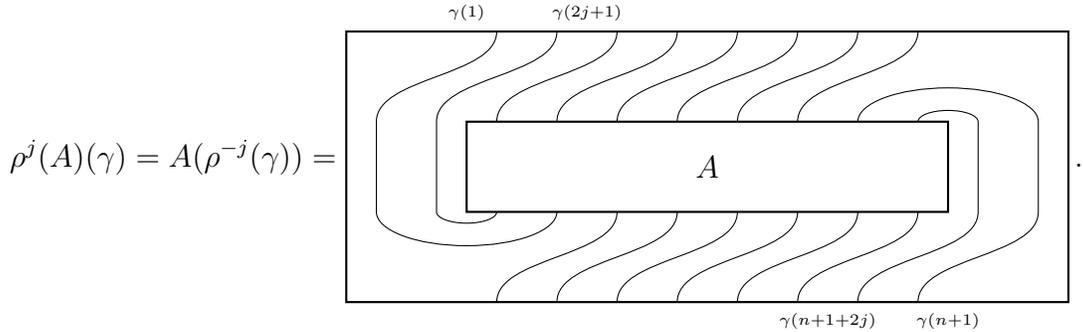
\rho^{j}(A)(\gamma)=A(\rho^{-j}(\gamma))=
\begin{tikzpicture}[xscale=.8,yscale=.3, baseline]
	\draw[] (0,-2) arc (0:-180:.5cm) -- (-1,2) .. controls ++(90:2cm) and ++(-90:2cm) .. (1,6) -- (0,6) .. controls ++(-90:2cm) and ++(90:2cm) .. (-2,2) -- (-2,-2) arc (-180:0:1.5cm);
	\foreach \x in {0,2,4} \draw[] (\x cm,2 cm) .. controls ++(90:2cm) and ++(-90:2cm) .. (\x cm + 2 cm,6 cm) -- (\x cm + 3 cm, 6 cm) .. controls ++(-90:2cm) and ++(90:2cm) .. (\x cm + 1 cm,2 cm);

\begin{scope}[rotate=180, xshift=-7cm]
		\draw[] (0,-2) arc (0:-180:.5cm) -- (-1,2) .. controls ++(90:2cm) and ++(-90:2cm) .. (1,6) -- (0,6) .. controls ++(-90:2cm) and ++(90:2cm) .. (-2,2) -- (-2,-2) arc (-180:0:1.5cm);
		\foreach \x in {0,2,4} \draw[] (\x cm,2 cm) .. controls ++(90:2cm) and ++(-90:2cm) .. (\x cm + 2 cm,6 cm) -- (\x cm + 3 cm, 6 cm) .. controls ++(-90:2cm) and ++(90:2cm) .. (\x cm + 1 cm,2 cm);
\end{scope}

	\draw[thick] (-.5,-2) rectangle (7.5,2);
	\draw[thick] (-2.5,-6) rectangle (9.5,6);
	\node at (3.5,0) {$A$};
	
	\node at (1 cm - 1.5 cm, 6cm) [above] {\tiny{$\gamma(1)$}};
	\node at (3 cm - 1.5 cm, 6cm) [above] {\tiny{$\gamma(2j{+}1)$}};
	\node at (16.5 cm - 9 cm, -6 cm) [below] {\tiny {$\gamma(n{+}1)$}};
	\node at (16.5 cm - 11 cm, -6 cm) [below] {\tiny {$\gamma(n{+}1{+}2j)$}};
\end{tikzpicture}\,.
$$
Hence for all $j=0,\dots, n-1$, we define
\begin{equation}
\label{eq:rotation-spherical}
A_1(\gamma) = \omega^{-j} \sqrt{\frac{\dim(\gamma(2j+1))\dim(\gamma(n+2j+1))}{\dim(\gamma(1))\dim(\gamma(n+1))}}  A_0([\gamma]),
\end{equation}
modulo some modular arithmetic, namely $\gamma(b)=\gamma(b\mod 2n)$.

In the lopsided convention, the above equation is given by
\begin{equation}
\label{eq:rotation-lopsided}
A_1(\gamma) = \omega^{-j} 
\left(\prod_{k=1}^{2j} \frac{\dim(\gamma(1+k))}{\dim(\gamma(n+k))}\right)
A_0([\gamma]).
\end{equation}

We now define $A\in\cP\cA(\Gamma)_n$ as follows. First, for loops $\gamma$ of length $2n$ which stay in $\Lambda$, define $A(\gamma)=A_1(\gamma)$. Next, we define $A$ on loops $\gamma$ of length $2n$ for which $d(\gamma,\Lambda)=1$ by Lemma \ref{lem:TailAvoiding-spherical}. Finally, we define $A$ on loops $\gamma$ of length $2n$ for which $d(\gamma,\Lambda)>1$ by Lemma \ref{lem:collapsing-spherical}.
\end{proof}

\begin{rem}
When $\Gamma$ is a simply laced spoke graph, \cite[Appendix A]{1208.3637} gives a necessary and sufficient condition for a list of values on rotation orbit representatives to give a well-defined low-weight rotation eigenvector. 
We do not give such a condition for general $\Gamma$.

Moreover, for spoke graphs $\Gamma$, the value of $A$ on any loop which has more than $2k+1$ consecutive vertices which either lie on a particular arm of the graph of length $k-1$ or are the central vertex is zero \cite[Lemma A.1 (3)]{1208.3637}. We omit these values in our lists below for $\Gamma^{\Z/3}$ and $\Gamma^{\Z/2\times\Z/2}$.
\end{rem}

In the subsections that follow for each of our graph planar algebras, we use the discussion above to specify our generators by their values on a certain collection of loops. 
For $\Gamma^{\Z/3}$ and $\Gamma^{\Z/2\times\Z/2}$, we denote the value of $A$ on the collapsed loop which successively visits
arms $a_1,a_2,\dots, a_n$ by $A(a_1a_2\cdots a_n)$.
For $\Gamma^{\Z/4}$, we find our generators in the graph planar algebra of $\Gamma=2D2$, which has a central diamond. We label the vertices on the diamond by $W,S,E,N$, which stand for ``west," ``south," ``east," ``north" respectively. We denote the value of $A$ on the collapsed loop which stays inside the central diamond by $A(w)$, where $w$ is a word on $\{W,S,E,N\}$.

\subsection{$3^{\Z/3}$: Haagerup}

We work in the graph planar algebra of
$$
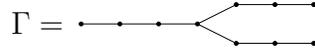
\Gamma=\bigraph{gbg1v1v1v1p1v1x0p0x1v1x0p0x1}
$$
where $\Lambda$ consists of three vertices and two edges: the central vertex, both vertices to the right of the central vertex, and the edges connecting them.

The self-adjoint generator $A$ for $\cP^{\Z/3}_\bullet$ has chiralities $\omega_A=-1$ and $\sigma_A=i$.
The values of $A$ on rotation orbit representatives of collapsed loops are as follows:
\begin{align*}
A(1112) & = \frac{1}{3} \left(4-\sqrt{13}\right) &
A(1122) & = \lambda_{3,0,-25,0,-1}^{(0.1995 i)} \displaybreak[1]\\
A(1212) & = \frac{1}{3} \left(\sqrt{13}-4\right) &
A(1222) & = \frac{1}{3} \left(4-\sqrt{13}\right)
\end{align*}
These entries lie in $\Q(\mu_{\Z/3})$, where $\mu_{\Z/3}$ is the root of
$$
x^8+6 x^6+3 x^4+34 x^2+9
$$
which is approximately $2.52 i$.

\subsection{$3^{\Z/2\times \Z/2}$}\label{generators:3333}

We work in the graph planar algebra of
$$
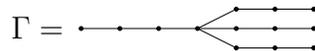
\Gamma=\bigraph{gbg1v1v1v1p1p1v1x0x0p0x1x0p0x0x1v1x0x0p0x1x0p0x0x1}
$$
where $\Lambda$ consists of four vertices and three edges: the central vertex, the vertices to the right of the central vertex, and the edges connecting them.

\begin{prop}
Let $A_0,B_0$ be the $\cP_\bullet^{\Z/2\times \Z/2}$ generators from \cite{1208.3637}. Let $P,Q,R$ be the minimal projections at depth 4 given by the linear combinations
\begin{align*}
P &=\frac{1}{4} \left(-1+\sqrt{5}\right)A_0+  -\frac{\sqrt{5}}{6}B_0+\frac{1}{3}\jw{4}\\
Q &=-\frac{1}{2}A_0 + \frac{1}{12} \left(-3+\sqrt{5}\right)B_0+\frac{1}{3}\jw{4}\\
R &=\frac{1}{4} \left(3-\sqrt{5}\right)A_0+ \frac{1}{12} \left(3+\sqrt{5}\right)B_0+\frac{1}{3}\jw{4}
\end{align*}
obtained in the proof of \cite[Theorem 5.9]{1208.3637}.
Then $P-Q,Q-R,R-P$ are all uncappable rotational low-weight vectors with eigenvalue $1$.
\end{prop}
\begin{proof}
It is clear that $P-Q,Q-R,R-P\in\spann\{A_0,B_0\}$, which is the low-weight space associated to the rotational eigenvalue $1$.
\end{proof}

Given the extremely simple formula for these low-weight vectors, we will use different generators for this article than those given in \cite{1208.3637}. We work with $A=P-Q$ and $B=2R-(P+Q)$, which have chiralities 
$$\omega_A=\sigma_A=\omega_B=\sigma_B=1.$$
Their values on rotation orbit representatives of collapsed loops are as follows:
\begin{align*}
A(1112) & = \frac{1}{4} \left(3 \sqrt{5}-7\right) &
A(1113) & = 0 \displaybreak[1]\\
A(1122) & = \frac{1}{2} \left(2-\sqrt{5}\right) &
A(1123) & = 0 \displaybreak[1]\\
A(1132) & = 0 &
A(1133) & = \frac{1}{4} \left(3-\sqrt{5}\right) \displaybreak[1]\\
A(1212) & = \frac{1}{2} \left(3-\sqrt{5}\right) &
A(1213) & = \frac{1}{4} \left(1-\sqrt{5}\right) \displaybreak[1]\\
A(1222) & = \frac{1}{4} \left(3 \sqrt{5}-7\right) &
A(1223) & = \frac{1}{4} \left(3-\sqrt{5}\right) \displaybreak[1]\\
A(1232) & = 0 &
A(1233) & = \frac{1}{4} \left(\sqrt{5}-3\right) \displaybreak[1]\\
A(1313) & = 0 &
A(1322) & = \frac{1}{4} \left(3-\sqrt{5}\right) \displaybreak[1]\\
A(1323) & = \frac{1}{4} \left(\sqrt{5}-1\right) &
A(1332) & = \frac{1}{4} \left(\sqrt{5}-3\right) \displaybreak[1]\\
A(1333) & = 0 &
A(2223) & = \frac{1}{4} \left(7-3 \sqrt{5}\right) \displaybreak[1]\\
A(2233) & = \frac{1}{4} \left(3 \sqrt{5}-7\right) &
A(2323) & = \frac{1}{2} \left(\sqrt{5}-3\right) \displaybreak[1]\\
A(2333) & = \frac{1}{4} \left(7-3 \sqrt{5}\right)
\end{align*}\begin{align*}
B(1112) & = \frac{1}{4} \left(7-3 \sqrt{5}\right) &
B(1113) & = \frac{1}{2} \left(3 \sqrt{5}-7\right) \displaybreak[1]\\
B(1122) & = \frac{1}{2} \left(2 \sqrt{5}-5\right) &
B(1123) & = \frac{1}{2} \left(3-\sqrt{5}\right) \displaybreak[1]\\
B(1132) & = \frac{1}{2} \left(3-\sqrt{5}\right) &
B(1133) & = \frac{1}{4} \left(11-5 \sqrt{5}\right) \displaybreak[1]\\
B(1212) & = \frac{1}{2} \left(\sqrt{5}-3\right) &
B(1213) & = \frac{1}{4} \left(1-\sqrt{5}\right) \displaybreak[1]\\
B(1222) & = \frac{1}{4} \left(7-3 \sqrt{5}\right) &
B(1223) & = \frac{1}{4} \left(\sqrt{5}-3\right) \displaybreak[1]\\
B(1232) & = \frac{1}{2} \left(\sqrt{5}-1\right) &
B(1233) & = \frac{1}{4} \left(\sqrt{5}-3\right) \displaybreak[1]\\
B(1313) & = 3-\sqrt{5} &
B(1322) & = \frac{1}{4} \left(\sqrt{5}-3\right) \displaybreak[1]\\
B(1323) & = \frac{1}{4} \left(1-\sqrt{5}\right) &
B(1332) & = \frac{1}{4} \left(\sqrt{5}-3\right) \displaybreak[1]\\
B(1333) & = \frac{1}{2} \left(3 \sqrt{5}-7\right) &
B(2223) & = \frac{1}{4} \left(7-3 \sqrt{5}\right) \displaybreak[1]\\
B(2233) & = \frac{1}{4} \left(\sqrt{5}-1\right) &
B(2323) & = \frac{1}{2} \left(\sqrt{5}-3\right) \displaybreak[1]\\
B(2333) & = \frac{1}{4} \left(7-3 \sqrt{5}\right)
\end{align*}
Clearly these entries lie in $\Q(\sqrt{5})$.

\subsection{$3^{\Z/4}$}\label{generators:3333ZMod4}
In an unpublished manuscript, Izumi constructs a $3^{\Z/4}$ subfactor with principal graphs
$$
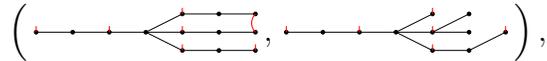
\left(
\bigraph{bwd1v1v1v1p1p1v1x0x0p0x1x0p0x0x1v1x0x0p0x1x0p0x0x1duals1v1v1x2x3v1x3x2},
\bigraph{bwd1v1v1v1p1p1v1x0x0p0x1x0p0x1x0v1x0x0duals1v1v1x2x3v1}
\right),
$$
and he claims there is a de-equivariantization, giving a subfactor with principal graph 2D2 (``2-diamond-2")
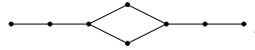
$$
\bigraph{gbg1v1v1p1v1x1v1v1}.
$$

In an independent calculation, the authors along with Scott Morrison have verified the existence of a 2D2 subfactor with principal graphs
$$
\text{2D2}
=
\left(
\bigraph{bwd1v1v1p1v1x1v1v1duals1v1v1v1},
\bigraph{bwd1v1v1p1v1x0p1x0p0x1p0x1v0x1x1x0duals1v1v1x3x2x4}
\right).
$$
We first find a bi-unitary connection on 2D2, and we find a flat generator with respect to that connection. 
We then verify that the resulting subfactor planar algebra has principal graphs 2D2. 
We obtain a $3^{\Z/4}$ subfactor planar algebra as an equivariantization of the 2D2 subfactor planar algebra. 
Note that 2D2 has annular multiplicities $*12$, so the $3^{\Z/4}$ generators must be the new low-weight vectors at depth 4. 
More details on this will appear in \cite{4442equi}. 

For our purposes, we do \underline{not} rely on the fact that our generators were obtained in this manner. 
Rather, we give candidate generators for $3^{\Z/4}$, show they satisfy Assumptions \ref{assume:Generators}, \ref{assume:SpanAlgebras}, and \ref{assume:Tetrahedral}, and use our formulas to show they generate an evaluable planar subalgebra of the graph planar algebra of 2D2, i.e., a subfactor planar algebra.

Hence we work in the graph planar algebra of 
$$
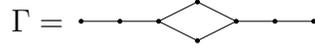
\Gamma=\bigraph{gbg1v1v1p1v1x1v1v1}
$$
where $\Lambda$ is the central diamond.

The self-adjoint generators $A,B$ for $\cP^{\Z/4}_\bullet$ have chiralities $\omega_A=-1$ and $\sigma_A=i$ and $\omega_B=\sigma_B=1$.
Their values on rotation orbit representatives of loops which remain in $\Lambda$ are as follows:

\begin{align*}
A(\mbox{WSWSWSWS}) & = 0 &
A(\mbox{WSWSWSWN}) & = \lambda_{16,0,-116,0,-1}^{(-0.09279 i)} \displaybreak[1]\\
A(\mbox{WSWSWSES}) & = 0 &
A(\mbox{WSWSWSEN}) & = \frac{1}{4} \left(3-\sqrt{5}\right) \displaybreak[1]\\
A(\mbox{WSWSWNWN}) & = 0 &
A(\mbox{WSWSWNES}) & = \frac{1}{4} \left(3-\sqrt{5}\right) \displaybreak[1]\\
A(\mbox{WSWSWNEN}) & = 0 &
A(\mbox{WSWSESWN}) & = \lambda_{16,0,-116,0,-1}^{(0.09279 i)} \displaybreak[1]\\
A(\mbox{WSWSESES}) & = 0 &
A(\mbox{WSWSESEN}) & = \frac{1}{4} \left(\sqrt{5}-3\right) \displaybreak[1]\\
A(\mbox{WSWSENWN}) & = 0 &
A(\mbox{WSWSENES}) & = \lambda_{16,0,-116,0,-1}^{(0.09279 i)} \displaybreak[1]\\
A(\mbox{WSWSENEN}) & = 0 &
A(\mbox{WSWNWSWN}) & = 0 \displaybreak[1]\\
A(\mbox{WSWNWSES}) & = \lambda_{16,0,-116,0,-1}^{(0.09279 i)} &
A(\mbox{WSWNWSEN}) & = 2-\sqrt{5} \displaybreak[1]\\
A(\mbox{WSWNWNWN}) & = \lambda_{16,0,-116,0,-1}^{(0.09279 i)} &
A(\mbox{WSWNWNES}) & = 0 \displaybreak[1]\\
A(\mbox{WSWNWNEN}) & = \lambda_{16,0,-116,0,-1}^{(-0.09279 i)} &
A(\mbox{WSWNESWN}) & = \sqrt{5}-2 \displaybreak[1]\\
A(\mbox{WSWNESES}) & = \frac{1}{4} \left(3-\sqrt{5}\right) &
A(\mbox{WSWNESEN}) & = \lambda_{1,0,-11,0,-1}^{(-0.3003 i)} \displaybreak[1]\\
A(\mbox{WSWNENWN}) & = \lambda_{16,0,-116,0,-1}^{(-0.09279 i)} &
A(\mbox{WSWNENES}) & = 0 \displaybreak[1]\\
A(\mbox{WSWNENEN}) & = \lambda_{16,0,-116,0,-1}^{(0.09279 i)} &
A(\mbox{WSESWSES}) & = 0 \displaybreak[1]\\
A(\mbox{WSESWSEN}) & = \frac{1}{4} \left(\sqrt{5}-3\right) &
A(\mbox{WSESWNWN}) & = 0 \displaybreak[1]\\
A(\mbox{WSESWNES}) & = \frac{1}{4} \left(\sqrt{5}-3\right) &
A(\mbox{WSESWNEN}) & = 0 \displaybreak[1]\\
A(\mbox{WSESESWN}) & = \lambda_{16,0,-116,0,-1}^{(-0.09279 i)} &
A(\mbox{WSESESES}) & = 0 \displaybreak[1]\\
A(\mbox{WSESESEN}) & = \frac{1}{4} \left(3-\sqrt{5}\right) &
A(\mbox{WSESENWN}) & = 0 \displaybreak[1]\\
A(\mbox{WSESENES}) & = \lambda_{16,0,-116,0,-1}^{(-0.09279 i)} &
A(\mbox{WSESENEN}) & = 0 \displaybreak[1]\\
A(\mbox{WSENWSEN}) & = 0 &
A(\mbox{WSENWNWN}) & = \frac{1}{4} \left(3-\sqrt{5}\right) \displaybreak[1]\\
A(\mbox{WSENWNES}) & = 0 &
A(\mbox{WSENWNEN}) & = \frac{1}{4} \left(\sqrt{5}-3\right) \displaybreak[1]\\
A(\mbox{WSENESWN}) & = \lambda_{1,0,-11,0,-1}^{(0.3003 i)} &
A(\mbox{WSENESES}) & = \lambda_{16,0,-116,0,-1}^{(0.09279 i)} \displaybreak[1]\\
A(\mbox{WSENESEN}) & = 2-\sqrt{5} &
A(\mbox{WSENENWN}) & = \frac{1}{4} \left(\sqrt{5}-3\right) \displaybreak[1]\\
A(\mbox{WSENENES}) & = 0 &
A(\mbox{WSENENEN}) & = \frac{1}{4} \left(3-\sqrt{5}\right) \displaybreak[1]\\
A(\mbox{WNWNWNWN}) & = 0 &
A(\mbox{WNWNWNES}) & = \frac{1}{4} \left(3-\sqrt{5}\right) \displaybreak[1]\\
A(\mbox{WNWNWNEN}) & = 0 &
A(\mbox{WNWNESES}) & = 0 \displaybreak[1]\\
A(\mbox{WNWNESEN}) & = \lambda_{16,0,-116,0,-1}^{(0.09279 i)} &
A(\mbox{WNWNENES}) & = \frac{1}{4} \left(\sqrt{5}-3\right) \displaybreak[1]\\
A(\mbox{WNWNENEN}) & = 0 &
A(\mbox{WNESWNES}) & = 0 \displaybreak[1]\\
A(\mbox{WNESWNEN}) & = \frac{1}{4} \left(\sqrt{5}-3\right) &
A(\mbox{WNESESES}) & = \frac{1}{4} \left(3-\sqrt{5}\right) \displaybreak[1]\\
A(\mbox{WNESESEN}) & = 0 &
A(\mbox{WNESENES}) & = 2-\sqrt{5} \displaybreak[1]\\
A(\mbox{WNESENEN}) & = \lambda_{16,0,-116,0,-1}^{(0.09279 i)} &
A(\mbox{WNENWNEN}) & = 0 \displaybreak[1]\\
A(\mbox{WNENESES}) & = 0 &
A(\mbox{WNENESEN}) & = \lambda_{16,0,-116,0,-1}^{(-0.09279 i)} \displaybreak[1]\\
A(\mbox{WNENENES}) & = \frac{1}{4} \left(3-\sqrt{5}\right) &
A(\mbox{WNENENEN}) & = 0 \displaybreak[1]\\
A(\mbox{ESESESES}) & = 0 &
A(\mbox{ESESESEN}) & = \lambda_{16,0,-116,0,-1}^{(-0.09279 i)} \displaybreak[1]\\
A(\mbox{ESESENEN}) & = 0 &
A(\mbox{ESENESEN}) & = 0 \displaybreak[1]\\
A(\mbox{ESENENEN}) & = \lambda_{16,0,-116,0,-1}^{(0.09279 i)} &
A(\mbox{ENENENEN}) & = 0
\end{align*}\begin{align*}
B(\mbox{WSWSWSWS}) & = \frac{1}{2} \left(5 \sqrt{5}-11\right) &
B(\mbox{WSWSWSWN}) & = \frac{1}{4} \left(\sqrt{5}-3\right) \displaybreak[1]\\
B(\mbox{WSWSWSES}) & = \frac{1}{2} \left(11-5 \sqrt{5}\right) &
B(\mbox{WSWSWSEN}) & = \lambda_{16,0,-1044,0,-81}^{(0.2784 i)} \displaybreak[1]\\
B(\mbox{WSWSWNWN}) & = \frac{1}{2} \left(7-3 \sqrt{5}\right) &
B(\mbox{WSWSWNES}) & = \lambda_{16,0,-1044,0,-81}^{(-0.2784 i)} \displaybreak[1]\\
B(\mbox{WSWSWNEN}) & = \frac{1}{2} \left(3 \sqrt{5}-7\right) &
B(\mbox{WSWSESWN}) & = \frac{1}{4} \left(3-\sqrt{5}\right) \displaybreak[1]\\
B(\mbox{WSWSESES}) & = \frac{1}{2} \left(5 \sqrt{5}-11\right) &
B(\mbox{WSWSESEN}) & = \lambda_{16,0,-1044,0,-81}^{(-0.2784 i)} \displaybreak[1]\\
B(\mbox{WSWSENWN}) & = 0 &
B(\mbox{WSWSENES}) & = \frac{1}{4} \left(\sqrt{5}-3\right) \displaybreak[1]\\
B(\mbox{WSWSENEN}) & = 0 &
B(\mbox{WSWNWSWN}) & = \frac{1}{2} \left(3-\sqrt{5}\right) \displaybreak[1]\\
B(\mbox{WSWNWSES}) & = \frac{1}{4} \left(3-\sqrt{5}\right) &
B(\mbox{WSWNWSEN}) & = 0 \displaybreak[1]\\
B(\mbox{WSWNWNWN}) & = \frac{1}{4} \left(\sqrt{5}-3\right) &
B(\mbox{WSWNWNES}) & = 0 \displaybreak[1]\\
B(\mbox{WSWNWNEN}) & = \frac{1}{4} \left(3-\sqrt{5}\right) &
B(\mbox{WSWNESWN}) & = 0 \displaybreak[1]\\
B(\mbox{WSWNESES}) & = \lambda_{16,0,-1044,0,-81}^{(0.2784 i)} &
B(\mbox{WSWNESEN}) & = 2-\sqrt{5} \displaybreak[1]\\
B(\mbox{WSWNENWN}) & = \frac{1}{4} \left(3-\sqrt{5}\right) &
B(\mbox{WSWNENES}) & = 0 \displaybreak[1]\\
B(\mbox{WSWNENEN}) & = \frac{1}{4} \left(\sqrt{5}-3\right) &
B(\mbox{WSESWSES}) & = \frac{1}{2} \left(5 \sqrt{5}-11\right) \displaybreak[1]\\
B(\mbox{WSESWSEN}) & = \lambda_{16,0,-1044,0,-81}^{(-0.2784 i)} &
B(\mbox{WSESWNWN}) & = \frac{1}{2} \left(3 \sqrt{5}-7\right) \displaybreak[1]\\
B(\mbox{WSESWNES}) & = \lambda_{16,0,-1044,0,-81}^{(0.2784 i)} &
B(\mbox{WSESWNEN}) & = \frac{1}{2} \left(7-3 \sqrt{5}\right) \displaybreak[1]\\
B(\mbox{WSESESWN}) & = \frac{1}{4} \left(\sqrt{5}-3\right) &
B(\mbox{WSESESES}) & = \frac{1}{2} \left(11-5 \sqrt{5}\right) \displaybreak[1]\\
B(\mbox{WSESESEN}) & = \lambda_{16,0,-1044,0,-81}^{(0.2784 i)} &
B(\mbox{WSESENWN}) & = 0 \displaybreak[1]\\
B(\mbox{WSESENES}) & = \frac{1}{4} \left(3-\sqrt{5}\right) &
B(\mbox{WSESENEN}) & = 0 \displaybreak[1]\\
B(\mbox{WSENWSEN}) & = \frac{1}{2} \left(3 \sqrt{5}-9\right) &
B(\mbox{WSENWNWN}) & = \lambda_{16,0,-1044,0,-81}^{(-0.2784 i)} \displaybreak[1]\\
B(\mbox{WSENWNES}) & = \frac{1}{2} \left(7-3 \sqrt{5}\right) &
B(\mbox{WSENWNEN}) & = \lambda_{16,0,-1044,0,-81}^{(0.2784 i)} \displaybreak[1]\\
B(\mbox{WSENESWN}) & = 2-\sqrt{5} &
B(\mbox{WSENESES}) & = \frac{1}{4} \left(3-\sqrt{5}\right) \displaybreak[1]\\
B(\mbox{WSENESEN}) & = 0 &
B(\mbox{WSENENWN}) & = \lambda_{16,0,-1044,0,-81}^{(0.2784 i)} \displaybreak[1]\\
B(\mbox{WSENENES}) & = \frac{1}{2} \left(3 \sqrt{5}-7\right) &
B(\mbox{WSENENEN}) & = \lambda_{16,0,-1044,0,-81}^{(-0.2784 i)} \displaybreak[1]\\
B(\mbox{WNWNWNWN}) & = \frac{1}{2} \left(5 \sqrt{5}-11\right) &
B(\mbox{WNWNWNES}) & = \lambda_{16,0,-1044,0,-81}^{(0.2784 i)} \displaybreak[1]\\
B(\mbox{WNWNWNEN}) & = \frac{1}{2} \left(11-5 \sqrt{5}\right) &
B(\mbox{WNWNESES}) & = 0 \displaybreak[1]\\
B(\mbox{WNWNESEN}) & = \frac{1}{4} \left(\sqrt{5}-3\right) &
B(\mbox{WNWNENES}) & = \lambda_{16,0,-1044,0,-81}^{(-0.2784 i)} \displaybreak[1]\\
B(\mbox{WNWNENEN}) & = \frac{1}{2} \left(5 \sqrt{5}-11\right) &
B(\mbox{WNESWNES}) & = \frac{1}{2} \left(3 \sqrt{5}-9\right) \displaybreak[1]\\
B(\mbox{WNESWNEN}) & = \lambda_{16,0,-1044,0,-81}^{(-0.2784 i)} &
B(\mbox{WNESESES}) & = \lambda_{16,0,-1044,0,-81}^{(-0.2784 i)} \displaybreak[1]\\
B(\mbox{WNESESEN}) & = \frac{1}{2} \left(3 \sqrt{5}-7\right) &
B(\mbox{WNESENES}) & = 0 \displaybreak[1]\\
B(\mbox{WNESENEN}) & = \frac{1}{4} \left(3-\sqrt{5}\right) &
B(\mbox{WNENWNEN}) & = \frac{1}{2} \left(5 \sqrt{5}-11\right) \displaybreak[1]\\
B(\mbox{WNENESES}) & = 0 &
B(\mbox{WNENESEN}) & = \frac{1}{4} \left(3-\sqrt{5}\right) \displaybreak[1]\\
B(\mbox{WNENENES}) & = \lambda_{16,0,-1044,0,-81}^{(0.2784 i)} &
B(\mbox{WNENENEN}) & = \frac{1}{2} \left(11-5 \sqrt{5}\right) \displaybreak[1]\\
B(\mbox{ESESESES}) & = \frac{1}{2} \left(5 \sqrt{5}-11\right) &
B(\mbox{ESESESEN}) & = \frac{1}{4} \left(\sqrt{5}-3\right) \displaybreak[1]\\
B(\mbox{ESESENEN}) & = \frac{1}{2} \left(7-3 \sqrt{5}\right) &
B(\mbox{ESENESEN}) & = \frac{1}{2} \left(3-\sqrt{5}\right) \displaybreak[1]\\
B(\mbox{ESENENEN}) & = \frac{1}{4} \left(\sqrt{5}-3\right) &
B(\mbox{ENENENEN}) & = \frac{1}{2} \left(5 \sqrt{5}-11\right)
\end{align*}

These entries lie in $\Q(\mu_{\Z/4})$, where $\mu_{\Z/4}$ is the root of
$$
x^8-38 x^6+100 x^5+343 x^4-2300 x^3+5102 x^2-5500 x+2581
$$
which is approximately $2.236+0.700 i$.

\section{Moments and tetrahedral constants}\label{sec:MomentsAndTetrahedrals}

For all of our planar algebras, our generators are self adjoint. In the following subsections, we list the necessary moments and tetrahedral structure constants needed for our calculations.

\subsection{$3^{\Z/3}$: Haagerup}\label{moments:333Haagerup}

\begin{align*}
\Tr(A A) &= 3+\sqrt{13} & \Tr(\check{A} \check{A}) &= 3+\sqrt{13}\displaybreak[1]\\
\Tr(A A A) &= 0 & \Tr(\check{A} \check{A} \check{A}) &= \lambda_{9,0,-144,0,-256}^{(4.20)}
\end{align*}

\begin{align*}
\Delta(A,A,A|A) &= \sqrt{\frac{22}{9}+\frac{2 \sqrt{13}}{9}}
\end{align*}

\subsection{$3^{\Z/2\times \Z/2}$}\label{moments:3333}

\begin{align*}
\Tr(A A) &= 4+2 \sqrt{5} & \Tr(\check{A} \check{A}) &= 4+2 \sqrt{5}\displaybreak[1]\\
\Tr(A B) &= 0 & \Tr(\check{A} \check{B}) &= 0\displaybreak[1]\\
\Tr(B B) &= 12+6 \sqrt{5} & \Tr(\check{B} \check{B}) &= 12+6 \sqrt{5}\displaybreak[1]\\
\Tr(A A A) &= 0 & \Tr(\check{A} \check{A} \check{A}) &= \lambda_{64,0,-1296,0,81}^{(4.493)}\displaybreak[1]\\
\Tr(A A B) &= -4-2 \sqrt{5} & \Tr(\check{A} \check{A} \check{B}) &= \lambda_{64,0,-720,0,25}^{(-3.349)}\displaybreak[1]\\
\Tr(A B B) &= 0 & \Tr(\check{A} \check{B} \check{B}) &= \lambda_{64,0,-11664,0,6561}^{(-13.5)}\displaybreak[1]\\
\Tr(B B B) &= 12+6 \sqrt{5} & \Tr(\check{B} \check{B} \check{B}) &= \lambda_{64,0,-6480,0,2025}^{(10.0)}
\end{align*}

\begin{align*}
\Delta(A,A,A|A) &= \frac{\sqrt{3-\sqrt{5}}}{4}\displaybreak[1]\\
\Delta(A,A,A|B) &= -\sqrt{\frac{63}{16}+\frac{27 \sqrt{5}}{16}}\displaybreak[1]\\
\Delta(A,A,B|A) &= \sqrt{\frac{63}{16}+\frac{27 \sqrt{5}}{16}}\displaybreak[1]\\
\Delta(A,A,B|B) &= -\sqrt{\frac{467}{16}+\frac{207 \sqrt{5}}{16}}\displaybreak[1]\\
\Delta(A,B,A|B) &= \sqrt{\frac{2027}{16}+\frac{903 \sqrt{5}}{16}}\displaybreak[1]\\
\Delta(A,B,B|B) &= -\sqrt{\frac{567}{16}+\frac{243 \sqrt{5}}{16}}\displaybreak[1]\\
\Delta(B,A,B|A) &= \sqrt{\frac{2027}{16}+\frac{903 \sqrt{5}}{16}}\displaybreak[1]\\
\Delta(B,A,B|B) &= \sqrt{\frac{567}{16}+\frac{243 \sqrt{5}}{16}}\displaybreak[1]\\
\Delta(B,B,B|B) &= \frac{9 \sqrt{3-\sqrt{5}}}{4}
\end{align*}

\subsection{$3^{\Z/4}$}\label{moments:3333ZMod4}

\begin{align*}
\Tr(A A) &= 4+2 \sqrt{5} & \Tr(\check{A} \check{A}) &= 4+2 \sqrt{5}\displaybreak[1]\\
\Tr(A B) &= 0 & \Tr(\check{A} \check{B}) &= 0\displaybreak[1]\\
\Tr(B B) &= 12+6 \sqrt{5} & \Tr(\check{B} \check{B}) &= 12+6 \sqrt{5}\displaybreak[1]\\
\Tr(A A A) &= 0 & \Tr(\check{A} \check{A} \check{A}) &= \lambda_{4,0,-40,0,-25}^{(-3.25)}\displaybreak[1]\\
\Tr(A A B) &= -4-2 \sqrt{5} & \Tr(\check{A} \check{A} \check{B}) &= \lambda_{4,0,-180,0,25}^{(6.698)}\displaybreak[1]\\
\Tr(A B B) &= 0 & \Tr(\check{A} \check{B} \check{B}) &= \lambda_{4,0,-648,0,-6561}^{(13.1)}\displaybreak[1]\\
\Tr(B B B) &= 12+6 \sqrt{5} & \Tr(\check{B} \check{B} \check{B}) &= \lambda_{4,0,-324,0,81}^{(0.501)}
\end{align*}

\begin{align*}
\Delta(A,A,A|A) &= -\sqrt{3+\sqrt{5}}\displaybreak[1]\\
\Delta(A,A,A|B) &= 0\displaybreak[1]\\
\Delta(A,A,B|A) &= -i \sqrt{11+5 \sqrt{5}}\displaybreak[1]\\
\Delta(A,A,B|B) &= -\sqrt{2}\displaybreak[1]\\
\Delta(A,B,A|B) &= \sqrt{107+39 \sqrt{5}}\displaybreak[1]\\
\Delta(A,B,B|B) &= -9 i \sqrt{1+\sqrt{5}}\displaybreak[1]\\
\Delta(B,A,B|A) &= \sqrt{47+21 \sqrt{5}}\displaybreak[1]\\
\Delta(B,A,B|B) &= 0\displaybreak[1]\\
\Delta(B,B,B|B) &= 9 \sqrt{3-\sqrt{5}}
\end{align*}

\bibliographystyle{amsalpha}
\bibliography{../../bibliography/bibliography}
\end{document}